\documentclass[10pt]{article}
\usepackage{epsf}
\usepackage{amsmath}

\allowdisplaybreaks

\usepackage[showframe=false]{geometry}
\usepackage{changepage}

\usepackage{epsfig}
\usepackage{amssymb}

\usepackage{amsthm}
\usepackage{setspace}
\usepackage{cite}
\usepackage{mcite}

\usepackage{algorithmic}  
\usepackage{algorithm}

\usepackage{shadow}
\usepackage{fancybox}
\usepackage{fancyhdr}

\usepackage{color}
\usepackage[usenames,dvipsnames,svgnames,table]{xcolor}
\newcommand{\bl}[1]{\textcolor{blue}{#1}}

\definecolor{mypurple}{rgb}{.4,.0,.5}
\newcommand{\prp}[1]{\textcolor{mypurple}{#1}}

\usepackage[hyphens]{url}

\usepackage[colorlinks=true,
            linkcolor=black,
            urlcolor=blue,
            citecolor=purple]{hyperref}

\usepackage{breakurl}

\def\y{{\bf y}}

\def\x{{\bf x}}

\def\x{{\mathbf x}}

\def\u{{\bf u}}

\def\x{{\bf x}}
\def\y{{\bf y}}

\def\h{{\bf h}}

\def\be{\begin{equation}}
\def\ee{\end{equation}}
\def\ba{\left[\begin{array}}
\def\ea{\end{array}\right]}

\def\u{{\bf u}}

\def\x{{\bf x}}
\def\y{{\bf y}}

\def\1{{\bf 1}}

\def\0{{\bf 0}}

\def\calX{{\cal X}}

\def\mR{{\mathbb R}}

\def\mE{{\mathbb E}}

\def\lp{\left (}
\def\rp{\right )}

\newtheorem{theorem}{Theorem}
\newtheorem{corollary}{Corollary}

\begin{document}

\begin{singlespace}

\title {Generic and lifted probabilistic comparisons -- max replaces minmax 
}
\author{
\textsc{Mihailo Stojnic
\footnote{e-mail: {\tt flatoyer@gmail.com}} }}
\date{}
\maketitle

\centerline{{\bf Abstract}} \vspace*{0.1in}

In this paper we introduce a collection of powerful statistical comparison results. We first present the results that we obtained while developing a general comparison concept. After that we introduce a separate lifting procedure that is a comparison concept on its own. We then show how in certain scenarios the lifting procedure basically represents a substantial upgrade over the general strategy. We complement the introduced results with a fairly large collection of numerical experiments that are in an overwhelming agreement with what the theory predicts. We also show how many well known comparison results (e.g. Slepian's max and Gordon's minmax principle) can be obtained as special cases. Moreover, it turns out that the minmax principle can be viewed as a single max principle as well. The range of applications is enormous. It starts with revisiting many of the results we created in recent years in various mathematical fields and recognizing that they are fully self-contained as their starting blocks are specialized variants of the concepts introduced here. Further upgrades relate to core comparison extensions on the one side and more practically oriented modifications on the other. Those that we deem the most important we discuss in several separate companion papers to ensure preserving the introductory elegance and simplicity of what is presented here.

\vspace*{0.25in} \noindent {\bf Index Terms: Random processes; comparison principles, lifting}.

\end{singlespace}

\section{Introduction}
\label{sec:back}

In this paper we study comparisons of random processes. This topic has a fairly rich history and is among the most successful tools in modern probability and statistics. The range of applications is also quite large and our own interest started through such an application. Namely, while studying many random optimization problems we characterized their performance by developing a collection of novel probabilistic mechanisms that used as starting blocks some forms of random processes comparisons (see, e.g. \cite{StojnicISIT2010binary,StojnicCSetam09,StojnicUpper10,StojnicCSetamBlock09,StojnicICASSP10knownsupp} and references therein). After initially recognizing the importance of the role that the strong random comparisons can play, we eventually went much further and designed much stronger concepts through which we were able to achieve substantial improvements in studying many problems previously essentially untouchable by any other tool (more on this can be found in, e.g. \cite{StojnicLiftStrSec13,StojnicMoreSophHopBnds10,StojnicRicBnds13} and references therein).

As mentioned above, many great results have been obtained in the theory of comparisons of random processes. In our view, two the most typical and influential are the Slepian's max \cite{Slep62} and the Gordon's minmax \cite{Gordon85} principle (see also \cite{Sudakov71,Fernique74,Fernique75,Kahane86}). What makes these results particularly unique is that they are fairly self-contained and are basically derived without relying on much of other knowledge. In fact, the line of work before the Slepian's is relatively short and for most parts goes through e.g. \cite{Schlafli858,Placket54,Chover61}. Gordon's work, on the other hand, is pretty much the first upgrade after Slepian in the direction that it pursues. Of course many applications and extensions of these principles are known and we stop short of reviewing these here as our topic of study won't necessarily go into these directions (for more extensive studies in these directions we instead refer to classical works, e.g. \cite{Adler90,Lifshits85,LedTal91,Tal05}). Namely, Slepian's and Gordon's principles by definition in its core deal with the extrema of the random processes. Here, we focus on some fundamentals of comparisons that don't necessarily deal with the extrema but rather with general functions of random processes (as it will turn out, the special cases of the results that we develop can be used to deal with the random processes extrema as well and in later sections we will discuss how it can be done).

We will split the presentation into two main parts. In the first part we will introduce a general comparison principle. This principle will turn out to be a very powerful tool and, as we will show, it will contain as special cases both, the Slepian's max and the Gordon's minmax principle. As such it can then also serve as a replacing starting block in many of our earlier works where the max and minmax principles were used (see, e.g. \cite{StojnicISIT2010binary,StojnicCSetam09,StojnicUpper10,StojnicCSetamBlock09,StojnicICASSP10knownsupp}) which in turn makes all these results basically fully self-contained. On the other hand, in the second part we will introduce a lifting procedure which is a comparison principle on its own as well. However, in certain scenarios it will turn out that the two principles can be connected to each other with the lifting one often acting as a substantial upgrade. Moreover, it will turn out that the special cases of the lifting principle can also be used as starting blocks in many of our other earlier works (see, e.g. \cite{StojnicLiftStrSec13,StojnicMoreSophHopBnds10,StojnicRicBnds13}) making those fully self-contained as well.

\section{General concepts}
\label{sec:gencon}

For a given set $\calX=\{\x^{(1)},\x^{(2)},\dots,\x^{(l)}\}$, where $\x^{(i)}\in S^{n-1},1\leq i\leq l$, and $S^{n-1}$ is the unit Euclidean sphere in $\mR^n$, we will be interested in the following function
\begin{eqnarray}\label{eq:genanal1}
 f(G,\calX,\beta,s)= \frac{1}{\beta\sqrt{n}} \log\lp \sum_{i=1}^{l}e^{\beta\lp s\|
 G\x^{(i)}\|_2\rp} \rp,
\end{eqnarray}
where $s\in\{-1,1\}$ and $\beta>0$ is a real parameter. A typical behavior of this function we will study through a Gaussian framework, i.e. we will assume that $G\in \mR^{m\times n}$ is an $(m\times n)$ dimensional matrix with i.i.d. standard normal components. In such a context (and especially so if $m$, $n$, and $l$ are large) the expected value of the above function is usually its most relevant statistical characteristic. We will denote this expected value by $\xi(\calX,\beta,s)$ and set
\begin{eqnarray}\label{eq:genanal2}
\xi(\calX,\beta,s)\triangleq \mE_G f(G,\calX,\beta,s)= \mE_G\frac{1}{\beta\sqrt{n}} \log\lp \sum_{i=1}^{l}e^{\beta\lp s\|
 G\x^{(i)}\|_2\rp} \rp.
\end{eqnarray}
To study $\xi(\calX,\beta,s)$ we will employ the following interpolating function $\psi(\cdot)$
\begin{eqnarray}\label{eq:genanal3}
\psi(\calX,\beta,s,t)= \mE_{G,\u^{(2)},\h} \frac{1}{\beta\sqrt{n}} \log\lp \sum_{i=1}^{l}e^{\beta\lp s\|\sqrt{t}
 G\x^{(i)}+\sqrt{1-t}\u^{(2)}\|_2+\sqrt{1-t}\h^T\x^{(i)}\rp} \rp,
\end{eqnarray}
where $\u^{(2)}$ and $\h$ are $m$ and $n$ dimensional vectors of i.i.d standard normals, respectively ($G$, $\u^{(2)}$, and $\h$ are all assumed to be independent of each other; $\mE$ denotes the expectation; although it will typically be clear from the context with respect to what randomness the expectations are being computed, we will still often in the subscript emphasize the underlying randomness). It is not that hard to see that $\xi(\calX,\beta,s)=\psi(\calX,\beta,s,1)$. Also, $\psi(\calX,\beta,s,0)$ is an object typically much easier to handle than $\psi(\calX,\beta,s,1)$. Below, we will try to connect $\psi(\calX,\beta,s,1)$ to $\psi(\calX,\beta,s,0)$ which will then automatically connect $\xi(\calX,\beta,s)$ to $\psi(\calX,\beta,s,0)$. To facilitate writing we also set
\begin{eqnarray}\label{eq:genanal4}
\u^{(i,1)} & =  & G\x^{(i)} \nonumber \\
\u^{(i,3)} & =  & \h^T\x^{(i)}.
\end{eqnarray}
Clearly, we then have
\begin{eqnarray}\label{eq:genanal5}
\u_j^{(i,1)} & =  & G_{j,1:n}\x^{(i)},1\leq j\leq m,
\end{eqnarray}
where $\u_j^{(i,1)}$ is the $j$-th component of $\u^{(i,1)}$ and $G_{j,1:n}$ is simply the $j$-th row of $G$. It is also rather obvious that for any fixed $i$, the elements of $\u^{(i,1)}$, $\u^{(2)}$, and $\u^{(i,3)}$ are i.i.d. standard normals. We can then rewrite (\ref{eq:genanal3}) in the following way
\begin{eqnarray}\label{eq:genanal6}
\psi(\calX,\beta,s,t) & = & \mE_{G,\u^{(2)},\h} \frac{1}{\beta\sqrt{n}} \log\lp \sum_{i=1}^{l}e^{\beta\lp s\|\sqrt{t}
 G\x^{(i)}+\sqrt{1-t}\u^{(2)}\|_2+\sqrt{1-t}\h^T\x^{(i)}\rp} \rp \nonumber \\
& = & \mE_{G,\u^{(2)},\h} \frac{1}{\beta\sqrt{n}} \log\lp \sum_{i=1}^{l}e^{\beta\lp s\|\sqrt{t}
 \u^{(i,1)}+\sqrt{1-t}\u^{(2)}\|_2+\sqrt{1-t}\u^{(i,3)}\rp} \rp \nonumber \\
& = & \mE_{\u^{(i,1)},\u^{(2)},\u^{(i,3)}} \frac{1}{\beta\sqrt{n}} \log\lp \sum_{i=1}^{l}e^{\beta\lp s\sqrt{\sum_{j=1}^{m}\lp\sqrt{t}\u_j^{(i,1)}+\sqrt{1-t}\u_j^{(2)}\rp^2}+\sqrt{1-t}\u^{(i,3)}\rp} \rp.
\end{eqnarray}
In what follows we will also often find helpful in facilitating writing the following
\begin{eqnarray}\label{eq:genanal7}
B^{(i)} & \triangleq &  \sqrt{\sum_{j=1}^{m}\lp\sqrt{t}\u_j^{(i,1)}+\sqrt{1-t}\u_j^{(2)}\rp^2} \nonumber \\
A^{(i)} & \triangleq &  e^{\beta(sB^{(i)}+\sqrt{1-t}\u^{(i,3)})}\nonumber \\
Z & \triangleq & \sum_{i=1}^{l}e^{\beta\lp s\sqrt{\sum_{j=1}^{m}\lp\sqrt{t}\u_j^{(i,1)}+\sqrt{1-t}\u_j^{(2)}\rp^2}+\sqrt{1-t}\u^{(i,3)}\rp}= \sum_{i=1}^{l} A^{(i)}.
\end{eqnarray}
From (\ref{eq:genanal6}) and (\ref{eq:genanal7}) we obviously have
\begin{eqnarray}\label{eq:genanal8}
\psi(\calX,\beta,s,t) & = &  \mE_{\u^{(i,1)},\u^{(2)},\u^{(i,3)}} \frac{1}{\beta\sqrt{n}} \log(Z).
\end{eqnarray}
Below we show that $\psi(\calX,\beta,s,t)$ is a decreasing function of $t$. To that end we have
\begin{eqnarray}\label{eq:genanal9}
\frac{d\psi(\calX,\beta,s,t)}{dt} & = &  \mE_{\u^{(i,1)},\u^{(2)},\u^{(i,3)}} \frac{1}{\beta\sqrt{n}} \log\lp \sum_{i=1}^{l}e^{\beta\lp s\sqrt{\sum_{j=1}^{m}\lp\sqrt{t}\u_j^{(i,1)}+\sqrt{1-t}\u_j^{(2)}\rp^2}+\sqrt{1-t}\u^{(i,3)}\rp} \rp\nonumber \\
& = &  \mE_{\u^{(i,1)},\u^{(2)},\u^{(i,3)}} \frac{1}{Z\beta\sqrt{n}} \frac{d\lp \sum_{i=1}^{l}e^{\beta\lp s\sqrt{\sum_{j=1}^{m}\lp\sqrt{t}\u_j^{(i,1)}+\sqrt{1-t}\u_j^{(2)}\rp^2}+\sqrt{1-t}\u^{(i,3)}\rp} \rp }{dt}\nonumber \\
& = &  \mE_{\u^{(i,1)},\u^{(2)},\u^{(i,3)}} \frac{1}{Z\sqrt{n}}  \sum_{i=1}^{l}A^{(i)}\lp s\frac{dB^{(i)}}{dt}-\frac{\u^{(i,3)}}{2\sqrt{1-t}}\rp.\nonumber \\
\end{eqnarray}
From (\ref{eq:genanal7}) we also find
\begin{equation}\label{eq:genanal10}
\frac{dB^{(i)}}{dt} =   \frac{d\sqrt{\sum_{j=1}^{m}\lp\sqrt{t}\u_j^{(i,1)}+\sqrt{1-t}\u_j^{(2)}\rp^2}}{dt}=   \frac{\sum_{j=1}^{m}\lp(\u_j^{(i,1)})^2-(\u_j^{(2)})^2+\u_j^{(i,1)}\u_j^{(2)}\lp\frac{\sqrt{1-t}}{\sqrt{t}}-\frac{\sqrt{t}}{\sqrt{1-t}}\rp \rp }{2B^{(i)}}.
\end{equation}
A combination of (\ref{eq:genanal9}) and (\ref{eq:genanal10}) gives
\begin{eqnarray}\label{eq:genanal11}
\frac{d\psi(\calX,\beta,s,t)}{dt}
& = &  \mE_{\u^{(i,1)},\u^{(2)},\u^{(i,3)}} \frac{s}{2\sqrt{n}}\sum_{j=1}^{m} \sum_{i=1}^{l} \frac{ A^{(i)}\lp(\u_j^{(i,1)})^2-(\u_j^{(2)})^2+\u_j^{(i,1)}\u_j^{(2)}\lp\frac{\sqrt{1-t}}{\sqrt{t}}-\frac{\sqrt{t}}{\sqrt{1-t}}\rp \rp }{ZB^{(i)}}\nonumber \\
& & - \mE_{\u^{(i,1)},\u^{(2)},\u^{(i,3)}} \frac{1}{2\sqrt{n}}  \sum_{i=1}^{l}
\frac{A^{(i)}\u^{(i,3)}}{Z\sqrt{1-t}}.\nonumber \\
\end{eqnarray}
We will handle separately all the terms appearing in the above sums. These include $\mE_{\u^{(i,1)},\u^{(2)},\u^{(i,3)}}  \frac{ A^{(i)}(\u_j^{(i,1)})^2}{ZB^{(i)}}$, $\mE_{\u^{(i,1)},\u^{(2)},\u^{(i,3)}}  \frac{ A^{(i)}(\u_j^{(2)})^2}{ZB^{(i)}}$, $\mE_{\u^{(i,1)},\u^{(2)},\u^{(i,3)}}  \frac{ A^{(i)}\u_j^{(i,1)}\u_j^{(2)}}{ZB^{(i)}}$, and $\mE_{\u^{(i,1)},\u^{(2)},\u^{(i,3)}}  \frac{ A^{(i)}\u^{(i,3)}}{Z}$. We will start with computing $\mE_{\u^{(i,1)},\u^{(2)},\u^{(i,3)}}  \frac{ A^{(i)}\u_j^{(i,1)}\u_j^{(2)}}{ZB^{(i)}}$.

\subsection{Handling $\mE_{\u^{(i,1)},\u^{(2)},\u^{(i,3)}}  \frac{ A^{(i)}\u_j^{(i,1)}\u_j^{(2)}}{ZB^{(i)}}$}
\label{sec:hand1}

Below we study $\mE_{\u^{(i,1)},\u^{(2)},\u^{(i,3)}}  \frac{ A^{(i)}\u_j^{(i,1)}\u_j^{(2)}}{ZB^{(i)}}$. We will present two strategies how to handle $\mE_{\u^{(i,1)},\u^{(2)},\u^{(i,3)}}  \frac{ A^{(i)}\u_j^{(i,1)}\u_j^{(2)}}{ZB^{(i)}}$.

\textbf{\underline{\emph{1) Fixing $\u^{(i,1)}$}}}

We start with the following observation
\begin{equation}\label{eq:genanal12}
\mE_{\u^{(i,1)},\u^{(2)},\u^{(i,3)}}  \frac{ A^{(i)}\u_j^{(i,1)}\u_j^{(2)}}{ZB^{(i)}} =
\mE\lp\sum_{p=1,p\neq i}^{l} \mE (\u_j^{(i,1)}\u_j^{(p,1)})\frac{d}{d\u_j^{(p,1)}}\lp \frac{ A^{(i)}\u_j^{(2)}}{ZB^{(i)}}\rp
+\mE (\u_j^{(i,1)}\u_j^{(i,1)})\frac{d}{d\u_j^{(i,1)}}\lp \frac{ A^{(i)}\u_j^{(2)}}{ZB^{(i)}}\rp\rp,
\end{equation}
where we utilized Gaussian integration by parts. Since $\mE (\u_j^{(i,1)}\u_j^{(p,1)})=(\x^{(i)})^T\x^{(p)}$ we easily also have
\begin{equation}\label{eq:genanal13}
\mE_{\u^{(i,1)},\u^{(2)},\u^{(i,3)}}  \frac{ A^{(i)}\u_j^{(i,1)}\u_j^{(2)}}{ZB^{(i)}} =\mE\lp
\sum_{p=1,p\neq i}^{l} (\x^{(i)})^T\x^{(p)}\frac{d}{d\u_j^{(p,1)}}\lp \frac{ A^{(i)}\u_j^{(2)}}{ZB^{(i)}}\rp
+(\x^{(i)})^T\x^{(i)}\frac{d}{d\u_j^{(i,1)}}\lp \frac{ A^{(i)}\u_j^{(2)}}{ZB^{(i)}}\rp\rp.
\end{equation}
For $p\neq i$ we further have
\begin{eqnarray}\label{eq:genanal14}
\frac{d}{d\u_j^{(p,1)}}\lp \frac{ A^{(i)}\u_j^{(2)}}{ZB^{(i)}}\rp=\frac{ A^{(i)}\u_j^{(2)}}{B^{(i)}}\frac{d}{d\u_j^{(p,1)}}\lp \frac{1}{Z}\rp
=-\frac{ A^{(i)}\u_j^{(2)}}{Z^2B^{(i)}}\frac{dZ}{d\u_j^{(p,1)}}=-\frac{ A^{(i)}\u_j^{(2)}}{Z^2B^{(i)}}\frac{d \sum_{i=1}^{l }A^{(i)}}{d\u_j^{(p,1)}}.
\end{eqnarray}
From (\ref{eq:genanal7}) we find
\begin{equation}\label{eq:genanal15}
\frac{d B^{(p)}}{d\u_j^{(p,1)}} =  \frac{2\u_j^{(p,1)}t+2\sqrt{t}\sqrt{1-t}\u_j^{(2)}}{2B^{(p)}}
= \frac{\u_j^{(p,1)}\sqrt{t}+\sqrt{1-t}\u_j^{(2)}}{B^{(p)}}\sqrt{t},
\end{equation}
and
\begin{equation}\label{eq:genanal16}
\frac{d A^{(p)}}{d\u_j^{(p,1)}} = \beta s A^{(p)} \frac{d B^{(p)}}{d\u_j^{(p,1)}}= \beta s A^{(p)} \frac{\u_j^{(p,1)}\sqrt{t}+\sqrt{1-t}\u_j^{(2)}}{B^{(p)}}\sqrt{t}.
\end{equation}
Combining (\ref{eq:genanal14}), (\ref{eq:genanal15}), and (\ref{eq:genanal16}) we obtain
\begin{equation}\label{eq:genanal17}
\frac{d}{d\u_j^{(p,1)}}\lp \frac{ A^{(i)}\u_j^{(2)}}{ZB^{(i)}}\rp=-\frac{ A^{(i)}\u_j^{(2)}}{Z^2B^{(i)}}\frac{d \sum_{i=1}^{l }A^{(i)}}{d\u_j^{(p,1)}}
=-\frac{ A^{(i)}\u_j^{(2)}}{Z^2B^{(i)}}\frac{d A^{(p)}}{d\u_j^{(p,1)}}=-\frac{\beta s  A^{(i)}A^{(p)}\u_j^{(2)}\sqrt{t}(\u_j^{(p,1)}\sqrt{t}+\sqrt{1-t}\u_j^{(2)})}{Z^2B^{(i)}B^{(p)}}.
\end{equation}
For $p=i$ we have
\begin{eqnarray}\label{eq:genanal18}
\frac{d}{d\u_j^{(i,1)}}\lp \frac{ A^{(i)}\u_j^{(2)}}{ZB^{(i)}}\rp & = &\u_j^{(2)}\frac{d}{d\u_j^{(i,1)}}\lp \frac{ A^{(i)}}{ZB^{(i)}}\rp
=\frac{\u_j^{(2)}}{Z}\frac{d}{d\u_j^{(i,1)}}\lp \frac{ A^{(i)}}{B^{(i)}}\rp-\frac{ A^{(i)}\u_j^{(2)}}{Z^2B^{(i)}}\frac{dZ}{d\u_j^{(i,1)}}\nonumber \\
& = & \frac{\u_j^{(2)}}{ZB^{(i)}}\frac{d A^{(i)}}{d\u_j^{(i,1)}}-\frac{\u_j^{(2)}A^{(i)}}{Z(B^{(i)})^2}\frac{d B^{(i)}}{d\u_j^{(i,1)}}
 -\frac{ A^{(i)}\u_j^{(2)}}{Z^2B^{(i)}}\frac{d A^{(i)}}{d\u_j^{(i,1)}} \nonumber \\
 & = & \frac{\beta s \u_j^{(2)}A^{(i)}}{ZB^{(i)}}\frac{d B^{(i)}}{d\u_j^{(i,1)}}-\frac{\u_j^{(2)}A^{(i)}}{Z(B^{(i)})^2}\frac{d B^{(i)}}{d\u_j^{(i,1)}}
 -\frac{ A^{(i)}\u_j^{(2)}}{Z^2B^{(i)}}\frac{d A^{(i)}}{d\u_j^{(i,1)}}\nonumber \\
 & = & \frac{\beta s \u_j^{(2)}A^{(i)}}{ZB^{(i)}}\frac{d B^{(i)}}{d\u_j^{(i,1)}}-\frac{\u_j^{(2)}A^{(i)}}{Z(B^{(i)})^2}\frac{d B^{(i)}}{d\u_j^{(i,1)}}
 -\frac{\beta s  A^{(i)}A^{(i)}\u_j^{(2)}\sqrt{t}(\u_j^{(i,1)}\sqrt{t}+\sqrt{1-t}\u_j^{(2)})}{Z^2B^{(i)}B^{(i)}}.\nonumber \\
\end{eqnarray}
A combination of (\ref{eq:genanal13}), (\ref{eq:genanal17}), and (\ref{eq:genanal18}) gives
\begin{eqnarray}\label{eq:genanal19}
\mE_{\u^{(i,1)},\u^{(2)},\u^{(i,3)}}  \frac{ A^{(i)}\u_j^{(i,1)}\u_j^{(2)}}{ZB^{(i)}} & = &
(\x^{(i)})^T\x^{(i)} \mE\lp\frac{\beta s \u_j^{(2)}A^{(i)}}{ZB^{(i)}}\frac{d B^{(i)}}{d\u_j^{(i,1)}}-\frac{\u_j^{(2)}A^{(i)}}{Z(B^{(i)})^2}\frac{d B^{(i)}}{d\u_j^{(i,1)}}\rp \nonumber \\
& & -\mE\sum_{p=1}^{l} (\x^{(i)})^T\x^{(p)} \frac{\beta s  A^{(i)}A^{(p)}\u_j^{(2)}\sqrt{t}(\u_j^{(p,1)}\sqrt{t}+\sqrt{1-t}\u_j^{(2)})}{Z^2B^{(i)}B^{(p)}}\nonumber \\
& = &
\mE\lp\frac{ A^{(i)}}{ZB^{(i)}}\lp \lp \beta s -\frac{1}{B^{(i)}}\rp\frac{\u_j^{(2)}\sqrt{1-t}+\sqrt{t}\u_j^{(i,1)}}{B^{(i)}}\u^{(2)}\sqrt{t}\rp\rp \nonumber  \\
& & -\mE\sum_{p=1}^{l} (\x^{(i)})^T\x^{(p)} \frac{\beta s  A^{(i)}A^{(p)}\u_j^{(2)}\sqrt{t}(\u_j^{(p,1)}\sqrt{t}+\sqrt{1-t}\u_j^{(2)})}{Z^2B^{(i)}B^{(p)}}.
\end{eqnarray}

\textbf{\underline{\emph{2) Fixing $\u^{(2)}$}}}

To ensure the easiness of presentation we will try to parallel as much as possible what we presented above. We start with the following
\begin{equation}\label{eq:genCanal1}
\mE_{\u^{(i,1)},\u^{(2)},\u^{(i,3)}}  \frac{ A^{(i)}\u_j^{(2)}\u^{(i,1)}}{ZB^{(i)}} =
\mE\lp\mE (\u_j^{(2)}\u_j^{(2)})\frac{d}{d\u_j^{(2)}}\lp \frac{ A^{(i)}\u^{(i,1)}}{ZB^{(i)}}\rp\rp
=\mE\lp\u^{(i,1)}\frac{d}{d\u_j^{(2)}}\lp \frac{ A^{(i)}}{ZB^{(i)}}\rp\rp.
\end{equation}
From (\ref{eq:genanal7}) we have
\begin{equation}\label{eq:genCanal2}
\frac{d B^{(p)}}{d\u_j^{(2)}} =  \frac{2\u_j^{(2)}(1-t)+2\sqrt{t}\sqrt{1-t}\u_j^{(i,1)}}{2B^{(p)}}
= \frac{\u_j^{(2)}\sqrt{1-t}+\sqrt{t}\u_j^{(i,1)}}{B^{(p)}}\sqrt{1-t},
\end{equation}
and
\begin{equation}\label{eq:genCanal3}
\frac{d A^{(p)}}{d\u_j^{(2)}} = \beta s A^{(p)} \frac{d B^{(p)}}{d\u_j^{(2)}}= \beta s A^{(i)} \frac{\u_j^{(2)}\sqrt{1-t}+\sqrt{t}\u_j^{(i,1)}}{B^{(p)}}\sqrt{1-t}.
\end{equation}
Now, we also have
\begin{eqnarray}\label{eq:genCanal4}
\frac{d}{d\u_j^{(2)}}\lp \frac{ A^{(i)}}{ZB^{(i)}}\rp & = &
\frac{1}{Z}\frac{d}{d\u_j^{(2)}}\lp \frac{ A^{(i)}}{B^{(i)}}\rp-\frac{ A^{(i)}}{Z^2B^{(i)}}\frac{dZ}{d\u_j^{(2)}}\nonumber \\
& = & \frac{1}{ZB^{(i)}}\frac{d A^{(i)}}{d\u_j^{(2)}}-\frac{A^{(i)}}{Z(B^{(i)})^2}\frac{d B^{(i)}}{d\u_j^{(2)}}
 -\frac{ A^{(i)}}{Z^2B^{(i)}}\frac{d \sum_{p=1}^{l}A^{(p)}}{d\u_j^{(2)}} \nonumber \\
 & = & \frac{\beta s A^{(i)}}{ZB^{(i)}}\frac{d B^{(i)}}{d\u_j^{(2)}}-\frac{A^{(i)}}{Z(B^{(i)})^2}\frac{d B^{(i)}}{d\u_j^{(2)}}
 -\frac{ A^{(i)}}{Z^2B^{(i)}}\sum_{p=1}^{l}\frac{d A^{(p)}}{d\u_j^{(2)}}\nonumber \\
 & = & \frac{\beta s A^{(i)}}{ZB^{(i)}}\frac{d B^{(i)}}{d\u_j^{(2)}}-\frac{A^{(i)}}{Z(B^{(i)})^2}\frac{d B^{(i)}}{d\u_j^{(2)}}
 -\frac{ A^{(i)}}{Z^2B^{(i)}}\sum_{p=1}^{l}\beta s A^{(p)} \frac{\u_j^{(2)}\sqrt{1-t}+\sqrt{t}\u_j^{(p,1)}}{B^{(p)}}\sqrt{1-t}\nonumber \\
\end{eqnarray}
A combination of (\ref{eq:genCanal1}), (\ref{eq:genCanal2}), (\ref{eq:genCanal3}), and (\ref{eq:genCanal4}) gives
\begin{eqnarray}\label{eq:genCanal5}
\mE_{\u^{(i,1)},\u^{(2)},\u^{(i,3)}}  \frac{ A^{(i)}\u_j^{(i,1)}\u_j^{(2)}}{ZB^{(i)}}
 & = &
\mE\lp\u_j^{(i,1)}\frac{d}{d\u_j^{(2)}}\lp \frac{ A^{(i)}}{ZB^{(i)}}\rp\rp \nonumber \\
 & = &
\mE\lp\frac{ A^{(i)}}{ZB^{(i)}}\lp \lp \beta s -\frac{1}{B^{(i)}}\rp\frac{\u_j^{(2)}\sqrt{1-t}+\sqrt{t}\u_j^{(i,1)}}{B^{(i)}}\u_j^{(i,1)}\sqrt{1-t}\rp\rp \nonumber \\
& &  -\mE\sum_{p=1}^{l}\frac{\beta s  A^{(i)}A^{(p)}\u_j^{(i,1)}\sqrt{1-t}}{Z^2B^{(i)}B^{(p)}} (\u_j^{(2)}\sqrt{1-t}+\sqrt{t}\u_j^{(p,1)})\nonumber \\
\end{eqnarray}

\subsection{Handling $\mE_{\u^{(i,1)},\u^{(2)},\u^{(i,3)}}  \frac{ A^{(i)}(\u_j^{(2)})^2}{ZB^{(i)}}$}
\label{sec:hand2}

In this section we study $\mE_{\u^{(i,1)},\u^{(2)},\u^{(i,3)}}  \frac{ A^{(i)}(\u_j^{(2)})^2}{ZB^{(i)}}$. We start with the following observation
\begin{equation}\label{eq:gen1anal1}
\mE_{\u^{(i,1)},\u^{(2)},\u^{(i,3)}}  \frac{ A^{(i)}(\u_j^{(2)})^2}{ZB^{(i)}} =
\mE\lp\mE (\u_j^{(2)}\u_j^{(2)})\frac{d}{d\u_j^{(2)}}\lp \frac{ A^{(i)}\u_j^{(2)}}{ZB^{(i)}}\rp\rp=\mE\lp\frac{ A^{(i)}}{ZB^{(i)}}+\u_j^{(2)}\frac{d}{d\u_j^{(2)}}\lp \frac{ A^{(i)}}{ZB^{(i)}}\rp\rp.
\end{equation}
A combination of (\ref{eq:genCanal4}) and (\ref{eq:gen1anal1}) gives
\begin{eqnarray}\label{eq:gen1anal5}
\mE_{\u^{(i,1)},\u^{(2)},\u^{(i,3)}}  \frac{ A^{(i)}(\u_j^{(2)})^2}{ZB^{(i)}}
 & = &
\mE\lp\frac{ A^{(i)}}{ZB^{(i)}}+\u_j^{(2)}\frac{d}{d\u_j^{(2)}}\lp \frac{ A^{(i)}}{ZB^{(i)}}\rp\rp \nonumber \\
 & = &
\mE\lp\frac{ A^{(i)}}{ZB^{(i)}}\lp 1 +\lp \beta s -\frac{1}{B^{(i)}}\rp\frac{\u_j^{(2)}\sqrt{1-t}+\sqrt{t}\u_j^{(i,1)}}{B^{(i)}}\u_j^{(2)}\sqrt{1-t}\rp\rp \nonumber \\
& &  -\mE\sum_{p=1}^{l}\frac{\beta s  A^{(i)}A^{(p)}\u_j^{(2)}\sqrt{1-t}}{Z^2B^{(i)}B^{(p)}} (\u_j^{(2)}\sqrt{1-t}+\sqrt{t}\u_j^{(p,1)})\nonumber \\
\end{eqnarray}

\subsection{Handling $\mE_{\u^{(i,1)},\u^{(2)},\u^{(i,3)}}  \frac{ A^{(i)}(\u_j^{(i,1)})^2}{ZB^{(i)}}$}
\label{sec:hand3}

Below we study $\mE_{\u^{(i,1)},\u^{(2)},\u^{(i,3)}}  \frac{ A^{(i)}(\u_j^{(i,1)})^2}{ZB^{(i)}}$. We start with the following
\begin{equation}\label{eq:gen2anal12}
\mE_{\u^{(i,1)},\u^{(2)},\u^{(i,3)}}  \frac{ A^{(i)}\u_j^{(i,1)}\u_j^{(2)}}{ZB^{(i)}} =
\mE\lp\sum_{p=1,p\neq i}^{l} \mE (\u_j^{(i,1)}\u_j^{(p,1)})\frac{d}{d\u_j^{(p,1)}}\lp \frac{ A^{(i)}\u_j^{(i,1)}}{ZB^{(i)}}\rp
+\mE (\u_j^{(i,1)}\u_j^{(i,1)})\frac{d}{d\u_j^{(i,1)}}\lp \frac{ A^{(i)}\u_j^{(i,1)}}{ZB^{(i)}}\rp\rp.
\end{equation}
We easily also have
\begin{equation}\label{eq:gen2anal13}
\mE_{\u^{(i,1)},\u^{(2)},\u^{(i,3)}}  \frac{ A^{(i)}\u_j^{(i,1)}\u_j^{(i,1)}}{ZB^{(i)}} =\mE\lp
\sum_{p=1,p\neq i}^{l} (\x^{(i)})^T\x^{(p)}\frac{d}{d\u_j^{(p,1)}}\lp \frac{ A^{(i)}\u_j^{(i,1)}}{ZB^{(i)}}\rp
+(\x^{(i)})^T\x^{(i)}\frac{d}{d\u_j^{(i,1)}}\lp \frac{ A^{(i)}\u_j^{(i,1)}}{ZB^{(i)}}\rp\rp.
\end{equation}
For $p\neq i$ we find
\begin{eqnarray}\label{eq:gen2anal14}
\frac{d}{d\u_j^{(p,1)}}\lp \frac{ A^{(i)}\u_j^{(i,1)}}{ZB^{(i)}}\rp=\frac{ A^{(i)}\u_j^{(i,1)}}{B^{(i)}}\frac{d}{d\u_j^{(p,1)}}\lp \frac{1}{Z}\rp
=-\frac{ A^{(i)}\u_j^{(i,1)}}{Z^2B^{(i)}}\frac{dZ}{d\u_j^{(p,1)}}=-\frac{ A^{(i)}\u_j^{(i,1)}}{Z^2B^{(i)}}\frac{d \sum_{i=1}^{l }A^{(i)}}{d\u_j^{(p,1)}}.
\end{eqnarray}
From (\ref{eq:gen2anal14}), (\ref{eq:genanal15}), and (\ref{eq:genanal16}) we find
\begin{equation}\label{eq:gen2anal17}
\frac{d}{d\u_j^{(p,1)}}\lp \frac{ A^{(i)}\u_j^{(i,1)}}{ZB^{(i)}}\rp=-\frac{ A^{(i)}\u_j^{(i,1)}}{Z^2B^{(i)}}\frac{d \sum_{i=1}^{l }A^{(i)}}{d\u_j^{(p,1)}}
=-\frac{ A^{(i)}\u_j^{(i,1)}}{Z^2B^{(i)}}\frac{d A^{(p)}}{d\u_j^{(p,1)}}=-\frac{ \beta s A^{(i)}A^{(p)}\u_j^{(i,1)}\sqrt{t}(\u_j^{(p,1)}\sqrt{t}+\sqrt{1-t}\u_j^{(2)})}{Z^2B^{(i)}B^{(p)}}.
\end{equation}
For $p=i$ we have
\begin{eqnarray}\label{eq:gen2anal18}
\frac{d}{d\u_j^{(i,1)}}\lp \frac{ A^{(i)}\u_j^{(i,1)}}{ZB^{(i)}}\rp & = &\u_j^{(i,1)}\frac{d}{d\u_j^{(i,1)}}\lp \frac{ A^{(i)}}{ZB^{(i)}}\rp
+\frac{ A^{(i)}}{ZB^{(i)}}.
\end{eqnarray}
Combining (\ref{eq:genanal18}) and (\ref{eq:gen2anal18}) we have
\begin{equation}\label{eq:gen2anal18a}
\frac{d}{d\u_j^{(i,1)}}\lp \frac{ A^{(i)}\u_j^{(i,1)}}{ZB^{(i)}}\rp  =  \frac{\beta s \u_j^{(i,1)}A^{(i)}}{ZB^{(i)}}\frac{d B^{(i)}}{d\u_j^{(i,1)}}-\frac{\u_j^{(i,1)}A^{(i)}}{Z(B^{(i)})^2}\frac{d B^{(i)}}{d\u_j^{(i,1)}}
 -\frac{\beta s  A^{(i)}A^{(i)}\u_j^{(i,1)}\sqrt{t}(\u_j^{(i,1)}\sqrt{t}+\sqrt{1-t}\u_j^{(2)})}{Z^2B^{(i)}B^{(i)}} +\frac{ A^{(i)}}{ZB^{(i)}}.
\end{equation}
A combination of (\ref{eq:gen2anal13}), (\ref{eq:gen2anal17}), and (\ref{eq:gen2anal18a}) gives
\begin{eqnarray}\label{eq:gen2anal19}
\mE_{\u^{(i,1)},\u^{(2)},\u^{(i,3)}}  \frac{ A^{(i)}(\u_j^{(i,1)})^2}{ZB^{(i)}} & = &
\mE\lp\frac{\beta s \u^{(i,1)}A^{(i)}}{ZB^{(i)}}\frac{d B^{(i)}}{d\u_j^{(i,1)}}-\frac{\u^{(i,1)}A^{(i)}}{Z(B^{(i)})^2}\frac{d B^{(i)}}{d\u_j^{(i,1)}} +\frac{ A^{(i)}}{ZB^{(i)}}\rp \nonumber \\
& & -\mE\sum_{p=1}^{l} (\x^{(i)})^T\x^{(p)} \frac{\beta s  A^{(i)}A^{(p)}\u^{(i,1)}\sqrt{t}(\u_j^{(p,1)}\sqrt{t}+\sqrt{1-t}\u_j^{(2)})}{Z^2B^{(i)}B^{(p)}}\nonumber \\
& = &
\mE\lp\frac{ A^{(i)}}{ZB^{(i)}}\lp 1 +\lp \beta s -\frac{1}{B^{(i)}}\rp\frac{\u_j^{(2)}\sqrt{1-t}+\sqrt{t}\u_j^{(i,1)}}{B^{(i)}}\u_j^{(i,1)}\sqrt{t}\rp\rp \nonumber \\ & & -\mE\sum_{p=1}^{l} (\x^{(i)})^T\x^{(p)} \frac{ \beta s A^{(i)}A^{(p)}\u_j^{(i,1)}\sqrt{t}(\u_j^{(p,1)}\sqrt{t}+\sqrt{1-t}\u_j^{(2)})}{Z^2B^{(i)}B^{(p)}}.
\end{eqnarray}

\subsection{Handling $\mE_{\u^{(i,1)},\u^{(2)},\u^{(i,3)}}  \frac{ A^{(i)}\u^{(i,3)}}{Z}$}
\label{sec:hand4}

Below we study $\mE_{\u^{(i,1)},\u^{(2)},\u^{(i,3)}}  \frac{ A^{(i)}\u^{(i,3)}}{Z}$.
We start with the following observation
\begin{equation}\label{eq:genDanal12}
\mE_{\u^{(i,1)},\u^{(2)},\u^{(i,3)}}  \frac{ A^{(i)}\u^{(i,3)}}{Z} =
\mE\lp\sum_{p=1,p\neq i}^{l} \mE (\u^{(i,3)}\u^{(p,3)})\frac{d}{d\u^{(p,3)}}\lp \frac{ A^{(i)}}{Z}\rp
+\mE (\u^{(i,3)}\u^{(i,3)})\frac{d}{d\u^{(i,3)}}\lp \frac{ A^{(i)}}{Z}\rp\rp.
\end{equation}
Trivially we then also have
\begin{equation}\label{eq:genDanal13}
\mE_{\u^{(i,1)},\u^{(2)},\u^{(i,3)}}  \frac{ A^{(i)}\u^{(i,3)}}{Z} =\mE\lp
\sum_{p=1,p\neq i}^{l} (\x^{(i)})^T\x^{(p)}\frac{d}{d\u^{(p,3)}}\lp \frac{ A^{(i)}}{Z}\rp
+(\x^{(i)})^T\x^{(i)}\frac{d}{d\u^{(i,3)}}\lp \frac{ A^{(i)}}{Z}\rp\rp.
\end{equation}
For $p\neq i$ we find
\begin{eqnarray}\label{eq:genDanal14}
\frac{d}{d\u^{(p,3)}}\lp \frac{ A^{(i)}}{Z}\rp=A^{(i)}\frac{d}{d\u^{(p,3)}}\lp \frac{1}{Z}\rp
=-\frac{ A^{(i)}}{Z^2}\frac{dZ}{d\u^{(p,3)}}=-\frac{ A^{(i)}}{Z^2}\frac{d \sum_{i=1}^{l }A^{(i)}}{d\u^{(p,3)}}.
\end{eqnarray}
From (\ref{eq:genanal7}) we find
\begin{equation}\label{eq:genDanal16}
\frac{d A^{(p)}}{d\u^{(p,3)}} = \beta  A^{(p)} \sqrt{1-t}.
\end{equation}
Combining (\ref{eq:genDanal14}) and (\ref{eq:genDanal16}) we obtain
\begin{equation}\label{eq:genDanal17}
\frac{d}{d\u^{(p,3)}}\lp \frac{ A^{(i)}}{Z}\rp=-\frac{\beta A^{(i)}  A^{(p)} \sqrt{1-t}}{Z^2}.
\end{equation}
For $p=i$ we have
\begin{equation}\label{eq:genDanal18}
\frac{d}{d\u^{(i,3)}}\lp \frac{ A^{(i)}}{Z}\rp  =
\frac{1}{Z}\frac{d A^{(i)}}{d\u^{(i,3)}}-\frac{ A^{(i)}}{Z^2}\frac{dZ}{d\u^{(i,3)}} =  \frac{\beta A^{(i)}\sqrt{1-t}}{Z} -\frac{ A^{(i)}}{Z^2}\frac{d A^{(i)}}{d\u^{(i,3)}}  = \frac{\beta A^{(i)}\sqrt{1-t}}{Z} -\frac{ A^{(i)}A^{(i)}\beta\sqrt{1-t}}{Z^2}.
\end{equation}
A combination of (\ref{eq:genDanal13}), (\ref{eq:genDanal17}), and (\ref{eq:genDanal18}) gives
\begin{equation}\label{eq:genDanal19}
\mE_{\u^{(i,1)},\u^{(2)},\u^{(i,3)}}  \frac{ A^{(i)}\u^{(i,3)}}{Z}  =
\mE \frac{\beta A^{(i)}\sqrt{1-t}}{Z} -\mE\sum_{p=1}^{l} (\x^{(i)})^T\x^{(p)}\frac{\beta A^{(i)}  A^{(p)} \sqrt{1-t}}{Z^2}.
\end{equation}

\subsection{Connecting all pieces together}
\label{sec:conalt}

Combining (\ref{eq:genanal11}), (\ref{eq:genanal19}), (\ref{eq:genCanal5}), (\ref{eq:gen1anal5}), (\ref{eq:gen2anal19}), and (\ref{eq:genDanal19}) we obtain
\begin{eqnarray}\label{eq:conalt1}
\frac{d\psi(\calX,\beta,s,t)}{dt}
& = &   \mE_{\u^{(i,1)},\u^{(2)},\u^{(i,3)}} \frac{s}{2\sqrt{n}}\sum_{j=1}^{m} \sum_{i=1}^{l} \frac{ A^{(i)}\lp(\u_j^{(i,1)})^2-(\u_j^{(2)})^2+\u_j^{(i,1)}\u_j^{(2)}\lp\frac{\sqrt{1-t}}{\sqrt{t}}-\frac{\sqrt{t}}{\sqrt{1-t}}\rp \rp }{ZB^{(i)}} \nonumber \\
& &  - \mE_{\u^{(i,1)},\u^{(2)},\u^{(i,3)}} \frac{1}{2\sqrt{n}}  \sum_{i=1}^{l}
\frac{A^{(i)}\u^{(i,3)}}{Z\sqrt{1-t}} \nonumber \\
& = & \frac{s}{2\sqrt{n}}\sum_{j=1}^{m} \sum_{i=1}^{l}
\mE\lp\frac{ A^{(i)}}{ZB^{(i)}}\lp 1 +\lp \beta s -\frac{1}{B^{(i)}}\rp\frac{\u_j^{(2)}\sqrt{1-t}+\sqrt{t}\u_j^{(i,1)}}{B^{(i)}}\u_j^{(i,1)}\sqrt{t}\rp\rp \nonumber \\ & & -\frac{s}{2\sqrt{n}}\sum_{j=1}^{m} \sum_{i=1}^{l}\mE\sum_{p=1}^{l} (\x^{(i)})^T\x^{(p)} \frac{ \beta s A^{(i)}A^{(p)}\u_j^{(i,1)}\sqrt{t}(\u_j^{(p,1)}\sqrt{t}+\sqrt{1-t}\u_j^{(2)})}{Z^2B^{(i)}B^{(p)}} \nonumber \\
&  & -\frac{s}{2\sqrt{n}}\sum_{j=1}^{m} \sum_{i=1}^{l}
\mE\lp\frac{ A^{(i)}}{ZB^{(i)}}\lp 1 +\lp \beta s -\frac{1}{B^{(i)}}\rp\frac{\u_j^{(2)}\sqrt{1-t}+\sqrt{t}\u_j^{(i,1)}}{B^{(i)}}\u_j^{(2)}\sqrt{1-t}\rp\rp \nonumber \\
& &  +\frac{s}{2\sqrt{n}}\sum_{j=1}^{m} \sum_{i=1}^{l}\mE\sum_{p=1}^{l}\frac{\beta s  A^{(i)}A^{(p)}\u_j^{(2)}\sqrt{1-t}}{Z^2B^{(i)}B^{(p)}} (\u_j^{(2)}\sqrt{1-t}+\sqrt{t}\u_j^{(p,1)})\nonumber \\
&  & +\frac{s}{2\sqrt{n}}\sum_{j=1}^{m} \sum_{i=1}^{l}
\mE\lp\frac{ A^{(i)}}{ZB^{(i)}}\lp \lp \beta s -\frac{1}{B^{(i)}}\rp\frac{\u_j^{(2)}\sqrt{1-t}+\sqrt{t}\u_j^{(i,1)}}{B^{(i)}}\u_j^{(2)}\sqrt{t}\rp\rp\frac{\sqrt{1-t}}{\sqrt{t}} \nonumber  \\
& & -\frac{s}{2\sqrt{n}}\sum_{j=1}^{m} \sum_{i=1}^{l}\mE\sum_{p=1}^{l} (\x^{(i)})^T\x^{(p)} \frac{\beta s  A^{(i)}A^{(p)}\u_j^{(2)}\sqrt{t}(\u_j^{(p,1)}\sqrt{t}+\sqrt{1-t}\u_j^{(2)})}{Z^2B^{(i)}B^{(p)}}\frac{\sqrt{1-t}}{\sqrt{t}}\nonumber \\
&  & -\frac{s}{2\sqrt{n}}\sum_{j=1}^{m} \sum_{i=1}^{l}
\mE\lp\frac{ A^{(i)}}{ZB^{(i)}}\lp \lp \beta s -\frac{1}{B^{(i)}}\rp\frac{\u_j^{(2)}\sqrt{1-t}+\sqrt{t}\u_j^{(i,1)}}{B^{(i)}}\u_j^{(i,1)}\sqrt{1-t}\rp\rp \frac{\sqrt{t}}{\sqrt{1-t}} \nonumber \\
& &  +\frac{s}{2\sqrt{n}}\sum_{j=1}^{m} \sum_{i=1}^{l}\mE\sum_{p=1}^{l}\frac{\beta s  A^{(i)}A^{(p)}\u_j^{(i,1)}\sqrt{1-t}}{Z^2B^{(i)}B^{(p)}} (\u_j^{(2)}\sqrt{1-t}+\sqrt{t}\u_j^{(p,1)})\frac{\sqrt{t}}{\sqrt{1-t}}\nonumber \\
& & -\frac{1}{2\sqrt{n}}  \sum_{i=1}^{l} \lp \frac{\beta A^{(i)}}{Z} -\mE\sum_{p=1}^{l} (\x^{(i)})^T\x^{(p)}\frac{\beta A^{(i)}  A^{(p)} }{Z^2}\rp.
\end{eqnarray}
After cancelling all terms that can be cancelled and rearranging a bit we have
\begin{eqnarray}\label{eq:conalt2}
\frac{d\psi(\calX,\beta,s,t)}{dt}
 & = & -\frac{s}{2\sqrt{n}} \sum_{i=1}^{l}\mE\sum_{p=1}^{l}\sum_{j=1}^{m} (\x^{(i)})^T\x^{(p)} \frac{ \beta s A^{(i)}A^{(p)}\u_j^{(i,1)}\sqrt{t}(\u_j^{(p,1)}\sqrt{t}+\sqrt{1-t}\u_j^{(2)})}{Z^2B^{(i)}B^{(p)}} \nonumber \\
& &  +\frac{s}{2\sqrt{n}} \sum_{i=1}^{l}\mE\sum_{p=1}^{l}\sum_{j=1}^{m}\frac{\beta s  A^{(i)}A^{(p)}\u_j^{(2)}\sqrt{1-t}}{Z^2B^{(i)}B^{(p)}} (\u_j^{(2)}\sqrt{1-t}+\sqrt{t}\u_j^{(i,1)})\nonumber \\
& & -\frac{s}{2\sqrt{n}} \sum_{i=1}^{l}\mE\sum_{p=1}^{l}\sum_{j=1}^{m} (\x^{(i)})^T\x^{(p)} \frac{\beta s  A^{(i)}A^{(p)}\u_j^{(2)}\sqrt{1-t}(\u_j^{(p,1)}\sqrt{t}+\sqrt{1-t}\u_j^{(2)})}{Z^2B^{(i)}B^{(p)}}\nonumber \\
& &  +\frac{s}{2\sqrt{n}} \sum_{i=1}^{l}\mE\sum_{p=1}^{l}\sum_{j=1}^{m}\frac{\beta s  A^{(i)}A^{(p)}\u_j^{(i,1)}\sqrt{t}}{Z^2B^{(i)}B^{(p)}} (\u_j^{(2)}\sqrt{1-t}+\sqrt{t}\u_j^{(p,1)})\nonumber \\
& & -\frac{1}{2\sqrt{n}}   \mE\sum_{i=1}^{l}\sum_{p=1}^{l} (1-(\x^{(i)})^T\x^{(p)})\frac{\beta A^{(i)}  A^{(p)} }{Z^2}.
\end{eqnarray}
Grouping the first and third and second and fourth term we can write
\begin{eqnarray}\label{eq:conalt3}
\frac{d\psi(\calX,\beta,s,t)}{dt}
 & = & -\frac{s}{2\sqrt{n}} \sum_{i=1}^{l}\mE\sum_{p=1}^{l}\sum_{j=1}^{m} (\x^{(i)})^T\x^{(p)} \frac{ \beta s A^{(i)}A^{(p)}(\u_j^{(i,1)}\sqrt{t}+\u_j^{(2)}\sqrt{1-t})(\u_j^{(p,1)}\sqrt{t}+\sqrt{1-t}\u_j^{(2)})}{Z^2B^{(i)}B^{(p)}} \nonumber \\
& &  +\frac{s}{2\sqrt{n}} \sum_{i=1}^{l}\mE\sum_{p=1}^{l}\sum_{j=1}^{m}\frac{\beta s  A^{(i)}A^{(p)}(\u_j^{(i,1)}\sqrt{t}+\u_j^{(2)}\sqrt{1-t})(\u_j^{(2)}\sqrt{1-t}+\sqrt{t}\u_j^{(p,1)})}{Z^2B^{(i)}B^{(p)}} \nonumber \\
& & -\frac{1}{2\sqrt{n}}   \mE\sum_{i=1}^{l}\sum_{p=1}^{l} (1-(\x^{(i)})^T\x^{(p)})\frac{\beta A^{(i)}  A^{(p)} }{Z^2}.
\end{eqnarray}
Grouping the first and second term we can also write
\begin{multline}\label{eq:conalt4}
\frac{d\psi(\calX,\beta,s,t)}{dt}
  =  \frac{s}{2\sqrt{n}} \sum_{i=1}^{l}\mE\sum_{p=1}^{l}\sum_{j=1}^{m} (1-(\x^{(i)})^T\x^{(p)}) \frac{ \beta s A^{(i)}A^{(p)}(\u_j^{(i,1)}\sqrt{t}+\u_j^{(2)}\sqrt{1-t})(\u_j^{(p,1)}\sqrt{t}+\sqrt{1-t}\u_j^{(2)})}{Z^2B^{(i)}B^{(p)}} \\
  -\frac{1}{2\sqrt{n}}   \mE\sum_{i=1}^{l}\sum_{p=1}^{l} (1-(\x^{(i)})^T\x^{(p)})\frac{\beta A^{(i)}  A^{(p)} }{Z^2},
\end{multline}
and
\begin{equation}\label{eq:conalt5}
\frac{d\psi(\calX,\beta,s,t)}{dt} =  -\frac{\beta}{2\sqrt{n}} \sum_{i=1}^{l}\mE\sum_{p=1}^{l} \frac{ A^{(i)}A^{(p)}}{Z^2}(1-(\x^{(i)})^T\x^{(p)})\lp1-
\sum_{j=1}^{m}\frac{(\u_j^{(i,1)}\sqrt{t}+\u_j^{(2)}\sqrt{1-t})(\u_j^{(p,1)}\sqrt{t}+\sqrt{1-t}\u_j^{(2)})}{B^{(i)}B^{(p)}}\rp.
\end{equation}
Finally we have
\begin{equation}\label{eq:conalt5}
\frac{d\psi(\calX,\beta,s,t)}{dt} =  -\frac{\beta}{2\sqrt{n}} \sum_{i=1}^{l}\mE\sum_{p=1}^{l} \frac{ A^{(i)}A^{(p)}}{Z^2}(1-(\x^{(i)})^T\x^{(p)})\lp1-
\frac{(\u^{(i,1)}\sqrt{t}+\u^{(2)}\sqrt{1-t})^T(\u^{(p,1)}\sqrt{t}+\sqrt{1-t}\u^{(2)})}{B^{(i)}B^{(p)}}\rp.
\end{equation}
Since $\|\u^{(i,1)}\sqrt{t}+\u^{(2)}\sqrt{1-t}\|_2=B^{(i)}$ we easily have $\frac{\psi(\calX,\beta,s,t)}{dt}\leq 0$ and function $\psi(\calX,\beta,s,t)$ is indeed decreasing in $t$. We summarize the obtained results in the following theorem.

\begin{theorem}
\label{thm:thm1}
Let $\calX=\{\x^{(1)},\x^{(2)},\dots,\x^{(l)}\}$, where $\x^{(i)}\in S^{n-1},1\leq i\leq l$, and $S^{n-1}$ is the unit sphere in $\mR^n$. Let $G$ be an $m \times n$ matrix of i.i.d. standard normals, $\u^{(2)}$ an $m \times 1$ vector of i.i.d. standard normals, and $\h$ an $n \times 1$ vector of i.i.d. standard normals. Moreover, let $G$, $\u^{(2)}$, and $\h$ be independent of each other. Then the following function $\psi(\calX,\beta,s,t)$
\begin{eqnarray}\label{eq:thm1eq1}
\psi(\calX,\beta,s,t)= \mE_{G,\u^{(2)},\h} \frac{1}{\beta\sqrt{n}} \log\lp \sum_{i=1}^{l}e^{\beta\lp s\|\sqrt{t}
 G\x^{(i)}+\sqrt{1-t}\u^{(2)}\|_2+\sqrt{1-t}\h^T\x^{(i)}\rp} \rp,
\end{eqnarray}
is decreasing in $t$.
\end{theorem}
\begin{proof}
  Follows from the above discussion.
\end{proof}
\begin{corollary}
  Assuming the setup of Theorem \ref{thm:thm1}, we have
\begin{eqnarray}\label{eq:co1eq1}
\psi(\calX,\beta,s,t)= \psi(\calX,\beta,s,0)+\int_{0}^{t}\frac{d\psi(\calX,\beta,s,t)}{dt}dt,
\end{eqnarray}
as well as the following comparison principle
\begin{eqnarray}\label{eq:co1eq2}
\psi(\calX,\beta,s,0) \geq  \psi(\calX,\beta,s,t)\geq \psi(\calX,\beta,s,1).
\end{eqnarray}
\end{corollary}
\begin{proof}
  Follows trivially from the above theorem by noting that $\frac{d\psi(\calX,\beta,s,t)}{dt}\leq 0$.
\end{proof}

\subsection{The power of simulations}
\label{sec:genconsim}

In this section we present a few numerical results that showcase the strength of the above theoretical results as well as the ultimate power of numerical simulations. We chose $m=5$, $n=5$, $l=10$, and selected set $\calX$ as the columns of the following matrix
\begin{equation}
X^{+}=\begin{bmatrix}
-0.7998 & 0.1004 & -0.7599 & 0.6616 & 0.5864 & -0.4010 & -0.0148 & -0.8320 & 0.3187 & -0.4861 \\
0.1760 & 0.0704 & 0.1056 & -0.1369 & -0.6259 & -0.5289 & -0.3740 & 0.3140 & 0.6299 & -0.5494 \\
0.0806 & -0.9085 & -0.3381 & -0.1970 & -0.1438 & 0.4863 & 0.5832 & 0.0840 & -0.2299 & -0.2647 \\
0.5487 & -0.3120 & -0.5447 & 0.5673 & 0.4870 & -0.5239 & 0.0407 & -0.2955 & 0.3913 & 0.5113 \\
-0.1476 & 0.2497 & -0.0208 & 0.4276 & 0.0808 & -0.2202 & -0.7198 & 0.3389 & 0.5438 & -0.3611
\end{bmatrix}.
\end{equation}
In other words, we have
\begin{equation}\label{eq:sim1}
  \calX^{+}=\{X^{+}_{:,1},X^{+}_{:,2},\dots,X^{+}_{:,l}\}.
\end{equation}
Clearly, set $\calX^{+}$ (and matrix $X^{+}$) is totally random, i.e. there is nothing specific about it besides that the norm of each of its elements is equal to one. For such a set we then simulated derivatives $\frac{d\psi(\calX,\beta,s,t)}{dt}$ according to (\ref{eq:genanal11}) (which we refer to as the standard interpolation) and according to (\ref{eq:conalt5}) (which we refer to as the computed interpolation) and computed $\psi(\calX,\beta,s,t)$ according to (\ref{eq:co1eq1}). We also simulated $\psi(\calX,\beta,s,t)$ directly (i.e. without interpolating computations) through (\ref{eq:genanal8}). Throughout all simulations, every quantity of interest that needed to be averaged was averaged over a set of $3e4$ experiments. Also, through all the simulations we set $\beta=10$. We simulated two different scenarios with all other parameters being the same: 1) $s=1$ and 2) $s=-1$.

\textbf{\underline{\emph{1) $s=1$ -- numerical results}}}

The results that we obtained for $s=1$ are presented in Figure \ref{fig:gensplus1xnorm1psi} and Table \ref{tab:gensplus1xnorm1psi}. Figure \ref{fig:gensplus1xnorm1psi} shows the entire range for $t$ (i.e. its shows the values for $t\in(0,1)$) whereas Table \ref{tab:gensplus1xnorm1psi} focuses on several concrete values of $t$ and shows obtained values of all key quantities. As both, Figure \ref{fig:gensplus1xnorm1psi} and Table \ref{tab:gensplus1xnorm1psi}, show, the agreement between all presented results is rather overwhelming.

\begin{figure}[htb]
\centering
\centerline{\epsfig{figure=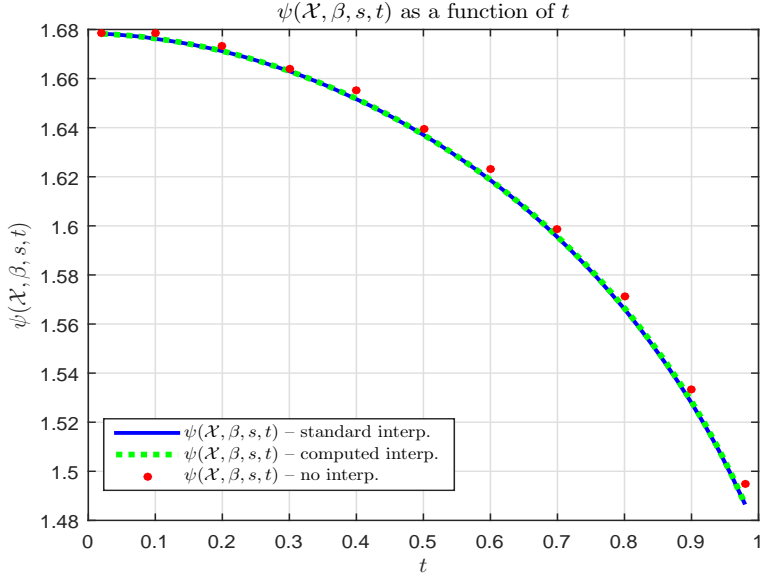,width=11.5cm,height=8cm}}
\caption{$\psi(\calX,\beta,s,t)$ as a function of $t$; $m=5$, $n=5$, $l=10$, $\calX=\calX^{+}$, $\beta=3$, $s=1$}
\label{fig:gensplus1xnorm1psi}
\end{figure}
\begin{table}[h]
\caption{Simulated results --- $m=5$, $n=5$, $l=10$, $\calX=\calX^{+}$, $\beta=3$, $s=1$}\vspace{.1in}
\hspace{-0in}\centering
\begin{tabular}{||c||c|c|c|c|c||}\hline\hline
$ t$  &  $\frac{d\psi}{dt}$; (\ref{eq:genanal11}) & $\frac{d\psi}{dt}$;  (\ref{eq:conalt5}) & $\psi$;  (\ref{eq:genanal11}) and (\ref{eq:co1eq1}) & $\psi$; (\ref{eq:conalt5}) and (\ref{eq:co1eq1}) & $\psi$;  (\ref{eq:genanal8})\\  \hline\hline
$ 0.1000 $ & $ -0.0411 $ & $ -0.0331 $ & $\bl{\mathbf{ 1.6762 }}$ & $\bl{\mathbf{ 1.6763 }}$ & $\mathbf{ 1.6787 }$  \\ \hline
$ 0.2000 $ & $ -0.0651 $ & $ -0.0635 $ & $\bl{\mathbf{ 1.6711 }}$ & $\bl{\mathbf{ 1.6712 }}$ & $\mathbf{ 1.6733 }$  \\ \hline
$ 0.3000 $ & $ -0.0932 $ & $ -0.0937 $ & $\bl{\mathbf{ 1.6630 }}$ & $\bl{\mathbf{ 1.6631 }}$ & $\mathbf{ 1.6640 }$  \\ \hline
$ 0.4000 $ & $ -0.1251 $ & $ -0.1258 $ & $\bl{\mathbf{ 1.6516 }}$ & $\bl{\mathbf{ 1.6518 }}$ & $\mathbf{ 1.6551 }$  \\ \hline
$ 0.5000 $ & $ -0.1610 $ & $ -0.1613 $ & $\bl{\mathbf{ 1.6371 }}$ & $\bl{\mathbf{ 1.6371 }}$ & $\mathbf{ 1.6395 }$  \\ \hline
$ 0.6000 $ & $ -0.2041 $ & $ -0.2005 $ & $\bl{\mathbf{ 1.6186 }}$ & $\bl{\mathbf{ 1.6187 }}$ & $\mathbf{ 1.6232 }$  \\ \hline
$ 0.7000 $ & $ -0.2541 $ & $ -0.2528 $ & $\bl{\mathbf{ 1.5954 }}$ & $\bl{\mathbf{ 1.5956 }}$ & $\mathbf{ 1.5986 }$  \\ \hline
$ 0.8000 $ & $ -0.3185 $ & $ -0.3219 $ & $\bl{\mathbf{ 1.5661 }}$ & $\bl{\mathbf{ 1.5665 }}$ & $\mathbf{ 1.5711 }$  \\ \hline
$ 0.9000 $ & $ -0.4287 $ & $ -0.4277 $ & $\bl{\mathbf{ 1.5279 }}$ & $\bl{\mathbf{ 1.5285 }}$ & $\mathbf{ 1.5336 }$  \\ \hline \hline
\end{tabular}
\label{tab:gensplus1xnorm1psi}
\end{table}

\textbf{\underline{\emph{2) $s=-1$ -- numerical results}}}

The results that we obtained for $s=-1$ are presented in Figure \ref{fig:gensmin1xnorm1psi} and Table \ref{tab:gensmin1xnorm1psi}. Figure \ref{fig:gensmin1xnorm1psi} again shows the entire range for $t$, whereas Table \ref{tab:gensmin1xnorm1psi} focuses on several concrete values of $t$ and shows explicitly the obtained results. As was the case for $s=1$, here we again have that both, Figure \ref{fig:gensmin1xnorm1psi} and Table \ref{tab:gensmin1xnorm1psi}, show that the agreement between all presented results is rather overwhelming.

\begin{figure}[htb]
\centering
\centerline{\epsfig{figure=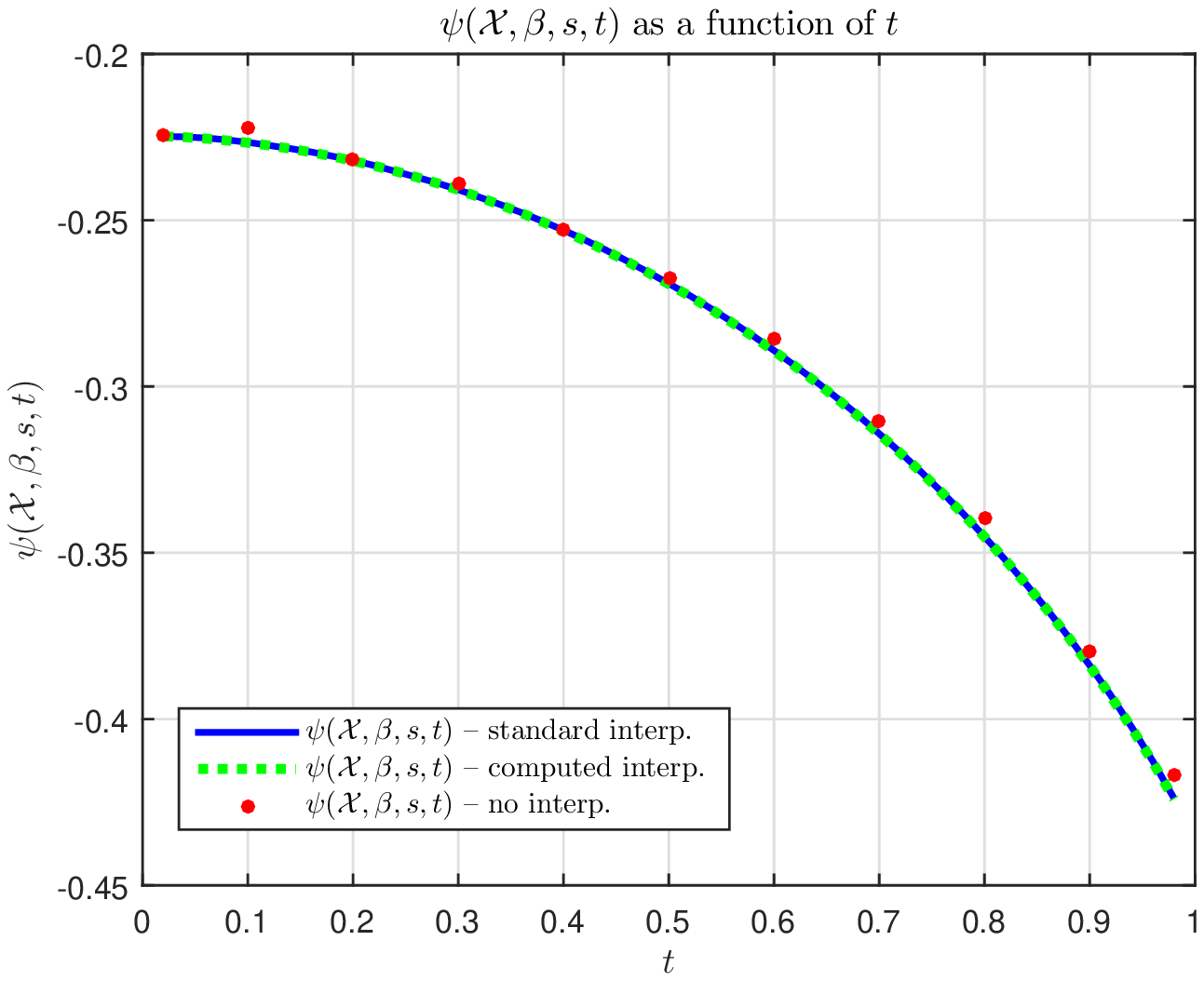,width=11.5cm,height=8cm}}
\caption{$\psi(\calX,\beta,s,t)$ as a function of $t$; $m=5$, $n=5$, $l=10$, $\calX=\calX^{+}$, $\beta=3$, $s=-1$}
\label{fig:gensmin1xnorm1psi}
\end{figure}
\begin{table}[h]
\caption{Simulated results --- $m=5$, $n=5$, $l=10$, $\calX=\calX^{+}$, $\beta=3$, $s=-1$}\vspace{.1in}
\hspace{-0in}\centering
\begin{tabular}{||c||c|c|c|c|c||}\hline\hline
$ t$  &  $\frac{d\psi}{dt}$; (\ref{eq:genanal11}) & $\frac{d\psi}{dt}$;  (\ref{eq:conalt5}) & $\psi$;  (\ref{eq:genanal11}) and (\ref{eq:co1eq1}) & $\psi$; (\ref{eq:conalt5}) and (\ref{eq:co1eq1}) & $\psi$;  (\ref{eq:genanal8})\\  \hline\hline
$ 0.1000 $ & $ -0.0363 $ & $ -0.0344 $ & $\bl{\mathbf{ -0.2266 }}$ & $\bl{\mathbf{ -0.2267 }}$ & $\mathbf{ -0.2224 }$\\ \hline
$ 0.2000 $ & $ -0.0655 $ & $ -0.0680 $ & $\bl{\mathbf{ -0.2320 }}$ & $\bl{\mathbf{ -0.2321 }}$ & $\mathbf{ -0.2320 }$\\ \hline
$ 0.3000 $ & $ -0.0998 $ & $ -0.1010 $ & $\bl{\mathbf{ -0.2408 }}$ & $\bl{\mathbf{ -0.2409 }}$ & $\mathbf{ -0.2391 }$\\ \hline
$ 0.4000 $ & $ -0.1333 $ & $ -0.1377 $ & $\bl{\mathbf{ -0.2530 }}$ & $\bl{\mathbf{ -0.2532 }}$ & $\mathbf{ -0.2525 }$\\ \hline
$ 0.5000 $ & $ -0.1737 $ & $ -0.1751 $ & $\bl{\mathbf{ -0.2690 }}$ & $\bl{\mathbf{ -0.2692 }}$ & $\mathbf{ -0.2676 }$\\ \hline
$ 0.6000 $ & $ -0.2208 $ & $ -0.2189 $ & $\bl{\mathbf{ -0.2892 }}$ & $\bl{\mathbf{ -0.2893 }}$ & $\mathbf{ -0.2854 }$\\ \hline
$ 0.7000 $ & $ -0.2768 $ & $ -0.2716 $ & $\bl{\mathbf{ -0.3142 }}$ & $\bl{\mathbf{ -0.3143 }}$ & $\mathbf{ -0.3106 }$\\ \hline
$ 0.8000 $ & $ -0.3431 $ & $ -0.3348 $ & $\bl{\mathbf{ -0.3452 }}$ & $\bl{\mathbf{ -0.3452 }}$ & $\mathbf{ -0.3399 }$ \\ \hline
$ 0.9000 $ & $ -0.4270 $ & $ -0.4283 $ & $\bl{\mathbf{ -0.3839 }}$ & $\bl{\mathbf{ -0.3840 }}$ & $\mathbf{ -0.3798 }$\\ \hline \hline
\end{tabular}
\label{tab:gensmin1xnorm1psi}
\end{table}

\subsection{$\beta\rightarrow \infty$}
\label{sec:betainf}

It is often of particular interest to study the following limiting behavior of $\xi(\calX,\beta,s)$
\begin{eqnarray}\label{eq:betainf1}
\lim_{\beta\rightarrow\infty} \xi(\calX,\beta,s)=\lim_{\beta\rightarrow\infty} \mE_G \frac{1}{\beta\sqrt{n}} \log\lp \sum_{i=1}^{l}e^{\beta\lp s\|
 G\x^{(i)}\|_2\rp} \rp=\mE_G \frac{\max_{\x^{(i)}\in \calX} \lp s\|
 G\x^{(i)}\|_2\rp}{\sqrt{n}},
\end{eqnarray}

\subsubsection{$s=1$ -- a Slepian's comparison principle}
\label{sec:betainfsplus1}

In particular when $s=1$ we have
\begin{eqnarray}\label{eq:betainfsplus1}
\lim_{\beta\rightarrow\infty} \xi(\calX,\beta,1)=\mE_G \frac{\max_{\x^{(i)}\in \calX} \lp \|
 G\x^{(i)}\|_2\rp}{\sqrt{n}}.
\end{eqnarray}
Recalling that $\xi(\calX,\beta,1)=\psi(\calX,\beta,1,1)$ and utilizing the above machinery we have
\begin{eqnarray}\label{eq:betainfsplus2}
\mE_G \frac{\max_{\x^{(i)}\in \calX} \lp \|
 G\x^{(i)}\|_2\rp}{\sqrt{n}} & = & \lim_{\beta\rightarrow\infty} \xi(\calX,\beta,1)=
 \lim_{\beta\rightarrow\infty} \psi(\calX,\beta,1,1) \nonumber \\
& \leq &  \lim_{\beta\rightarrow\infty} \psi(\calX,\beta,1,0)=
\mE_{\u^{(2)},\h} \frac{\max_{\x^{(i)}\in \calX} \lp \|\u^{(2)}\|_2 +\h^T\x^{(i)}\rp}{\sqrt{n}}.
\end{eqnarray}
(\ref{eq:betainfsplus2}) is of course a form of the well-known Slepian comparison principle \cite{Slep62}. Namely, based on the Slepian comparison principle one has
\begin{eqnarray}\label{eq:betainfsplus2}
\mE_G \frac{\max_{\x^{(i)}\in \calX} \lp \|
 G\x^{(i)}\|_2\rp}{\sqrt{n}} & = & \mE_G \frac{\max_{\x^{(i)}\in \calX,\y\in S^{m-1}}  \y^T G\x^{(i)}}{\sqrt{n}}\nonumber \\
 & \leq & \mE_{\u^{(2)},\h} \frac{\max_{\x^{(i)}\in \calX,\y\in S^{m-1}} \lp \y^T\u^{(2)} +\h^T\x^{(i)}\rp}{\sqrt{n}}\nonumber \\
 & = &
\mE_{\u^{(2)},\h} \frac{\max_{\x^{(i)}\in \calX} \lp \|\u^{(2)}\|_2 +\h^T\x^{(i)}\rp}{\sqrt{n}}.\nonumber \\
\end{eqnarray}
Of course, this form is only a special case of a much stronger concept introduced in Theorem \ref{thm:thm1}.

\vspace{.1in}
\textbf{\underline{\emph{Numerical results}}}

We below show in Figure \ref{fig:genbetainfsplus1xnorm1psi} and Table \ref{tab:genbetainfsplus1xnorm1psi} simulated results. All parameters are the same as earlier, with the exception that now $\beta=10$, which in a way emulates $\beta\rightarrow\infty$.
As can be seen from both, Figure \ref{fig:genbetainfsplus1xnorm1psi} and Table \ref{tab:genbetainfsplus1xnorm1psi}, the agreement between all presented results is again overwhelming. Moreover, even rather small value of $\beta$, namely, $\beta=10$, is already a fairly good approximation of $\beta\rightarrow\infty$.


\begin{figure}[htb]
\begin{minipage}[b]{.5\linewidth}
\centering
\centerline{\epsfig{figure=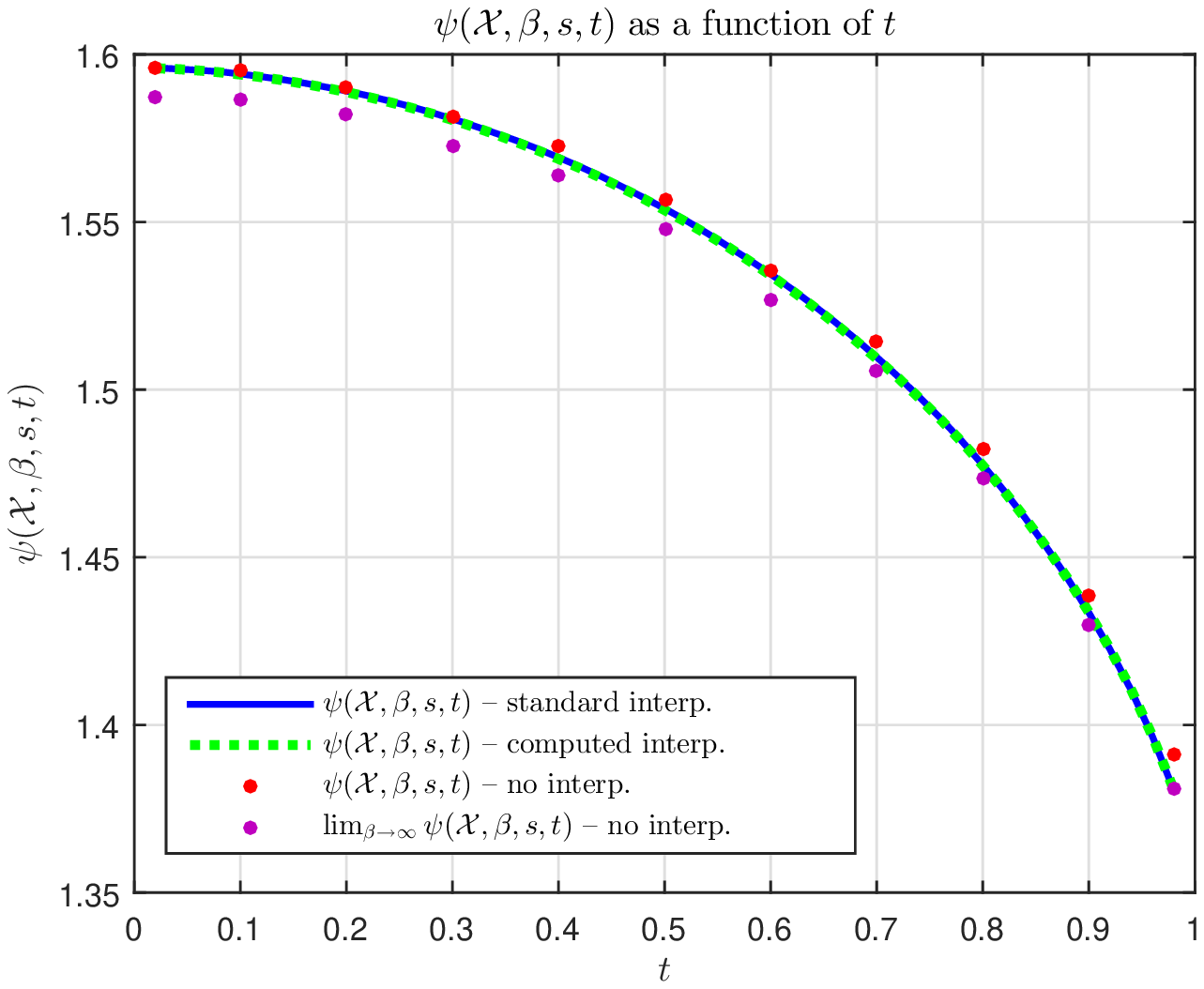,width=9cm,height=7cm}}
\end{minipage}
\begin{minipage}[b]{.5\linewidth}
\centering
\centerline{\epsfig{figure=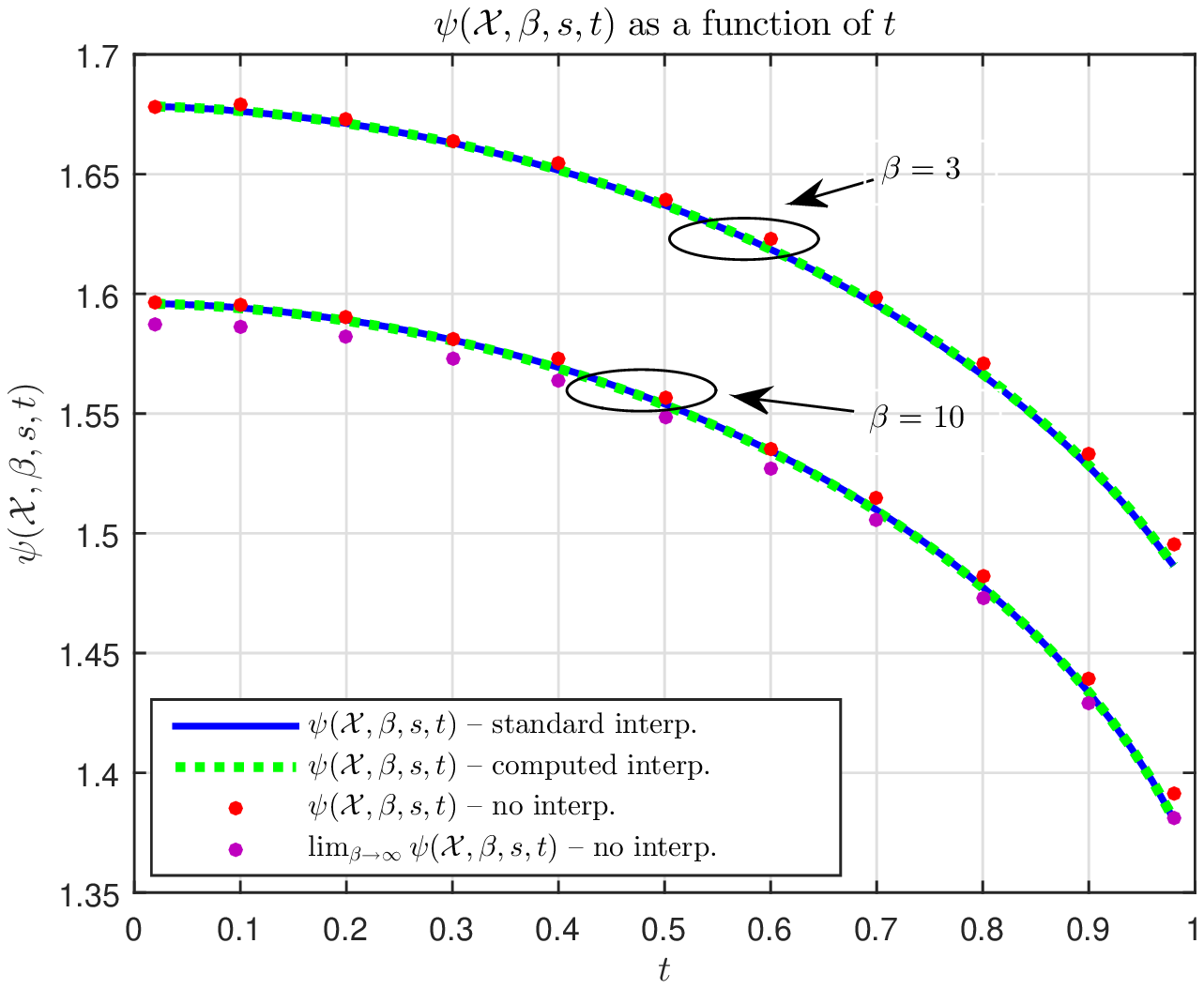,width=9cm,height=7cm}}
\end{minipage}
\caption{Left -- $\psi(\calX,\beta,s,t)$ as a function of $t$; $m=5$, $n=5$, $l=10$, $\calX=\calX^{+}$, $\beta=10$, $s=1$; right -- comparison between $\beta=3$ and $\beta=10$}
\label{fig:genbetainfsplus1xnorm1psi}
\end{figure}

\begin{table}[h]
\caption{Simulated results --- $m=5$, $n=5$, $l=10$, $\calX=\calX^{+}$, $\beta=10$, $s=1$}\vspace{.1in}
\hspace{-0in}\centering
\begin{tabular}{||c||c|c|c|c|c|c||}\hline\hline
$ t$  &  $\frac{d\psi}{dt}$; (\ref{eq:genanal11}) & $\frac{d\psi}{dt}$;  (\ref{eq:conalt5}) & $\psi$;  (\ref{eq:genanal11}) and (\ref{eq:co1eq1}) & $\psi$; (\ref{eq:conalt5}) and (\ref{eq:co1eq1}) & $\psi$;  (\ref{eq:genanal8})& $\lim_{\beta\rightarrow\infty}\psi$;  (\ref{eq:genanal8})\\  \hline\hline
$ 0.1000 $ & $ -0.0334 $ & $ -0.0332 $ & $\bl{\mathbf{ 1.5942 }}$ & $\bl{\mathbf{ 1.5940 }}$ & $\mathbf{ 1.5951 }$& $\prp{\mathbf{ 1.5866 }}$\\ \hline
$ 0.2000 $ & $ -0.0616 $ & $ -0.0648 $ & $\bl{\mathbf{ 1.5891 }}$ & $\bl{\mathbf{ 1.5888 }}$ & $\mathbf{ 1.5904 }$& $\prp{\mathbf{ 1.5819 }}$\\ \hline
$ 0.3000 $ & $ -0.0953 $ & $ -0.0960 $ & $\bl{\mathbf{ 1.5808 }}$ & $\bl{\mathbf{ 1.5805 }}$ & $\mathbf{ 1.5814 }$& $\prp{\mathbf{ 1.5729 }}$\\ \hline
$ 0.4000 $ & $ -0.1293 $ & $ -0.1296 $ & $\bl{\mathbf{ 1.5693 }}$ & $\bl{\mathbf{ 1.5688 }}$ & $\mathbf{ 1.5727 }$& $\prp{\mathbf{ 1.5642 }}$\\ \hline
$ 0.5000 $ & $ -0.1666 $ & $ -0.1692 $ & $\bl{\mathbf{ 1.5540 }}$ & $\bl{\mathbf{ 1.5536 }}$ & $\mathbf{ 1.5565 }$& $\prp{\mathbf{ 1.5481 }}$\\ \hline
$ 0.6000 $ & $ -0.2115 $ & $ -0.2144 $ & $\bl{\mathbf{ 1.5345 }}$ & $\bl{\mathbf{ 1.5341 }}$ & $\mathbf{ 1.5356 }$& $\prp{\mathbf{ 1.5271 }}$\\ \hline
$ 0.7000 $ & $ -0.2733 $ & $ -0.2716 $ & $\bl{\mathbf{ 1.5097 }}$ & $\bl{\mathbf{ 1.5093 }}$ & $\mathbf{ 1.5144 }$& $\prp{\mathbf{ 1.5058 }}$\\ \hline
$ 0.8000 $ & $ -0.3604 $ & $ -0.3558 $ & $\bl{\mathbf{ 1.4775 }}$ & $\bl{\mathbf{ 1.4774 }}$ & $\mathbf{ 1.4820 }$& $\prp{\mathbf{ 1.4732 }}$\\ \hline
$ 0.9000 $ & $ -0.5081 $ & $ -0.5016 $ & $\bl{\mathbf{ 1.4335 }}$ & $\bl{\mathbf{ 1.4336 }}$ & $\mathbf{ 1.4388 }$& $\prp{\mathbf{ 1.4296 }}$
\\ \hline\hline
\end{tabular}
\label{tab:genbetainfsplus1xnorm1psi}
\end{table}

\subsubsection{$s=-1$ -- a Gordon's comparison principle}
\label{sec:betainfsminus1}

When $s=-1$ we have
\begin{eqnarray}\label{eq:betainfsminus1}
\lim_{\beta\rightarrow\infty} \xi(\calX,\beta,-1)=\mE_G \frac{\max_{\x^{(i)}\in \calX} \lp - \|
 G\x^{(i)}\|_2\rp}{\sqrt{n}}= -\mE_G \frac{\min_{\x^{(i)}\in \calX} \lp  \|
 G\x^{(i)}\|_2\rp}{\sqrt{n}}.
\end{eqnarray}
Relying again on $\xi(\calX,\beta,1)=\psi(\calX,\beta,1,1)$ and the above machinery we have
\begin{eqnarray}\label{eq:betainfsminus2}
-\mE_G \frac{\min_{\x^{(i)}\in \calX} \lp \|
 G\x^{(i)}\|_2\rp}{\sqrt{n}} & = & \lim_{\beta\rightarrow\infty} \xi(\calX,\beta,1)=
 \lim_{\beta\rightarrow\infty} \psi(\calX,\beta,1,1) \nonumber \\
& \leq &  \lim_{\beta\rightarrow\infty} \psi(\calX,\beta,1,0)=
\mE_{\u^{(2)},\h} \frac{\max_{\x^{(i)}\in \calX} \lp -\|\u^{(2)}\|_2 +\h^T\x^{(i)}\rp}{\sqrt{n}} \nonumber \\
& = &
- \mE_{\u^{(2)},\h} \frac{\min_{\x^{(i)}\in \calX} \lp \|\u^{(2)}\|_2 -\h^T\x^{(i)}\rp}{\sqrt{n}}.
\end{eqnarray}
(\ref{eq:betainfsminus2}) is of course a form of the well-known Gordon comparison principle \cite{Gordon85}. Gordon's principle is an upgrade on the above Slepian's comparison principle, and assumes the following
\begin{eqnarray}\label{eq:betainfsminus3}
\mE_G \frac{\min_{\x^{(i)}\in \calX} \lp \|
 G\x^{(i)}\|_2\rp}{\sqrt{n}} & = & \mE_G \frac{\min_{\x^{(i)}\in \calX}\max_{\y\in S^{m-1}}  \y^T G\x^{(i)}}{\sqrt{n}}\nonumber \\
 & \geq & \mE_{\u^{(2)},\h} \frac{\min_{\x^{(i)}\in \calX}\max_{\y\in S^{m-1}} \lp \y^T\u^{(2)} +\h^T\x^{(i)}\rp}{\sqrt{n}} \nonumber \\
 & = &
\mE_{\u^{(2)},\h} \frac{\min_{\x^{(i)}\in \calX} \lp \|\u^{(2)}\|_2 +\h^T\x^{(i)}\rp}{\sqrt{n}}.\nonumber \\
\end{eqnarray}
Clearly, connecting beginning and end in both, (\ref{eq:betainfsminus2}) and (\ref{eq:betainfsminus3}) we obtain the same inequalities ($-\h$ and $\h$ have the same distribution). As was the case above with the Slepian's form, forms,  (\ref{eq:betainfsminus2}) and (\ref{eq:betainfsminus3}) are again only a special case of a much stronger concept introduced in Theorem \ref{thm:thm1}.


\textbf{\underline{\emph{Numerical results}}}

In Figure \ref{fig:genbetainfsmin1xnorm1psi} and Table \ref{tab:genbetainfsmin1xnorm1psi} simulated results are shown. All parameters are again the same as earlier, with the exception that now $\beta=10$ which again in a way emulates $\beta\rightarrow\infty$.
Both, Figure \ref{fig:genbetainfsmin1xnorm1psi} and Table \ref{tab:genbetainfsmin1xnorm1psi}, again demonstrate a solid agreement between all the presented results with $\beta=10$ being a pretty good approximation of $\beta\rightarrow\infty$.


\begin{figure}[htb]
\begin{minipage}[b]{.5\linewidth}
\centering
\centerline{\epsfig{figure=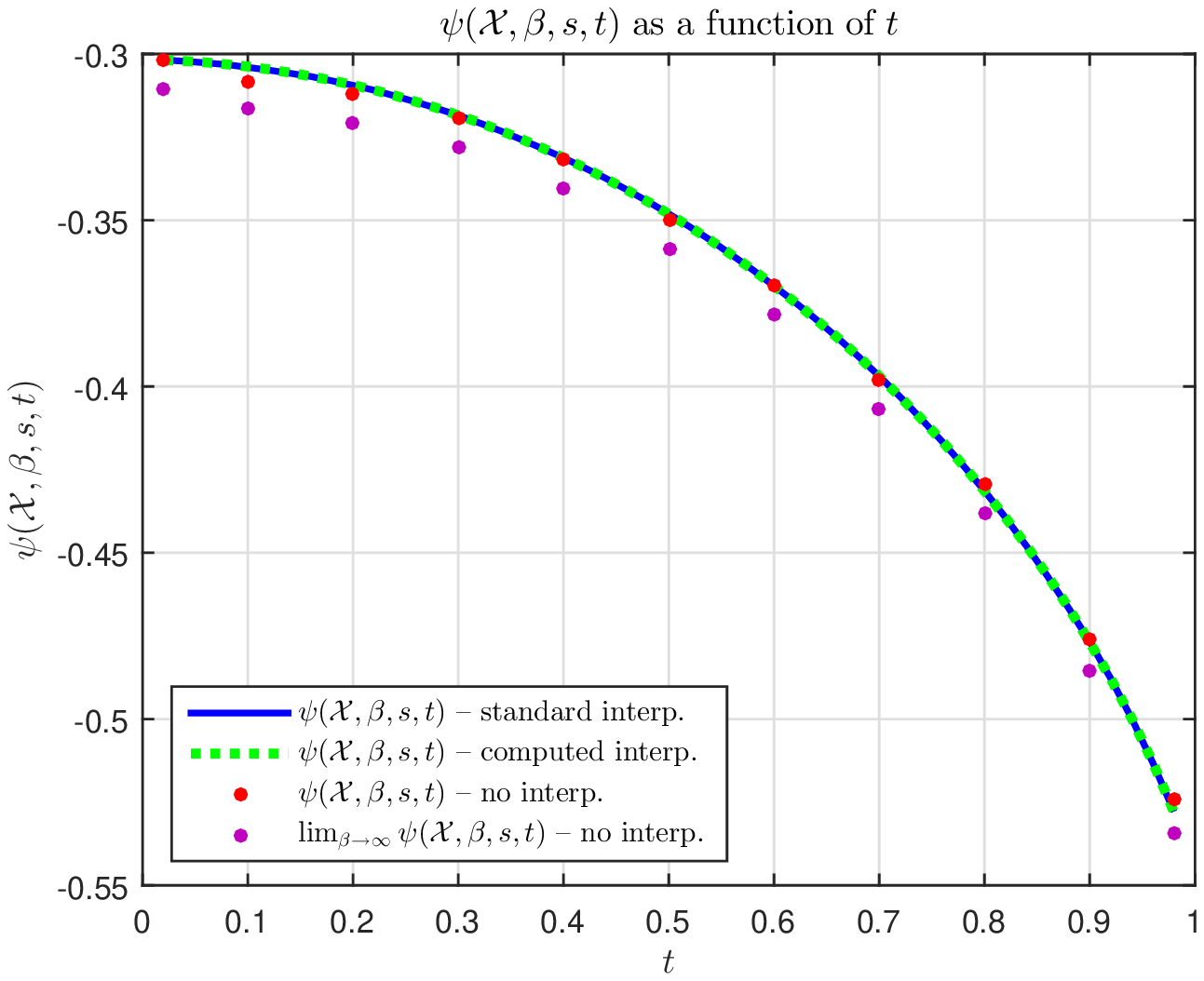,width=9cm,height=7cm}}
\end{minipage}
\begin{minipage}[b]{.5\linewidth}
\centering
\centerline{\epsfig{figure=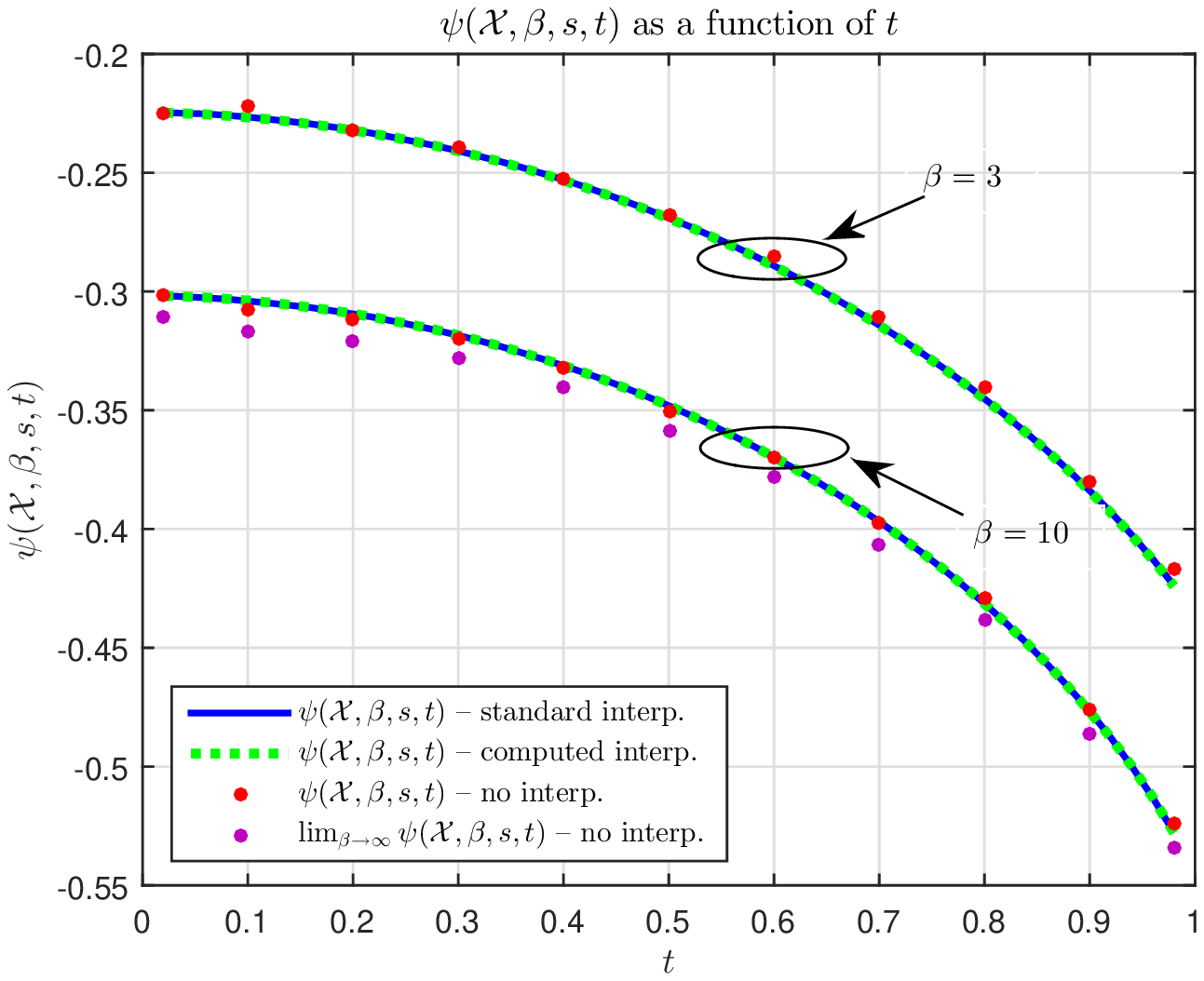,width=9cm,height=7cm}}
\end{minipage}
\caption{Left -- $\psi(\calX,\beta,s,t)$ as a function of $t$; $m=5$, $n=5$, $l=10$, $\calX=\calX^{+}$, $\beta=10$, $s=-1$; right -- comparison between $\beta=3$ and $\beta=10$}
\label{fig:genbetainfsmin1xnorm1psi}
\end{figure}

\begin{table}[h]
\caption{Simulated results --- $m=5$, $n=5$, $l=10$, $\calX=\calX^{+}$, $\beta=10$, $s=-1$}\vspace{.1in}
\hspace{-0in}\centering
\begin{tabular}{||c||c|c|c|c|c|c||}\hline\hline
$ t$  &  $\frac{d\psi}{dt}$; (\ref{eq:genanal11}) & $\frac{d\psi}{dt}$;  (\ref{eq:conalt5}) & $\psi$;  (\ref{eq:genanal11}) and (\ref{eq:co1eq1}) & $\psi$; (\ref{eq:conalt5}) and (\ref{eq:co1eq1}) & $\psi$;  (\ref{eq:genanal8})& $\lim_{\beta\rightarrow\infty}\psi$;  (\ref{eq:genanal8})\\  \hline\hline
$ 0.1000 $ & $ -0.0398 $ & $ -0.0349 $ & $\bl{\mathbf{ -0.3040 }}$ & $\bl{\mathbf{ -0.3038 }}$ & $\mathbf{ -0.3080 }$& $\prp{\mathbf{ -0.3167 }}$ \\ \hline
$ 0.2000 $ & $ -0.0670 $ & $ -0.0693 $ & $\bl{\mathbf{ -0.3094 }}$ & $\bl{\mathbf{ -0.3093 }}$ & $\mathbf{ -0.3122 }$& $\prp{\mathbf{ -0.3209 }}$ \\ \hline
$ 0.3000 $ & $ -0.1047 $ & $ -0.1059 $ & $\bl{\mathbf{ -0.3185 }}$ & $\bl{\mathbf{ -0.3184 }}$ & $\mathbf{ -0.3195 }$& $\prp{\mathbf{ -0.3280 }}$ \\ \hline
$ 0.4000 $ & $ -0.1411 $ & $ -0.1444 $ & $\bl{\mathbf{ -0.3313 }}$ & $\bl{\mathbf{ -0.3312 }}$ & $\mathbf{ -0.3320 }$& $\prp{\mathbf{ -0.3405 }}$ \\ \hline
$ 0.5000 $ & $ -0.1866 $ & $ -0.1848 $ & $\bl{\mathbf{ -0.3481 }}$ & $\bl{\mathbf{ -0.3481 }}$ & $\mathbf{ -0.3500 }$& $\prp{\mathbf{ -0.3585 }}$ \\ \hline
$ 0.6000 $ & $ -0.2382 $ & $ -0.2366 $ & $\bl{\mathbf{ -0.3698 }}$ & $\bl{\mathbf{ -0.3697 }}$ & $\mathbf{ -0.3695 }$& $\prp{\mathbf{ -0.3780 }}$ \\ \hline
$ 0.7000 $ & $ -0.2995 $ & $ -0.2996 $ & $\bl{\mathbf{ -0.3969 }}$ & $\bl{\mathbf{ -0.3969 }}$ & $\mathbf{ -0.3978 }$& $\prp{\mathbf{ -0.4066 }}$ \\ \hline
$ 0.8000 $ & $ -0.3780 $ & $ -0.3806 $ & $\bl{\mathbf{ -0.4314 }}$ & $\bl{\mathbf{ -0.4314 }}$ & $\mathbf{ -0.4292 }$& $\prp{\mathbf{ -0.4381 }}$ \\ \hline
$ 0.9000 $ & $ -0.5042 $ & $ -0.5164 $ & $\bl{\mathbf{ -0.4767 }}$ & $\bl{\mathbf{ -0.4770 }}$ & $\mathbf{ -0.4762 }$& $\prp{\mathbf{ -0.4858 }}$ \\ \hline\hline
\end{tabular}
\label{tab:genbetainfsmin1xnorm1psi}
\end{table}

\subsection{Moving away from the unit sphere}
\label{sec:moving}

What we presented so far assumed a fairly general sets $\calX$. The only restriction was that their elements should have unit norm. Now, we will remove such a restriction and show how the results that we obtained earlier can be modified to account for fully general sets $\calX$. Let $\calX=\{\x^{(1)},\x^{(2)},\dots,\x^{(l)}\}$, where $\x^{(i)}\in \mR^n,1\leq i\leq l$, be a given set. We will be interested in the following function
\begin{eqnarray}\label{eq:Mgenanal1}
 f_{M}(G,u^{(4)},\calX,\beta,s)= \frac{1}{\beta\sqrt{n}} \log\lp \sum_{i=1}^{l}e^{\beta\lp s\|
 G\x^{(i)}\|_2+\|\x^{(i)}\|_2u^{(4)}\rp} \rp,
\end{eqnarray}
where all objects are as earlier, and $u^{(4)}$ is a standard normal random variable independent of $G$. Given that we continue to consider random medium where the expected values play the crucial role, we introduce the following non-spherical analogues to $\xi(\calX,\beta,s)$
\begin{eqnarray}\label{eq:Mgenanal2}
\xi_M(\calX,\beta,s)\triangleq \mE_{G,u^{(4)}} f_M(G,u^{(4)},\calX,\beta,s)= \mE_{G,u^{(4)}}\frac{1}{\beta\sqrt{n}} \log\lp \sum_{i=1}^{l}e^{\beta\lp s\|
 G\x^{(i)}\|_2+\|\x^{(i)}\|_2u^{(4)}\rp} \rp,
\end{eqnarray}
and to $\psi(\calX,\beta,s,t)$
\begin{eqnarray}\label{eq:Mgenanal3}
\psi_M(\calX,\beta,s,t)= \mE_{G,u^{(4)},\u^{(2)},\h} \frac{1}{\beta\sqrt{n}} \log\lp \sum_{i=1}^{l}e^{\beta\lp s\|\sqrt{t}
 G\x^{(i)}+\sqrt{1-t}\|\x^{(i)}\|_2\u^{(2)}\|_2+\sqrt{t}\|\x^{(i)}\|_2 u^{(4)}+\sqrt{1-t}\h^T\x^{(i)}\rp} \rp,
\end{eqnarray}
Clearly,$\xi_M(\calX,\beta,s)=\psi_M(\calX,\beta,s,1)$. As earlier, $\psi_M(\calX,\beta,s,0)$ is an object much easier to handle than $\psi_M(\calX,\beta,s,1)$. Below, we will follow the strategy from earlier sections and will try to connect $\psi_M(\calX,\beta,s,1)$ to $\psi_M(\calX,\beta,s,0)$ which will then automatically connect $\xi_M(\calX,\beta,s)$ to $\psi_M(\calX,\beta,s,0)$. We can also rewrite (\ref{eq:Mgenanal3}) in the following way
\begin{equation}\label{eq:Mgenanal6}
\psi_M(\calX,\beta,s,t)
 =  \mE_{\u^{(i,1)},\u^{(2)},\u^{(i,3)},u^{(4)}} \frac{1}{\beta\sqrt{n}} \log\lp \sum_{i=1}^{l}e^{\beta_i\lp s\sqrt{\sum_{j=1}^{m}\lp\sqrt{t}\u_j^{(i,1)}+\sqrt{1-t}\u_j^{(2)}\rp^2}+\sqrt{t} u^{(4)}+\sqrt{1-t}\u^{(i,3)}\rp} \rp.
\end{equation}
where $\u_j^{(i,1)}$ and $\u_j^{(i,3)}$ are as earlier (in a sense that they are products of vectors of i.i.d. standard normals and unit norm vectors $\frac{\x^{(i)}}{\|\x^{(i)}\|_2}$) and $\beta_i=\beta\|\x^{(i)}\|_2$. We recall on the definition of $B^{(i)}$ from (\ref{eq:genanal7}) and set
\begin{eqnarray}\label{eq:Mgenanal7}
B^{(i)} & \triangleq &  \sqrt{\sum_{j=1}^{m}\lp\sqrt{t}\u_j^{(i,1)}+\sqrt{1-t}\u_j^{(2)}\rp^2} \nonumber \\
A^{(i)}_M & \triangleq &  e^{\beta_i(sB^{(i)}+\sqrt{t}u^{(4)}+\sqrt{1-t}\u^{(i,3)})}\nonumber \\
Z_M & \triangleq & \sum_{i=1}^{l}e^{\beta_i\lp s\sqrt{\sum_{j=1}^{m}\lp\sqrt{t}\u_j^{(i,1)}+\sqrt{1-t}\u_j^{(2)}\rp^2}+\sqrt{t}u^{(4)}+\sqrt{1-t}\u^{(i,3)}\rp}= \sum_{i=1}^{l} A^{(i)}_M.
\end{eqnarray}
From (\ref{eq:Mgenanal6}) and (\ref{eq:Mgenanal7}) we obviously have
\begin{eqnarray}\label{eq:Mgenanal8}
\psi_M(\calX,\beta,s,t) & = &  \mE_{\u^{(i,1)},\u^{(2)},\u^{(i,3)},u^{(4)}} \frac{1}{\beta\sqrt{n}} \log(Z_M).
\end{eqnarray}
and analogously to (\ref{eq:genanal11})
\begin{multline}\label{eq:Mgenanal11}
\frac{d\psi_M(\calX,\beta,s,t)}{dt}
 =   \mE_{\u^{(i,1)},\u^{(2)},\u^{(i,3)},u^{(4)}} \frac{s}{2\sqrt{n}}\sum_{j=1}^{m} \sum_{i=1}^{l} \frac{ \|\x^{(i)}\|_2A^{(i)}_M\lp(\u_j^{(i,1)})^2-(\u_j^{(2)})^2+\u_j^{(i,1)}\u_j^{(2)}\lp\frac{\sqrt{1-t}}{\sqrt{t}}-\frac{\sqrt{t}}{\sqrt{1-t}}\rp \rp }{ZB^{(i)}}\\
 + \mE_{\u^{(i,1)},\u^{(2)},\u^{(i,3)},u^{(4)}} \frac{1}{2\sqrt{n}}  \sum_{i=1}^{l}
\frac{A^{(i)}_M\|\x^{(i)}\|_2 u^{(4)}}{Z_M\sqrt{t}}- \mE_{\u^{(i,1)},\u^{(2)},\u^{(i,3)},u^{(4)}} \frac{1}{2\sqrt{n}}  \sum_{i=1}^{l}
\frac{A^{(i)}_M\|\x^{(i)}\|_2\u^{(i,3)}}{Z_M\sqrt{1-t}}.
\end{multline}

\subsubsection{Handling the derivatives}
\label{sec:movingder}

Below we show how already computed quantities can be used to characterize each of the terms appearing in the above sums.

\textbf{\underline{\emph{1) Computing $\mE_{\u^{(i,1)},\u^{(2)},\u^{(i,3)},u^{(4)}}  \frac{ A^{(i)}_M\u_j^{(i,1)}\u_j^{(2)}}{Z_MB^{(i)}}$}}}

\noindent As in Section \ref{sec:gencon} we present two different characterizations.

\textbf{\underline{\emph{1.1) Fixing $\u^{(i,1)}$}}}

\noindent Similarly to what we had in (\ref{eq:genanal13}) we now have
\begin{eqnarray}\label{eq:Mgenanal13}
\mE_{\u^{(i,1)},\u^{(2)},\u^{(i,3)},u^{(4)}}  \frac{ A^{(i)}_M\u_j^{(i,1)}\u_j^{(2)}}{Z_MB^{(i)}} & = &\mE\lp
\sum_{p=1,p\neq i}^{l} \frac{(\x^{(i)})^T\x^{(p)}}{\|\x^{(i)}\|_2\|\x^{(p)}\|_2}\frac{d}{d\u_j^{(p,1)}}\lp \frac{ A^{(i)}_M\u_j^{(2)}}{Z_MB^{(i)}}\rp \rp\nonumber \\
& & +\mE\lp\frac{(\x^{(i)})^T\x^{(i)}}{\|\x^{(i)}\|_2\|\x^{(i)}\|_2}\frac{d}{d\u_j^{(i,1)}}\lp \frac{ A^{(i)}_M\u_j^{(2)}}{Z_MB^{(i)}}\rp\rp.
\end{eqnarray}
Following closely the derivations from Section \ref{sec:hand1} one can observe that every single step can be repeated and the only difference will be that $\beta$ changes to $\beta_i$ and now $E(\u_j^{(i,1)}\u_j^{(p,1)})=\frac{(\x^{(i)})^T\x^{(p)}}{\|\x^{(i)}\|_2\|\x^{(p)}\|_2}$ (also, we note right here that $E(\u_j^{(i,3)}\u_j^{(p,3)})=\frac{(\x^{(i)})^T\x^{(p)}}{\|\x^{(i)}\|_2\|\x^{(p)}\|_2}$ as well). Following  (\ref{eq:genanal19}) we then have
\begin{eqnarray}\label{eq:Mgenanal19}
\mE_{\u^{(i,1)},\u^{(2)},\u^{(i,3)},u^{(4)}}  \frac{ A^{(i)}_M\u_j^{(i,1)}\u_j^{(2)}}{Z_MB^{(i)}}
& = &
\mE\lp\frac{ A^{(i)}_M}{Z_MB^{(i)}}\lp \lp \beta_i s -\frac{1}{B^{(i)}}\rp\frac{\u_j^{(2)}\sqrt{1-t}+\sqrt{t}\u_j^{(i,1)}}{B^{(i)}}\u^{(2)}\sqrt{t}\rp\rp \nonumber  \\
& & -\mE\sum_{p=1}^{l} \frac{(\x^{(i)})^T\x^{(p)}}{\|\x^{(i)}\|_2\|\x^{(p)}\|_2} \frac{\beta_p s  A^{(i)}_MA^{(p)}_M\u_j^{(2)}\sqrt{t}(\u_j^{(p,1)}\sqrt{t}+\sqrt{1-t}\u_j^{(2)})}{Z_MB^{(i)}B^{(p)}}.\nonumber \\
\end{eqnarray}

\textbf{\underline{\emph{1.2) Fixing $\u^{(2)}$}}}

\noindent Similarly to what we had in (\ref{eq:genCanal1}) we now have
\begin{equation}\label{eq:MgenCanal1}
\mE_{\u^{(i,1)},\u^{(2)},\u^{(i,3)},u^{(4)}}  \frac{ A^{(i)}_M\u_j^{(2)}\u^{(i,1)}}{Z_MB^{(i)}} =
\mE\lp\mE (\u_j^{(2)}\u_j^{(2)})\frac{d}{d\u_j^{(2)}}\lp \frac{ A^{(i)}_M\u^{(i,1)}}{Z_MB^{(i)}}\rp\rp
=\mE\lp\u^{(i,1)}\frac{d}{d\u_j^{(2)}}\lp \frac{ A^{(i)}_M}{Z_MB^{(i)}}\rp\rp,
\end{equation}
and utilizing the above observations from (\ref{eq:genCanal5}) we have
\begin{eqnarray}\label{eq:MgenCanal5}
\mE_{\u^{(i,1)},\u^{(2)},\u^{(i,3)},u^{(4)}}  \frac{ A^{(i)}_M\u_j^{(i,1)}\u_j^{(2)}}{Z_MB^{(i)}}
 & = &
\mE\lp\frac{ A^{(i)}_M}{Z_MB^{(i)}}\lp \lp \beta_i s -\frac{1}{B^{(i)}}\rp\frac{\u_j^{(2)}\sqrt{1-t}+\sqrt{t}\u_j^{(i,1)}}{B^{(i)}}\u_j^{(i,1)}\sqrt{1-t}\rp\rp \nonumber \\
& &  -\mE\sum_{p=1}^{l}\frac{\beta_p s  A^{(i)}_MA^{(p)}_M\u_j^{(i,1)}\sqrt{1-t}}{Z_MB^{(i)}B^{(p)}} (\u_j^{(2)}\sqrt{1-t}+\sqrt{t}\u_j^{(p,1)}).\nonumber \\
\end{eqnarray}

\textbf{\underline{\emph{2) Computing $\mE_{\u^{(i,1)},\u^{(2)},\u^{(i,3)},u^{(4)}}  \frac{ A^{(i)}_M(\u_j^{(2)})^2}{Z_MB^{(i)}}$}}}

\noindent Following (\ref{eq:gen1anal1}) we have
\begin{equation}\label{eq:Mgen1anal1}
\mE_{\u^{(i,1)},\u^{(2)},\u^{(i,3)},u^{(4)}}  \frac{ A^{(i)}_M(\u_j^{(2)})^2}{Z_MB^{(i)}} =
\mE\lp\mE (\u_j^{(2)}\u_j^{(2)})\frac{d}{d\u_j^{(2)}}\lp \frac{ A^{(i)}_M\u_j^{(2)}}{Z_MB^{(i)}}\rp\rp=\mE\lp\frac{ A^{(i)}_M}{Z_MB^{(i)}}+\u_j^{(2)}\frac{d}{d\u_j^{(2)}}\lp \frac{ A^{(i)}_M}{Z_MB^{(i)}}\rp\rp.
\end{equation}
Utilizing the above reasoning we then easily have
\begin{eqnarray}\label{eq:Mgen1anal5}
\mE_{\u^{(i,1)},\u^{(2)},\u^{(i,3)},u^{(4)}}  \frac{ A^{(i)}_M(\u_j^{(2)})^2}{Z_MB^{(i)}}
 & = &
\mE\lp\frac{ A^{(i)}_M}{Z_MB^{(i)}}\lp 1 +\lp \beta_i s -\frac{1}{B^{(i)}}\rp\frac{\u_j^{(2)}\sqrt{1-t}+\sqrt{t}\u_j^{(i,1)}}{B^{(i)}}\u_j^{(2)}\sqrt{1-t}\rp\rp \nonumber \\
& &  -\mE\sum_{p=1}^{l}\frac{\beta_p s  A^{(i)}_MA^{(p)}_M\u_j^{(2)}\sqrt{1-t}}{Z_MB^{(i)}B^{(p)}} (\u_j^{(2)}\sqrt{1-t}+\sqrt{t}\u_j^{(p,1)}).\nonumber \\
\end{eqnarray}

\textbf{\underline{\emph{3) Computing $\mE_{\u^{(i,1)},\u^{(2)},\u^{(i,3)},u^{(4)}}  \frac{ A^{(i)}_M(\u_j^{(i,1)})^2}{Z_MB^{(i)}}$}}}

\noindent Following (\ref{eq:gen2anal12}) we have
\begin{eqnarray}\label{eq:Mgen2anal12}
\mE_{\u^{(i,1)},\u^{(2)},\u^{(i,3)},u^{(4)}}  \frac{ A^{(i)}_M\u_j^{(i,1)}\u_j^{(2)}}{Z_MB^{(i)}}  & = &
\mE\lp\sum_{p=1,p\neq i}^{l} \mE (\u_j^{(i,1)}\u_j^{(p,1)})\frac{d}{d\u_j^{(p,1)}}\lp \frac{ A^{(i)}_M\u_j^{(i,1)}}{Z_MB^{(i)}}\rp \rp \nonumber \\
& & +\mE\lp\mE (\u_j^{(i,1)}\u_j^{(i,1)})\frac{d}{d\u_j^{(i,1)}}\lp \frac{ A^{(i)}_M\u_j^{(i,1)}}{Z_MB^{(i)}}\rp\rp,
\end{eqnarray}
and through the above mentioned changes from (\ref{eq:gen2anal19})
\begin{eqnarray}\label{eq:Mgen2anal19}
\mE_{\u^{(i,1)},\u^{(2)},\u^{(i,3)},u^{(4)}}  \frac{ A^{(i)}_M(\u_j^{(i,1)})^2}{Z_MB^{(i)}}
& = &
\mE\lp\frac{ A^{(i)}_M}{Z_MB^{(i)}}\lp 1 +\lp \beta_i s -\frac{1}{B^{(i)}}\rp\frac{\u_j^{(2)}\sqrt{1-t}+\sqrt{t}\u_j^{(i,1)}}{B^{(i)}}\u_j^{(i,1)}\sqrt{t}\rp\rp \nonumber \\ & & -\mE\sum_{p=1}^{l} \frac{(\x^{(i)})^T\x^{(p)}}{\|\x^{(i)}\|_2\|\x^{(p)}\|_2} \frac{ \beta_p s A^{(i)}_MA^{(p)}_M\u_j^{(i,1)}\sqrt{t}(\u_j^{(p,1)}\sqrt{t}+\sqrt{1-t}\u_j^{(2)})}{Z_MB^{(i)}B^{(p)}}.\nonumber \\
\end{eqnarray}

\textbf{\underline{\emph{4) Computing $\mE_{\u^{(i,1)},\u^{(2)},\u^{(i,3)},u^{(4)}}  \frac{ A^{(i)}_M\u^{(i,3)}}{Z_M}$}}}

\noindent  From (\ref{eq:genDanal12}) we also have
\begin{equation}\label{eq:MgenDanal12}
\mE_{\u^{(i,1)},\u^{(2)},\u^{(i,3)},u^{(4)}}  \frac{ A^{(i)}_M\u^{(i,3)}}{Z_M} =
\mE\lp\sum_{p=1,p\neq i}^{l} \mE (\u^{(i,3)}\u^{(p,3)})\frac{d}{d\u^{(p,3)}}\lp \frac{ A^{(i)}_M}{Z_M}\rp
+\mE (\u^{(i,3)}\u^{(i,3)})\frac{d}{d\u^{(i,3)}}\lp \frac{ A^{(i)}_M}{Z_M}\rp\rp,
\end{equation}
and similarly as above from (\ref{eq:genDanal19})
\begin{equation}\label{eq:MgenDanal19}
\mE_{\u^{(i,1)},\u^{(2)},\u^{(i,3)},u^{(4)}}  \frac{ A^{(i)}_M\u^{(i,3)}}{Z_M}  =
 \mE\frac{\beta_i A^{(i)}_M\sqrt{1-t}}{Z_M} -\mE\sum_{p=1}^{l} \frac{(\x^{(i)})^T\x^{(p)}}{\|\x^{(i)}\|_2\|\x^{(p)}\|_2}\frac{\beta_p A^{(i)}_M  A^{(p)}_M \sqrt{1-t}}{Z_M}.
\end{equation}

\textbf{\underline{\emph{5) Computing $\mE_{\u^{(i,1)},\u^{(2)},\u^{(i,3)},u^{(4)}}  \frac{ A^{(i)}_M\u^{(4)}}{Z_M}$}}}

\noindent Utilizing integration by parts we obtain
\begin{equation}\label{eq:MgenEanal12}
\mE_{\u^{(i,1)},\u^{(2)},\u^{(i,3)},u^{(4)}}  \frac{ A^{(i)}_M\u^{(4)}}{Z_M} =
\mE\lp\mE (\u^{(4)}\u^{(4)})\frac{d}{d\u^{(4)}}\lp \frac{ A^{(i)}_M}{Z_M}\rp\rp
=\mE\lp\frac{d}{d\u^{(4)}}\lp \frac{ A^{(i)}_M}{Z_M}\rp\rp.
\end{equation}
We also have
\begin{eqnarray}\label{eq:MgenEanal18}
\frac{d}{d\u^{(4)}}\lp \frac{ A^{(i)}_M}{Z_M}\rp  & = &
\frac{1}{Z_M}\frac{d A^{(i)}_M}{d\u^{(4)}}-\frac{A^{(i)}_M}{Z^{2}_M}\frac{dZ_M}{d\u^{(4)}} \nonumber \\
& = &   \frac{\beta_i A^{(i)}_M\sqrt{t}}{Z_M}
-\frac{ A^{(i)}_M}{Z^{2}_M}\sum_{p=1}^{l}\frac{d A^{(i)}_M}{d\u^{(4)}}  = \frac{\beta_i A^{(i)}_M\sqrt{t}}{Z_M} -\sum_{p=1}^{l}\frac{ A^{(i)}_MA^{(p)}_M\beta_p\sqrt{t}}{Z^{2}_M}.
\end{eqnarray}
Combining (\ref{eq:MgenEanal12}) and (\ref{eq:MgenEanal18}) we have
\begin{equation}\label{eq:MgenEanal19}
\mE_{\u^{(i,1)},\u^{(2)},\u^{(i,3)},u^{(4)}}  \frac{ A^{(i)}_M\u^{(4)}}{Z_M}  =
\mE \frac{\beta_i A^{(i)}_M\sqrt{t}}{Z_M} -\mE\sum_{p=1}^{l}\frac{ A^{(i)}_MA^{(p)}_M\beta_p\sqrt{t}}{Z^{2}_M}.
\end{equation}

\subsubsection{Connecting everything together}
\label{sec:Mconalt}

Now, we can combine (\ref{eq:Mgenanal11}), (\ref{eq:Mgenanal19}), (\ref{eq:MgenCanal5}), (\ref{eq:Mgen1anal5}), (\ref{eq:Mgen2anal19}), (\ref{eq:MgenDanal19}), and (\ref{eq:MgenEanal19}) and follow the procedure from Section \ref{sec:conalt} to obtain similarly to (\ref{eq:conalt4})
\begin{eqnarray}\label{eq:Mconalt4}
\frac{d\psi_M(\calX,\beta,s,c_3,t)}{dt}
 & = & \frac{s}{2\sqrt{n}} \sum_{i=1}^{l}\mE\sum_{p=1}^{l}\sum_{j=1}^{m} (\|\x^{(i)}\|_2\|\x^{(p)}\|_2-(\x^{(i)})^T\x^{(p)})\nonumber \\
  & & \times \frac{ \beta s A^{(i)}_MA^{(p)}_M(\u_j^{(i,1)}\sqrt{t}+\u_j^{(2)}\sqrt{1-t})(\u_j^{(p,1)}\sqrt{t}+\sqrt{1-t}\u_j^{(2)})}{Z^{2}_MB^{(i)}B^{(p)}} \nonumber \\
& & +\frac{1}{2\sqrt{n}}   \mE\sum_{i=1}^{l}\lp \mE \frac{\beta_i A^{(i)}_M}{Z_M} -\mE\sum_{p=1}^{l}\frac{ A^{(i)}_MA^{(p)}_M\beta_p}{Z^{2}_M}\rp \nonumber \\
& & - \frac{1}{2\sqrt{n}}   \mE\sum_{i=1}^{l}\lp\mE\frac{\beta_i A^{(i)}_M}{Z_M} -\mE\sum_{p=1}^{l} \frac{(\x^{(i)})^T\x^{(p)}}{\|\x^{(i)}\|_2\|\x^{(p)}\|_2}\frac{\beta_p A^{(i)}_M  A^{(p)}_M }{Z^{2}_M}\rp.
\end{eqnarray}
The sum to the right of the equality sign in the first row is obtained in the same way as the corresponding one in (\ref{eq:conalt4}) with the above mentioned small adjustments. The second row follows from (\ref{eq:MgenEanal19}) and the third row follows from (\ref{eq:MgenDanal19}). From (\ref{eq:Mconalt4}) we further have
\begin{eqnarray}\label{eq:Mconalt5}
\frac{d\psi_M(\calX,\beta,s,t)}{dt}
 & = & \frac{1}{2\sqrt{n}} \sum_{i=1}^{l}\mE\sum_{p=1}^{l}\sum_{j=1}^{m} (\|\x^{(i)}\|_2\|\x^{(p)}\|_2-(\x^{(i)})^T\x^{(p)})\nonumber \\
  & & \times \frac{ \beta  A^{(i)}_MA^{(p)}_M(\u_j^{(i,1)}\sqrt{t}+\u_j^{(2)}\sqrt{1-t})(\u_j^{(p,1)}\sqrt{t}+\sqrt{1-t}\u_j^{(2)})}{Z^{2}_MB^{(i)}B^{(p)}} \nonumber \\
& & - \frac{1}{2} \mE \lp\sum_{i=1}^{l}\sum_{p=1}^{l} (\|\x^{(i)}\|_2\|\x^{(p)}\|_2-(\x^{(i)})^T\x^{(p)})\frac{\beta A^{(i)}_M  A^{(p)}_M }{Z^{2}_M}\rp.
\end{eqnarray}
(\ref{eq:Mconalt5}) then easily gives
\begin{eqnarray}\label{eq:Mconalt6}
\frac{d\psi_M(\calX,\beta,s,t)}{dt} & = &  -\frac{\beta }{2\sqrt{n}} \sum_{i=1}^{l}\mE\sum_{p=1}^{l}\frac{ A^{(i)}_MA^{(p)}_M}{Z^{(2)}_M}
 (\|\x^{(i)}\|_2\|\x^{(p)}\|_2-(\x^{(i)})^T\x^{(p)})
\nonumber \\
 & & \times \lp 1-
\sum_{j=1}^{m}\frac{(\u_j^{(i,1)}\sqrt{t}+\u_j^{(2)}\sqrt{1-t})(\u_j^{(p,1)}\sqrt{t}+\sqrt{1-t}\u_j^{(2)})}{B^{(i)}B^{(p)}}\rp.\nonumber\\
\end{eqnarray}
Finally we have
\begin{eqnarray}\label{eq:Mconalt7}
\frac{d\psi_M(\calX,\beta,s,t)}{dt} & = &  -\frac{\beta}{2\sqrt{n}} \sum_{i=1}^{l}\mE\sum_{p=1}^{l} \frac{ A^{(i)}_MA^{(p)}_M}{Z^{(2)}_M}
 (\|\x^{(i)}\|_2\|\x^{(p)}\|_2-(\x^{(i)})^T\x^{(p)})\nonumber \\
& & \times \lp1-
\frac{(\u^{(i,1)}\sqrt{t}+\u^{(2)}\sqrt{1-t})^T(\u^{(p,1)}\sqrt{t}+\sqrt{1-t}\u^{(2)})}{B^{(i)}B^{(p)}}\rp.\nonumber\\
\end{eqnarray}
We summarize the above results in the following theorem.
\begin{theorem}
\label{thm:thm1a}
  Assume the setup of Theorem \ref{thm:thm1} without the restriction that the elements of $\calX$ have the unit norm. Function $\psi_M(\calX,\beta,s,t)$
\begin{eqnarray}\label{eq:thm1aeq1}
\psi_M(\calX,\beta,s,t)= \mE_{G,u^{(4)},\u^{(2)},\h} \frac{1}{\beta\sqrt{n}} \log\lp \sum_{i=1}^{l}e^{\beta\lp s\|\sqrt{t}
 G\x^{(i)}+\sqrt{1-t}\u^{(2)}\|_2+\sqrt{t}\|\x^{(i)}\|_2u^{(4)}+\sqrt{1-t}\h^T\x^{(i)}\rp} \rp,
\end{eqnarray}
is decreasing in $t$.
\end{theorem}
\begin{proof}
  Follows from the above discussion.
\end{proof}
\begin{corollary}
  Assuming the setup of Theorem \ref{thm:thm1a}, we have
\begin{eqnarray}\label{eq:co1aeq1}
\psi_M(\calX,\beta,s,t)= \psi_M(\calX,\beta,s,0)+\int_{0}^{t}\frac{d\psi_M(\calX,\beta,s,t)}{dt}dt,
\end{eqnarray}
as well as the following comparison principle
\begin{eqnarray}\label{eq:co1aeq2}
\psi_M(\calX,\beta,s,0) \geq  \psi_M(\calX,\beta,s,t)\geq \psi_M(\calX,\beta,s,1).
\end{eqnarray}
\end{corollary}
\begin{proof}
  Follows trivially from the above theorem by noting that $\frac{d\psi_M(\calX,\beta,s,t)}{dt}\leq 0$.
\end{proof}

\subsubsection{Simulations}
\label{sec:Mgenconsim}

In this section we present a few additional simulation results. We again chose $m=5$, $n=5$, $l=10$, and selected set $\calX$ as the columns of the following matrix
\begin{equation}
X^{-}=\begin{bmatrix}
-0.3624 & -0.9364 & 1.1566 & -0.8076 & -1.1066 & 1.3148 & -0.3405 & -0.7938 & -3.0744 & 0.2493 \\
-0.6616 & -1.4250 & -1.4638 & -0.1997 & 0.1102 & 0.9261 & 1.2240 & -0.1874 & -0.4569 & -0.1518 \\
-0.4980 & -0.0708 & -0.7947 & -1.3493 & 0.3226 & -0.4982 & 1.0334 & -0.2817 & 0.3247 & -2.4773 \\
-1.4281 & -0.7722 & 0.9885 & -0.4056 & -0.2903 & -0.1814 & 1.4318 & 1.0533 & 1.3286 & -0.6086 \\
-0.7196 & 1.3075 & 1.0363 & -0.9904 & 0.4357 & -1.6953 & 0.2346 & -0.5735 & -1.0376 & 0.1766
\end{bmatrix}.
\end{equation}
In other words, we have
\begin{equation}\label{eq:sim1}
  \calX^{-}=\{X^{-}_{:,1},X^{-}_{:,2},\dots,X^{-}_{:,l}\}.
\end{equation}
As earlier, set $\calX^{-}$ (and matrix $X^{-}$) is chosen totally randomly, i.e. there is nothing specific about it (this time even the norm of its elements is not equal to one). For such a set we repeated simulation experiments from earlier section. Namely, derivatives $\frac{d\psi_M(\calX,\beta,s,t)}{dt}$ were simulated according to (\ref{eq:Mgenanal11}) and according to (\ref{eq:Mconalt5}) and with such derivatives $\psi_M(\calX,\beta,s,t)$ is computed according to (\ref{eq:co1aeq1}). We also again simulated $\psi_M(\calX,\beta,s,t)$ directly (i.e. without interpolating computations) through (\ref{eq:Mgenanal8}). Throughout all simulations, the number of repetitions is again kept at $5e4$. Moreover, we set $\beta=3$ and simulated two different scenarios with all other parameters being the same: 1) $s=1$ and 2) $s=-1$.

\textbf{\underline{\emph{1) $s=1$ -- numerical results}}}

The results that we obtained for $s=1$ are presented in Figure \ref{fig:Mgensplus1xdiffnormpsi} and Table \ref{tab:Mgensplus1xdiffnormpsi} in a fashion that fully parallels what we have done in earlier sections. Both, Figure \ref{fig:Mgensplus1xdiffnormpsi} and Table \ref{tab:Mgensplus1xdiffnormpsi}, show an overwhelming agreement between the presented results.

\begin{figure}[htb]
\centering
\centerline{\epsfig{figure=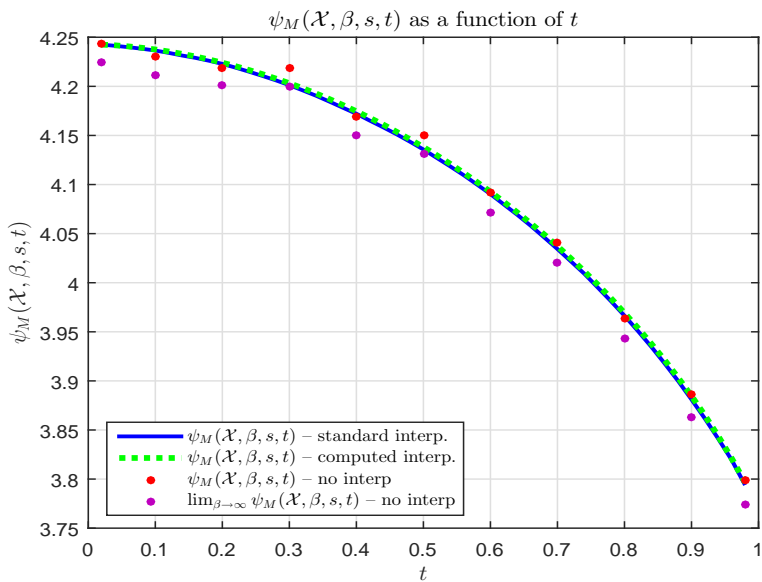,width=11.5cm,height=8cm}}
\caption{$\psi(\calX,\beta,s,t)$ as a function of $t$; $m=5$, $n=5$, $l=10$, $\calX=\calX^{-}$, $\beta=3$, $s=1$}
\label{fig:Mgensplus1xdiffnormpsi}
\end{figure}
\begin{table}[h]
\caption{Simulated results --- $m=5$, $n=5$, $l=10$, $\calX=\calX^{-}$, $\beta=3$, $s=1$}\vspace{.1in}
\hspace{-0in}\centering
\begin{tabular}{||c||c|c|c|c|c||}\hline\hline
$ t$  &  $\frac{d\psi}{dt}$; (\ref{eq:genanal11}) & $\frac{d\psi}{dt}$;  (\ref{eq:conalt5}) & $\psi$;  (\ref{eq:genanal11}) and (\ref{eq:co1eq1}) & $\psi$; (\ref{eq:conalt5}) and (\ref{eq:co1eq1}) & $\psi$;  (\ref{eq:genanal8})\\  \hline\hline
$ 0.1000 $ & $ -0.0868 $ & $ -0.0877 $ & $\bl{\mathbf{ 4.2365 }}$ & $\bl{\mathbf{ 4.2376 }}$ & $\mathbf{ 4.2299 }$  \\ \hline
$ 0.2000 $ & $ -0.1643 $ & $ -0.1662 $ & $\bl{\mathbf{ 4.2229 }}$ & $\bl{\mathbf{ 4.2241 }}$ & $\mathbf{ 4.2189 }$  \\ \hline
$ 0.3000 $ & $ -0.2493 $ & $ -0.2342 $ & $\bl{\mathbf{ 4.2012 }}$ & $\bl{\mathbf{ 4.2032 }}$ & $\mathbf{ 4.2180 }$  \\ \hline
$ 0.4000 $ & $ -0.3099 $ & $ -0.3181 $ & $\bl{\mathbf{ 4.1721 }}$ & $\bl{\mathbf{ 4.1746 }}$ & $\mathbf{ 4.1691 }$  \\ \hline
$ 0.5000 $ & $ -0.3966 $ & $ -0.3936 $ & $\bl{\mathbf{ 4.1357 }}$ & $\bl{\mathbf{ 4.1382 }}$ & $\mathbf{ 4.1502 }$  \\ \hline
$ 0.6000 $ & $ -0.4833 $ & $ -0.4948 $ & $\bl{\mathbf{ 4.0903 }}$ & $\bl{\mathbf{ 4.0926 }}$ & $\mathbf{ 4.0912 }$  \\ \hline
$ 0.7000 $ & $ -0.6006 $ & $ -0.5995 $ & $\bl{\mathbf{ 4.0343 }}$ & $\bl{\mathbf{ 4.0372 }}$ & $\mathbf{ 4.0402 }$  \\ \hline
$ 0.8000 $ & $ -0.7336 $ & $ -0.7447 $ & $\bl{\mathbf{ 3.9664 }}$ & $\bl{\mathbf{ 3.9691 }}$ & $\mathbf{ 3.9636 }$  \\ \hline
$ 0.9000 $ & $ -0.9417 $ & $ -0.9298 $ & $\bl{\mathbf{ 3.8811 }}$ & $\bl{\mathbf{ 3.8846 }}$ & $\mathbf{ 3.8858 }$
  \\ \hline \hline
\end{tabular}
\label{tab:Mgensplus1xdiffnormpsi}
\end{table}

\textbf{\underline{\emph{2) $s=-1$ -- numerical results}}}

The results that we obtained for $s=-1$ are presented in Figure \ref{fig:Mgensmin1xdiffnormpsi} and Table \ref{tab:Mgensmin1xdiffnormpsi}. As above for $s=1$, here we again have that both, Figure \ref{fig:Mgensmin1xdiffnormpsi} and Table \ref{tab:Mgensmin1xdiffnormpsi}, show an overwhelming agreement between the presented results.

\begin{figure}[htb]
\centering
\centerline{\epsfig{figure=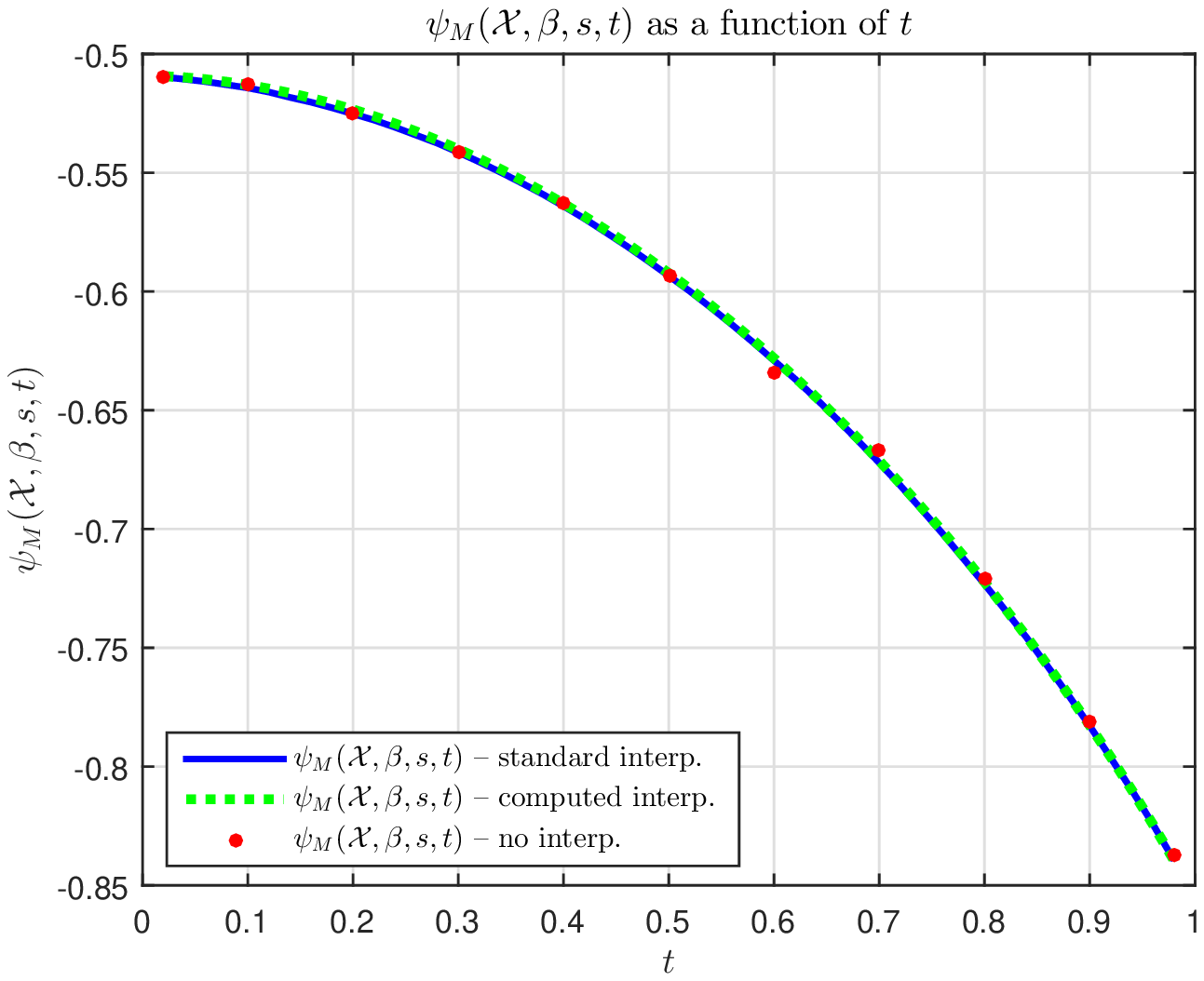,width=11.5cm,height=8cm}}
\caption{$\psi(\calX,\beta,s,t)$ as a function of $t$; $m=5$, $n=5$, $l=10$, $\calX=\calX^{-}$, $\beta=$, $s=-1$}
\label{fig:Mgensmin1xdiffnormpsi}
\end{figure}
\begin{table}[h]
\caption{Simulated results --- $m=5$, $n=5$, $l=10$, $\calX=\calX^{-}$, $\beta=3$, $s=-1$}\vspace{.1in}
\hspace{-0in}\centering
\begin{tabular}{||c||c|c|c|c|c||}\hline\hline
$ t$  &  $\frac{d\psi}{dt}$; (\ref{eq:genanal11}) & $\frac{d\psi}{dt}$;  (\ref{eq:conalt5}) & $\psi$;  (\ref{eq:genanal11}) and (\ref{eq:co1eq1}) & $\psi$; (\ref{eq:conalt5}) and (\ref{eq:co1eq1}) & $\psi$;  (\ref{eq:genanal8})\\  \hline\hline
$ 0.1000 $ & $ -0.0672 $ & $ -0.0643 $ & $\bl{\mathbf{ -0.5143 }}$ & $\bl{\mathbf{ -0.5130 }}$ & $\mathbf{ -0.5128 }$\\ \hline
$ 0.2000 $ & $ -0.1284 $ & $ -0.1278 $ & $\bl{\mathbf{ -0.5252 }}$ & $\bl{\mathbf{ -0.5234 }}$ & $\mathbf{ -0.5245 }$\\ \hline
$ 0.3000 $ & $ -0.1919 $ & $ -0.1926 $ & $\bl{\mathbf{ -0.5414 }}$ & $\bl{\mathbf{ -0.5401 }}$ & $\mathbf{ -0.5413 }$\\ \hline
$ 0.4000 $ & $ -0.2546 $ & $ -0.2550 $ & $\bl{\mathbf{ -0.5641 }}$ & $\bl{\mathbf{ -0.5630 }}$ & $\mathbf{ -0.5630 }$\\ \hline
$ 0.5000 $ & $ -0.3169 $ & $ -0.3186 $ & $\bl{\mathbf{ -0.5934 }}$ & $\bl{\mathbf{ -0.5923 }}$ & $\mathbf{ -0.5929 }$\\ \hline
$ 0.6000 $ & $ -0.3850 $ & $ -0.3855 $ & $\bl{\mathbf{ -0.6291 }}$ & $\bl{\mathbf{ -0.6281 }}$ & $\mathbf{ -0.6342 }$\\ \hline
$ 0.7000 $ & $ -0.4623 $ & $ -0.4621 $ & $\bl{\mathbf{ -0.6721 }}$ & $\bl{\mathbf{ -0.6713 }}$ & $\mathbf{ -0.6671 }$\\ \hline
$ 0.8000 $ & $ -0.5355 $ & $ -0.5447 $ & $\bl{\mathbf{ -0.7232 }}$ & $\bl{\mathbf{ -0.7223 }}$ & $\mathbf{ -0.7214 }$\\ \hline
$ 0.9000 $ & $ -0.6412 $ & $ -0.6439 $ & $\bl{\mathbf{ -0.7829 }}$ & $\bl{\mathbf{ -0.7827 }}$ & $\mathbf{ -0.7809 }$
\\ \hline \hline
\end{tabular}
\label{tab:Mgensmin1xdiffnormpsi}
\end{table}

\subsubsection{$\beta\rightarrow \infty$}
\label{sec:betainf}

It is often of particular interest to study the following limiting behavior of $\xi_M(\calX,\beta,s)$
\begin{eqnarray}\label{eq:betainf1}
\lim_{\beta\rightarrow\infty} \xi_M(\calX,\beta,s)& = & \lim_{\beta\rightarrow\infty} \mE_{G,u^{(4)}} \frac{1}{\beta\sqrt{n}} \log\lp \sum_{i=1}^{l}e^{\beta\lp s\|
 G\x^{(i)}\|_2+\|\x^{(i)}\|_2u^{(4)}\rp} \rp\nonumber \\
 & = & \mE_{G,u^{(4)}} \frac{\max_{\x^{(i)}\in \calX} \lp s\|
 G\x^{(i)}\|_2 +\|\x^{(i)}\|_2u^{(4)}\rp}{\sqrt{n}},
\end{eqnarray}
We again distinguish two scenarios.

\textbf{\underline{\emph{1) $s=1$ -- a Slepian's comparison principle}}}

\noindent When $s=1$ we have
\begin{eqnarray}\label{eq:Mbetainfsplus1}
\lim_{\beta\rightarrow\infty} \xi_M(\calX,\beta,1)=\mE_{G,u^{(4)}} \frac{\max_{\x^{(i)}\in \calX} \lp \|
 G\x^{(i)}\|_2+\|\x^{(i)}\|_2u^{(4)}\rp}{\sqrt{n}}.
\end{eqnarray}
Since $\xi_M(\calX,\beta,1)=\psi_M(\calX,\beta,1,1)$, the above machinery then gives
\begin{eqnarray}\label{eq:Mbetainfsplus2}
\mE_{G,u^{(4)}} \frac{\max_{\x^{(i)}\in \calX} \lp \|
 G\x^{(i)}\|_2+\|\x^{(i)}\|_2u^{(4)}\rp}{\sqrt{n}} & = & \lim_{\beta\rightarrow\infty} \xi_M(\calX,\beta,1)=
 \lim_{\beta\rightarrow\infty} \psi_M(\calX,\beta,1,1) \nonumber \\
& \leq &  \lim_{\beta\rightarrow\infty} \psi_M(\calX,\beta,1,0)\nonumber \\
& = &
\mE_{\u^{(2)},\h} \frac{\max_{\x^{(i)}\in \calX} \lp \|\x^{(i)}\|_2\|\u^{(2)}\|_2 +\h^T\x^{(i)}\rp}{\sqrt{n}}.
\end{eqnarray}
(\ref{eq:Mbetainfsplus2}) is again of course a form of the well-known Slepian comparison principle \cite{Slep62}. Indeed, based on the Slepian comparison principle one has
\begin{eqnarray}\label{eq:Mbetainfsplus2}
\mE_{G,u^{(4)}} \frac{\max_{\x^{(i)}\in \calX} \lp \|
 G\x^{(i)}\|_2+\|\x^{(i)}\|_2u^{(4)}\rp}{\sqrt{n}} & = & \mE_{G,u^{(4)}} \frac{\max_{\x^{(i)}\in \calX,\y\in S^{m-1}} \lp \y^T G\x^{(i)}+\|\x^{(i)}\|_2u^{(4)}\rp}{\sqrt{n}}\nonumber \\
 & \leq & \mE_{\u^{(2)},\h} \frac{\max_{\x^{(i)}\in \calX,\y\in S^{m-1}} \lp \|\x^{(i)}\|_2\y^T\u^{(2)} +\h^T\x^{(i)}\rp}{\sqrt{n}}\nonumber \\
 & = &
\mE_{\u^{(2)},\h} \frac{\max_{\x^{(i)}\in \calX} \lp \|\x^{(i)}\|_2\|\u^{(2)}\|_2 +\h^T\x^{(i)}\rp}{\sqrt{n}}.\nonumber \\
\end{eqnarray}
Of course, this form is only a special case of a much stronger concept introduced in Theorem \ref{thm:thm1a}.

\textbf{\underline{\emph{Numerical results}}}

We below show in Figure \ref{fig:Mgenbetainfsplus1xdiffnormpsi} and Table \ref{tab:Mgenbetainfsplus1xdiffnormpsi} simulated results. We again emulate what was done in earlier section and set all parameters to be the same as before. Along the same lines we again set $\beta=10$, as a way to approach closer $\beta\rightarrow\infty$ regime.
As can be seen from both, Figure \ref{fig:Mgenbetainfsplus1xdiffnormpsi} and Table \ref{tab:Mgenbetainfsplus1xdiffnormpsi}, the agreement between all presented results is again overwhelming. Moreover, $\beta=10$ doesn't bring much of a difference when compared to $\beta=3$ which hints that there are scenarios where $\beta$ as small as three is already a fairly good approximation of $\beta\rightarrow\infty$.


\begin{figure}[htb]
\begin{minipage}[b]{.5\linewidth}
\centering
\centerline{\epsfig{figure=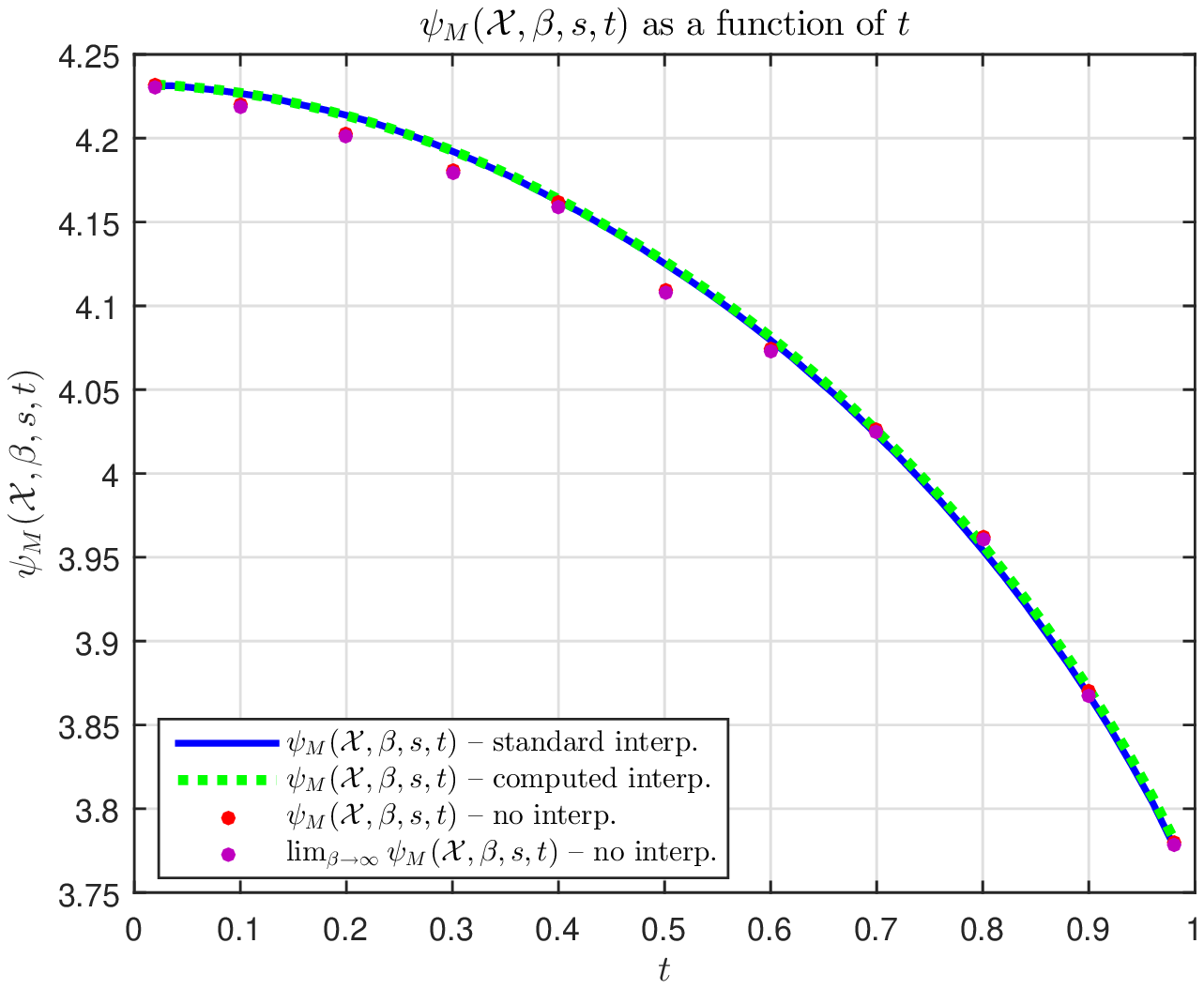,width=9cm,height=7cm}}
\end{minipage}
\begin{minipage}[b]{.5\linewidth}
\centering
\centerline{\epsfig{figure=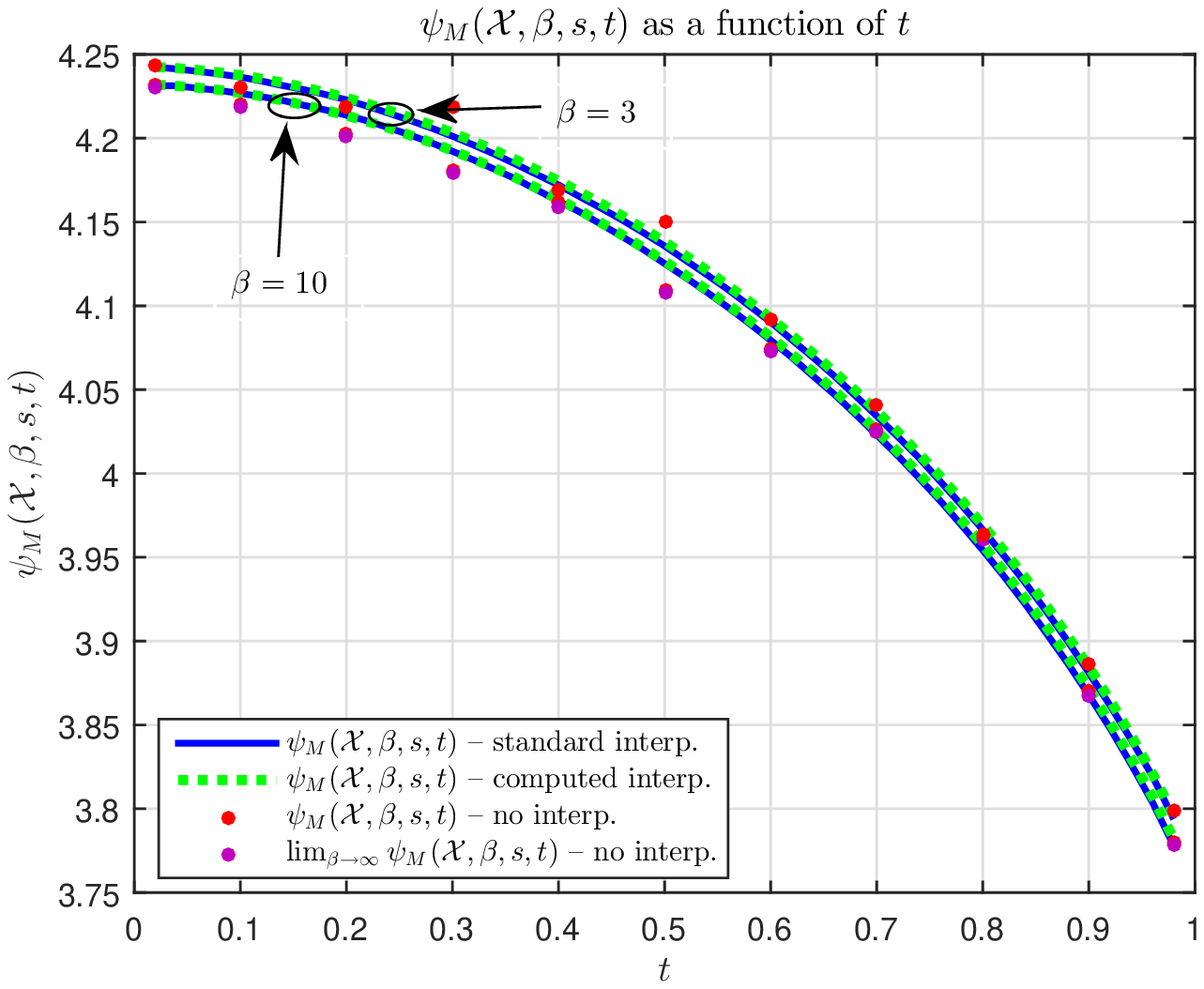,width=9cm,height=7cm}}
\end{minipage}
\caption{Left -- $\psi_M(\calX,\beta,s,t)$ as a function of $t$; $m=5$, $n=5$, $l=10$, $\calX=\calX^{-}$, $\beta=10$, $s=1$; right -- comparison between $\beta=3$ and $\beta=10$}
\label{fig:Mgenbetainfsplus1xdiffnormpsi}
\end{figure}

\begin{table}[h]
\caption{Simulated results --- $m=5$, $n=5$, $l=10$, $\calX=\calX^{-}$, $\beta=10$, $s=1$}\vspace{.1in}
\hspace{-0in}\centering
\begin{tabular}{||c||c|c|c|c|c|c||}\hline\hline
$ t$  &  $\frac{d\psi}{dt}$; (\ref{eq:genanal11}) & $\frac{d\psi}{dt}$;  (\ref{eq:conalt5}) & $\psi$;  (\ref{eq:genanal11}) and (\ref{eq:co1eq1}) & $\psi$; (\ref{eq:conalt5}) and (\ref{eq:co1eq1}) & $\psi$;  (\ref{eq:genanal8})& $\lim_{\beta\rightarrow\infty}\psi$;  (\ref{eq:genanal8})\\  \hline\hline
$ 0.1000 $ & $ -0.0863 $ & $ -0.0895 $ & $\bl{\mathbf{ 4.2267 }}$ & $\bl{\mathbf{ 4.2270 }}$ & $\mathbf{ 4.2196 }$& $\prp{\mathbf{ 4.2180 }}$\\ \hline
$ 0.2000 $ & $ -0.1582 $ & $ -0.1670 $ & $\bl{\mathbf{ 4.2137 }}$ & $\bl{\mathbf{ 4.2135 }}$ & $\mathbf{ 4.2021 }$& $\prp{\mathbf{ 4.2005 }}$\\ \hline
$ 0.3000 $ & $ -0.2516 $ & $ -0.2393 $ & $\bl{\mathbf{ 4.1922 }}$ & $\bl{\mathbf{ 4.1926 }}$ & $\mathbf{ 4.1807 }$& $\prp{\mathbf{ 4.1791 }}$\\ \hline
$ 0.4000 $ & $ -0.3230 $ & $ -0.3247 $ & $\bl{\mathbf{ 4.1625 }}$ & $\bl{\mathbf{ 4.1636 }}$ & $\mathbf{ 4.1611 }$& $\prp{\mathbf{ 4.1594 }}$\\ \hline
$ 0.5000 $ & $ -0.4159 $ & $ -0.4012 $ & $\bl{\mathbf{ 4.1252 }}$ & $\bl{\mathbf{ 4.1266 }}$ & $\mathbf{ 4.1094 }$& $\prp{\mathbf{ 4.1078 }}$\\ \hline
$ 0.6000 $ & $ -0.4946 $ & $ -0.4907 $ & $\bl{\mathbf{ 4.0795 }}$ & $\bl{\mathbf{ 4.0814 }}$ & $\mathbf{ 4.0740 }$& $\prp{\mathbf{ 4.0723 }}$\\ \hline
$ 0.7000 $ & $ -0.6124 $ & $ -0.6059 $ & $\bl{\mathbf{ 4.0230 }}$ & $\bl{\mathbf{ 4.0255 }}$ & $\mathbf{ 4.0264 }$& $\prp{\mathbf{ 4.0247 }}$\\ \hline
$ 0.8000 $ & $ -0.7450 $ & $ -0.7344 $ & $\bl{\mathbf{ 3.9542 }}$ & $\bl{\mathbf{ 3.9574 }}$ & $\mathbf{ 3.9625 }$& $\prp{\mathbf{ 3.9608 }}$\\ \hline
$ 0.9000 $ & $ -0.9485 $ & $ -0.9422 $ & $\bl{\mathbf{ 3.8678 }}$ & $\bl{\mathbf{ 3.8715 }}$ & $\mathbf{ 3.8699 }$& $\prp{\mathbf{ 3.8680 }}$
\\ \hline\hline
\end{tabular}
\label{tab:Mgenbetainfsplus1xdiffnormpsi}
\end{table}

\textbf{\underline{\emph{2) $s=1$ -- a Gordon's comparison principle}}}

\noindent When $s=-1$ we have

\begin{eqnarray}\label{eq:Mbetainfsminus1}
\lim_{\beta\rightarrow\infty} \xi_M(\calX,\beta,-1) & = & \mE_{G,u^{(4)}} \frac{\max_{\x^{(i)}\in \calX} \lp - \|
 G\x^{(i)}\|_2+\|\x^{(i)}\|_2u^{(4)}\rp}{\sqrt{n}}\nonumber \\
& = & -\mE_{G,u^{(4)}} \frac{\min_{\x^{(i)}\in \calX} \lp  \|
 G\x^{(i)}\|_2-\|\x^{(i)}\|_2u^{(4)}\rp}{\sqrt{n}}.
\end{eqnarray}
Relying again on $\xi_M(\calX,\beta,1)=\psi(\calX,\beta,1,1)$ and the above machinery we have
\begin{eqnarray}\label{eq:Mbetainfsminus2}
-\mE_{G,u^{(4)}} \frac{\min_{\x^{(i)}\in \calX} \lp  \|
 G\x^{(i)}\|_2-\|\x^{(i)}\|_2u^{(4)}\rp}{\sqrt{n}} & = & \lim_{\beta\rightarrow\infty} \xi_M(\calX,\beta,1)=
 \lim_{\beta\rightarrow\infty} \psi_M(\calX,\beta,1,1) \nonumber \\
& \leq &  \lim_{\beta\rightarrow\infty} \psi_M(\calX,\beta,1,0)\nonumber \\
& = &
\mE_{\u^{(2)},\h} \frac{\max_{\x^{(i)}\in \calX} \lp -\|\x^{(i)}\|_2\|\u^{(2)}\|_2 +\h^T\x^{(i)}\rp}{\sqrt{n}} \nonumber \\
& = &
- \mE_{\u^{(2)},\h} \frac{\min_{\x^{(i)}\in \calX} \lp \|\x^{(i)}\|_2\|\u^{(2)}\|_2 -\h^T\x^{(i)}\rp}{\sqrt{n}}.
\end{eqnarray}
(\ref{eq:betainfsminus2}) is of course again a form of the well-known Gordon comparison principle \cite{Gordon85}. Namely, according to the Gordon's principle one has the following
\begin{eqnarray}\label{eq:Mbetainfsminus3}
\mE_{G,u^{(4)}} \frac{\min_{\x^{(i)}\in \calX} \lp \|
 G\x^{(i)}\|_2+\|\x^{(i)}\|_2u^{(4)}\rp}{\sqrt{n}} & = & \mE_{G,u^{(4)}} \frac{\min_{\x^{(i)}\in \calX}\max_{\y\in S^{m-1}} \lp \y^T G\x^{(i)}+\|\x^{(i)}\|_2u^{(4)}\rp}{\sqrt{n}}\nonumber \\
 & \geq & \mE_{\u^{(2)},\h} \frac{\min_{\x^{(i)}\in \calX}\max_{\y\in S^{m-1}} \lp \|\x^{(i)}\|_2\y^T\u^{(2)} +\h^T\x^{(i)}\rp}{\sqrt{n}} \nonumber \\
 & = &
\mE_{\u^{(2)},\h} \frac{\min_{\x^{(i)}\in \calX} \lp \|\x^{(i)}\|_2\|\u^{(2)}\|_2 +\h^T\x^{(i)}\rp}{\sqrt{n}}.\nonumber \\
\end{eqnarray}
Clearly, connecting beginning and end in both, (\ref{eq:Mbetainfsminus2}) and (\ref{eq:Mbetainfsminus3}) we obtain the same inequalities (as earlier, $-\h$ and $\h$ have the same distribution, and so do $-u^{(4)}$ and $u^{(4)}$). SImilarly to what we had above,  (\ref{eq:Mbetainfsminus2}) and (\ref{eq:Mbetainfsminus3}) are again a special case of a much stronger concept introduced in Theorem \ref{thm:thm1a}.


\textbf{\underline{\emph{Numerical results}}}

In Figure \ref{fig:Mgenbetainfsmin1xdiffnormpsi} and Table \ref{tab:Mgenbetainfsmin1xdiffnormpsi} simulated results are shown. All parameters are again the same as earlier, and $\beta=10$ is selected to emulate $\beta\rightarrow\infty$.
Both, Figure \ref{fig:Mgenbetainfsmin1xdiffnormpsi} and Table \ref{tab:Mgenbetainfsmin1xdiffnormpsi}, again demonstrate a solid agreement between all the presented results with $\beta=10$ being a pretty good approximation of $\beta\rightarrow\infty$.


\begin{figure}[htb]
\begin{minipage}[b]{.5\linewidth}
\centering
\centerline{\epsfig{figure=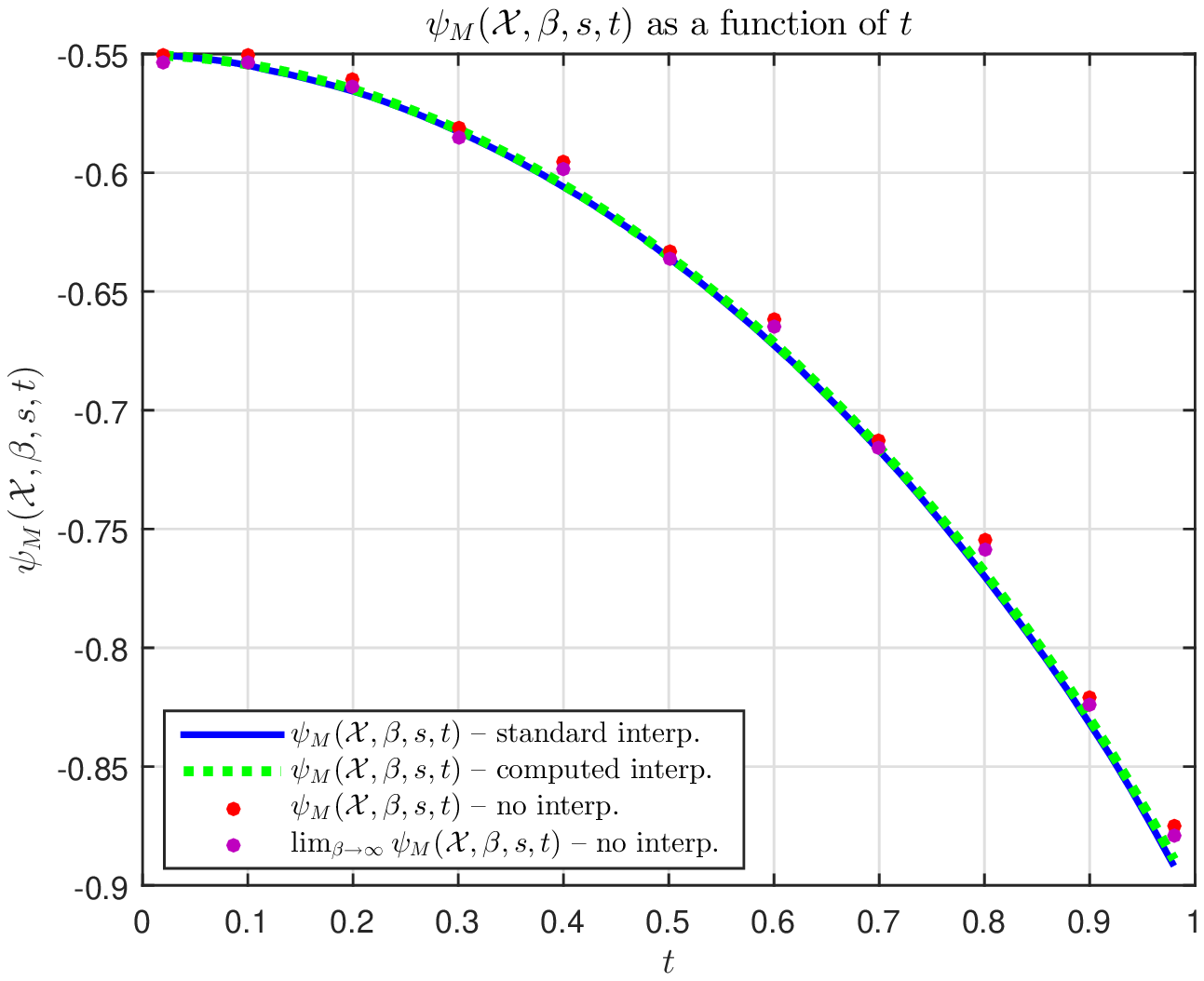,width=9cm,height=7cm}}
\end{minipage}
\begin{minipage}[b]{.5\linewidth}
\centering
\centerline{\epsfig{figure=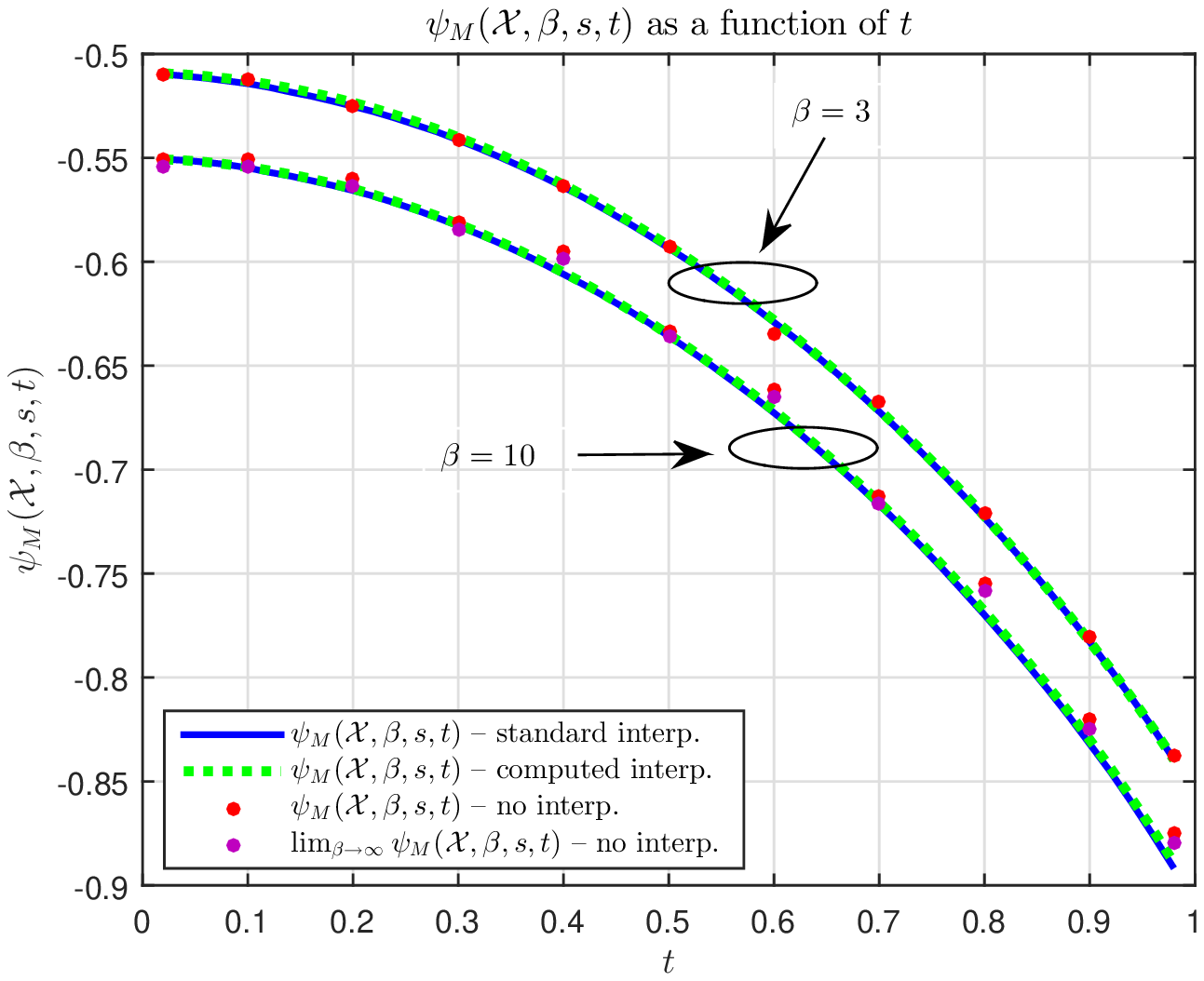,width=9cm,height=7cm}}
\end{minipage}
\caption{Left -- $\psi_M(\calX,\beta,s,t)$ as a function of $t$; $m=5$, $n=5$, $l=10$, $\calX=\calX^{-}$, $\beta=10$, $s=-1$; right -- comparison between $\beta=3$ and $\beta=10$}
\label{fig:Mgenbetainfsmin1xdiffnormpsi}
\end{figure}

\begin{table}[h]
\caption{Simulated results --- $m=5$, $n=5$, $l=10$, $\calX=\calX^{-}$, $\beta=10$, $s=-1$}\vspace{.1in}
\hspace{-0in}\centering
\begin{tabular}{||c||c|c|c|c|c|c||}\hline\hline
$ t$  &  $\frac{d\psi}{dt}$; (\ref{eq:genanal11}) & $\frac{d\psi}{dt}$;  (\ref{eq:conalt5}) & $\psi$;  (\ref{eq:genanal11}) and (\ref{eq:co1eq1}) & $\psi$; (\ref{eq:conalt5}) and (\ref{eq:co1eq1}) & $\psi$;  (\ref{eq:genanal8})& $\lim_{\beta\rightarrow\infty}\psi$;  (\ref{eq:genanal8})\\  \hline\hline
$ 0.1000 $ & $ -0.0785 $ & $ -0.0670 $ & $\bl{\mathbf{ -0.5548 }}$ & $\bl{\mathbf{ -0.5544 }}$ & $\mathbf{ -0.5504 }$& $\prp{\mathbf{ -0.5537 }}$\\ \hline
$ 0.2000 $ & $ -0.1263 $ & $ -0.1309 $ & $\bl{\mathbf{ -0.5657 }}$ & $\bl{\mathbf{ -0.5649 }}$ & $\mathbf{ -0.5604 }$& $\prp{\mathbf{ -0.5638 }}$\\ \hline
$ 0.3000 $ & $ -0.1932 $ & $ -0.1929 $ & $\bl{\mathbf{ -0.5827 }}$ & $\bl{\mathbf{ -0.5819 }}$ & $\mathbf{ -0.5815 }$& $\prp{\mathbf{ -0.5849 }}$\\ \hline
$ 0.4000 $ & $ -0.2583 $ & $ -0.2599 $ & $\bl{\mathbf{ -0.6060 }}$ & $\bl{\mathbf{ -0.6052 }}$ & $\mathbf{ -0.5950 }$& $\prp{\mathbf{ -0.5984 }}$\\ \hline
$ 0.5000 $ & $ -0.3379 $ & $ -0.3304 $ & $\bl{\mathbf{ -0.6358 }}$ & $\bl{\mathbf{ -0.6351 }}$ & $\mathbf{ -0.6330 }$& $\prp{\mathbf{ -0.6364 }}$\\ \hline
$ 0.6000 $ & $ -0.3975 $ & $ -0.3955 $ & $\bl{\mathbf{ -0.6727 }}$ & $\bl{\mathbf{ -0.6716 }}$ & $\mathbf{ -0.6615 }$& $\prp{\mathbf{ -0.6650 }}$\\ \hline
$ 0.7000 $ & $ -0.4738 $ & $ -0.4685 $ & $\bl{\mathbf{ -0.7170 }}$ & $\bl{\mathbf{ -0.7158 }}$ & $\mathbf{ -0.7123 }$& $\prp{\mathbf{ -0.7159 }}$\\ \hline
$ 0.8000 $ & $ -0.5622 $ & $ -0.5668 $ & $\bl{\mathbf{ -0.7699 }}$ & $\bl{\mathbf{ -0.7682 }}$ & $\mathbf{ -0.7549 }$& $\prp{\mathbf{ -0.7587 }}$\\ \hline
$ 0.9000 $ & $ -0.6710 $ & $ -0.6696 $ & $\bl{\mathbf{ -0.8320 }}$ & $\bl{\mathbf{ -0.8300 }}$ & $\mathbf{ -0.8205 }$& $\prp{\mathbf{ -0.8244 }}$
\\ \hline\hline
\end{tabular}
\label{tab:Mgenbetainfsmin1xdiffnormpsi}
\end{table}

\section{Lifting}
\label{sec:lifting}

In this section we introduce a powerful lifting principle that in some scenarios substantially improves on the above introduced comparison concepts. We will be interested in the following function
\begin{eqnarray}\label{eq:liftanal1}
 f_*(G,u^{(4)},\calX,\beta,s,c_3)= \lp \sum_{i=1}^{l}e^{\beta\lp s\|
 G\x^{(i)}\|_2\rp+\|\x^{(i)}\|_2u^{(4)}} \rp^{c_3},
\end{eqnarray}
where $G,u^{(4)},\calX,\beta$, and $s$ are as earlier, and $c_3\geq 0$ is a real number (a Gaussian framework is also assumed, i.e. we will continue to assume that $G\in \mR^{m\times n}$ is an $(m\times n)$ dimensional matrix with i.i.d. standard normal components; also we assume a general set $\calX$). In such a context, the expected value of the above function is again the most relevant. We will denote this expected value by $\xi_*(\calX,\beta,s,c_3)$ and set
\begin{eqnarray}\label{eq:liftanal2}
\xi_*(\calX,\beta,s,c_3)\triangleq \mE_{G,u^{(4)}} f_*(G,u^{(4)},\calX,\beta,s,c_3)= \mE_{G,u^{(4)}} \lp \sum_{i=1}^{l}e^{\beta\lp s\|
 G\x^{(i)}\|_2+\|\x^{(i)}\|_2u^{(4)}\rp} \rp^{c_3}.
\end{eqnarray}
Following into the footsteps of what we presented in Section \ref{sec:gencon}, we will employ the following interpolating function $\psi_*(\cdot)$ to study $\xi_*(\calX,\beta,s,c_3)$
\begin{eqnarray}\label{eq:liftanal3}
\psi_*(\calX,\beta,s,c_3,t)= \mE_{G,\u^{(4)},\u^{(2)},\h} \lp \sum_{i=1}^{l}e^{\beta\lp s\|\sqrt{t}
 G\x^{(i)}+\sqrt{1-t}\|\x^{(i)}\|_2\u^{(2)}\|_2+\sqrt{t}\u^{(4)}+\sqrt{1-t}\h^T\x^{(i)}\rp} \rp^{c_3},
\end{eqnarray}
where $\u^{(2)}$ is as earlier. Clearly, $\xi_*(\calX,\beta,s,c_3)=\psi_*(\calX,\beta,s,c_3,1)$. Similarly to what we had in Section \ref{sec:gencon}, $\psi_*(\calX,\beta,s,c_3,0)$ is an object much easier to handle than $\psi(\calX,\beta,s,c_3,1)$. Following further what was done in Section \ref{sec:gencon}, below we will try to connect $\psi_*(\calX,\beta,s,c_3,1)$ to $\psi_*(\calX,\beta,s,c_3,0)$ which will then automatically connect $\xi_*(\calX,\beta,s,c_3)$ to $\psi_*(\calX,\beta,s,c_3,0)$. Recalling on (\ref{eq:genanal4}) we have analogously to (\ref{eq:genanal6})
\begin{eqnarray}\label{eq:lifttanal6}
\psi_*(\calX,\beta,s,c_3,t)
 =  \mE_{\u^{(i,1)},\u^{(2)},\u^{(i,3)},u^{(4)}}\lp \sum_{i=1}^{l}e^{\beta\lp s\sqrt{\sum_{j=1}^{m}\lp\sqrt{t}\u_j^{(i,1)}+\sqrt{1-t}\u_j^{(2)}\rp^2}+\sqrt{t}u^{(4)}+\sqrt{1-t}\u^{(i,3)}\rp}\rp^{c_3},
\end{eqnarray}
where $\u_j^{(i,1)}$ and $\u_j^{(i,3)}$ are as earlier (in a sense that they are the inner products of vectors of i.i.d. standard normals and unit norm vectors $\frac{\x^{(i)}}{\|\x^{(i)}\|_2}$) and $\beta_i=\beta\|\x^{(i)}\|_2$).
Recalling on $B^{(i)}$, $A^{(i)}_{M}$, and $Z_M$ from (\ref{eq:genanal7}) and introducing $A^{(i)}_{*}$ and $Z_{*}$ to be more aligned with the notation of this section
\begin{eqnarray}\label{eq:liftanal7}
B^{(i)} & \triangleq &  \sqrt{\sum_{j=1}^{m}\lp\sqrt{t}\u_j^{(i,1)}+\sqrt{1-t}\u_j^{(2)}\rp^2} \nonumber \\
A^{(i)}_{*} & \triangleq & A^{(i)}_{M} =  e^{\beta_i(sB^{(i)}+\sqrt{t}u^{(4)}+\sqrt{1-t}\u^{(i,3)})}\nonumber \\
Z_{*} & \triangleq & Z_M =\sum_{i=1}^{l}e^{\beta_i\lp s\sqrt{\sum_{j=1}^{m}\lp\sqrt{t}\u_j^{(i,1)}+\sqrt{1-t}\u_j^{(2)}\rp^2}+\sqrt{t}u^{(4)}+\sqrt{1-t}\u^{(i,3)}\rp}= \sum_{i=1}^{l} A^{(i)}_{*},
\end{eqnarray}
we can rewrite (\ref{eq:liftanal7}) in the following way
\begin{eqnarray}\label{eq:liftanal8}
\psi_*(\calX,\beta,s,c_3,t) & = &  \mE_{\u^{(i,1)},\u^{(2)},\u^{(i,3)},u^{(4)}} (Z_*)^{c_3}.
\end{eqnarray}
Similarly to what we was done in Section \ref{sec:gencon}, we will study behavior of $\psi_*(\calX,\beta,s,t)$ when viewed as a function of $t$. To that end we have
\begin{multline}\label{eq:liftanal9}
\frac{d\psi_*(\calX,\beta,s,c_3,t)}{dt}  =   \mE_{\u^{(i,1)},\u^{(2)},\u^{(i,3)},u^{(4)}} \frac{d\lp \sum_{i=1}^{l}e^{\beta_i\lp s\sqrt{\sum_{j=1}^{m}\lp\sqrt{t}\u_j^{(i,1)}+\sqrt{1-t}\u_j^{(2)}\rp^2}+\sqrt{t}u^{(4)}+\sqrt{1-t}\u^{(i,3)}\rp} \rp^{c_3}}{dt} \\
 =   \mE_{\u^{(i,1)},\u^{(2)},\u^{(i,3)},u^{(4)}} c_3Z^{c_3-1}_* \frac{d\lp \sum_{i=1}^{l}e^{\beta_i\lp s\sqrt{\sum_{j=1}^{m}\lp\sqrt{t}\u_j^{(i,1)}+\sqrt{1-t}\u_j^{(2)}\rp^2}+\sqrt{t}u^{(4)}+\sqrt{1-t}\u^{(i,3)}\rp} \rp }{dt} \\
 =   \mE_{\u^{(i,1)},\u^{(2)},\u^{(i,3)},u^{(4)}} c_3Z^{c_3-1}_*  \sum_{i=1}^{l}\beta_iA^{(i)}_*\lp s\frac{dB^{(i)}}{dt}+\frac{u^{(4)}}{2\sqrt{t}}-\frac{\u^{(i,3)}}{2\sqrt{1-t}}\rp.
\end{multline}
Utilizing $\frac{dB^{(i)}}{dt}$ from (\ref{eq:genanal10}) we have
\begin{multline}\label{eq:liftanal11}
\frac{d\psi_*(\calX,\beta,s,c_3,t)}{dt}
 =   \mE_{\u^{(i,1)},\u^{(2)},\u^{(i,3)},u^{(4)}} \frac{sc_3}{2}\sum_{j=1}^{m} \sum_{i=1}^{l} \frac{ \beta_iA^{(i)}_*\lp(\u_j^{(i,1)})^2-(\u_j^{(2)})^2+\u_j^{(i,1)}\u_j^{(2)}\lp\frac{\sqrt{1-t}}{\sqrt{t}}-\frac{\sqrt{t}}{\sqrt{1-t}}\rp \rp }{Z^{1-c_3}_*B^{(i)}} \\
  + \mE_{\u^{(i,1)},\u^{(2)},\u^{(i,3)},u^{(4)}} \frac{c_3}{2}  \sum_{i=1}^{l}
\frac{\beta_iA^{(i)}_*u^{(4)}}{Z^{1-c_3}_*\sqrt{t}} - \mE_{\u^{(i,1)},\u^{(2)},\u^{(i,3)},u^{(4)}} \frac{c_3}{2}  \sum_{i=1}^{l}
\frac{\beta_iA^{(i)}_*\u^{(i,3)}}{Z^{1-c_3}_*\sqrt{1-t}}.
\end{multline}
As in Section \ref{sec:gencon}, all the terms appearing in the above sums can be handled separately. In fact, many of the computations already done in Section \ref{sec:gencon} can be repeated with minimal obvious changes. Below we present what the final results of these changes are for each of the terms.


\subsection{Computing the derivatives}
\label{sec:lifcompder}

Here we show how one can quickly determine all the relevant derivatives.

\textbf{\underline{\emph{1) Computing $\mE_{\u^{(i,1)},\u^{(2)},\u^{(i,3)},u^{(4)}}  \frac{ A^{(i)}_*\u_j^{(i,1)}\u_j^{(2)}}{Z^{1-c_3}_*B^{(i)}}$}}}

\noindent  As in Section \ref{sec:gencon} we present two different characterizations.

\textbf{\underline{\emph{1.1) Fixing $\u^{(i,1)}$}}}

\noindent Similarly to what we had in (\ref{eq:genanal13}) we now have
\begin{eqnarray}\label{eq:liftgenanal13}
\mE_{\u^{(i,1)},\u^{(2)},\u^{(i,3)},u^{(4)}}  \frac{ A^{(i)}_*\u_j^{(i,1)}\u_j^{(2)}}{Z^{1-c_3}_*B^{(i)}} & = & \mE\lp
\sum_{p=1,p\neq i}^{l} \frac{(\x^{(i)})^T\x^{(p)}}{\|\x^{(i)}\|_2\|\x^{(p)}\|_2}\frac{d}{d\u_j^{(p,1)}}\lp \frac{ A^{(i)}_*\u_j^{(2)}}{Z^{1-c_3}_*B^{(i)}}\rp\rp \nonumber \\
& & +\mE \lp(\frac{(\x^{(i)})^T\x^{(i)}}{\|\x^{(i)}\|_2\|\x^{(i)}\|_2}\frac{d}{d\u_j^{(i,1)}}\lp \frac{ A^{(i)}_*\u_j^{(2)}}{Z^{1-c_3}_*B^{(i)}}\rp\rp.
\end{eqnarray}
Similarly to what was noted in Section \ref{sec:moving}, one can again closely follow the derivations from Section \ref{sec:hand1} and observe that every single step can be repeated with rather minimal differences. Namely, besides already observed changes from Section \ref{sec:moving}, $E(\u_j^{(i,1)}\u_j^{(p,1)})=\frac{(\x^{(i)})^T\x^{(p)}}{\|\x^{(i)}\|_2\|\x^{(p)}\|_2}$ and $E(\u_j^{(i,3)}\u_j^{(p,3)})=\frac{(\x^{(i)})^T\x^{(p)}}{\|\x^{(i)}\|_2\|\x^{(p)}\|_2}$, the powers of $Z_*$ and the constants that multiply them will also change (where we used to have power one now we will have $(1-c_3)$ and where we used to have two now we will have $(2-c_3)$; the constant that multiplied power two was minus one while now the corresponding constant that multiplies power $(2-c_3)$ will be $-(1-c_3)$). Following  (\ref{eq:genanal19}) and (\ref{eq:Mgenanal19}) we then have
\begin{multline}\label{eq:liftgenanal19}
\mE_{\u^{(i,1)},\u^{(2)},\u^{(i,3)},u^{(4)}}  \frac{ A^{(i)}_*\u_j^{(i,1)}\u_j^{(2)}}{Z^{1-c_3}_*B^{(i)}}
 =
\mE\lp\frac{ A^{(i)}_*}{Z^{1-c_3}_*B^{(i)}}\lp \lp \beta_i s -\frac{1}{B^{(i)}}\rp\frac{\u_j^{(2)}\sqrt{1-t}+\sqrt{t}\u_j^{(i,1)}}{B^{(i)}}\u^{(2)}\sqrt{t}\rp\rp \\
-(1-c_3)\mE\sum_{p=1}^{l} \frac{(\x^{(i)})^T\x^{(p)}}{\|\x^{(i)}\|_2\|\x^{(p)}\|_2} \frac{\beta_p s  A^{(i)}_*A^{(p)}_*\u_j^{(2)}\sqrt{t}(\u_j^{(p,1)}\sqrt{t}+\sqrt{1-t}\u_j^{(2)})}{Z^{2-c_3}_*B^{(i)}B^{(p)}}.
\end{multline}

\textbf{\underline{\emph{1.2) Fixing $\u^{(2)}$}}}

\noindent Similarly to what we had in (\ref{eq:genCanal1}) we now have
\begin{equation}\label{eq:liftgenCanal1}
\mE_{\u^{(i,1)},\u^{(2)},\u^{(i,3)},u^{(4)}}  \frac{ A^{(i)}_*\u_j^{(2)}\u^{(i,1)}}{Z^{1-c_3}_*B^{(i)}} =
\mE\lp\mE (\u_j^{(2)}\u_j^{(2)})\frac{d}{d\u_j^{(2)}}\lp \frac{ A^{(i)}_*\u^{(i,1)}}{Z^{1-c_3}_*B^{(i)}}\rp\rp
=\mE\lp\u^{(i,1)}\frac{d}{d\u_j^{(2)}}\lp \frac{ A^{(i)}_*}{Z^{1-c_3}_*B^{(i)}}\rp\rp,
\end{equation}
and utilizing the above observation about changing the powers of $Z_*$
from (\ref{eq:genCanal5}) and (\ref{eq:Mgenanal5}) we have
\begin{eqnarray}\label{eq:liftgenCanal5}
\mE_{\u^{(i,1)},\u^{(2)},\u^{(i,3)},u^{(4)}}  \frac{ A^{(i)}_*\u_j^{(i,1)}\u_j^{(2)}}{Z^{1-c_3}_*B^{(i)}}
 & = &
\mE\lp\frac{ A^{(i)}_*}{Z^{1-c_3}_*B^{(i)}}\lp \lp \beta_i s -\frac{1}{B^{(i)}}\rp\frac{\u_j^{(2)}\sqrt{1-t}+\sqrt{t}\u_j^{(i,1)}}{B^{(i)}}\u_j^{(i,1)}\sqrt{1-t}\rp\rp \nonumber \\
& &  -(1-c_3)\mE\sum_{p=1}^{l}\frac{\beta_p s  A^{(i)}_*A^{(p)}_*\u_j^{(i,1)}\sqrt{1-t}}{Z^{2-c_3}_*B^{(i)}B^{(p)}} (\u_j^{(2)}\sqrt{1-t}+\sqrt{t}\u_j^{(p,1)}).\nonumber \\
\end{eqnarray}

\textbf{\underline{\emph{2) Computing $\mE_{\u^{(i,1)},\u^{(2)},\u^{(i,3)},u^{(4)}}  \frac{ A^{(i)}_*(\u_j^{(2)})^2}{Z^{1-c_3}_*B^{(i)}}$}}}

\noindent Following (\ref{eq:gen1anal1}) we have
\begin{eqnarray}\label{eq:liftgen1anal1}
\mE_{\u^{(i,1)},\u^{(2)},\u^{(i,3)},u^{(4)}}  \frac{ A^{(i)}_*(\u_j^{(2)})^2}{Z^{1-c_3}_*B^{(i)}} & = &
\mE\lp\mE (\u_j^{(2)}\u_j^{(2)})\frac{d}{d\u_j^{(2)}}\lp \frac{ A^{(i)}_*\u_j^{(2)}}{Z^{1-c_3}_*B^{(i)}}\rp\rp\nonumber \\
& = & \mE\lp\frac{ A^{(i)}_*}{Z^{1-c_3}_*B^{(i)}}+\u_j^{(2)}\frac{d}{d\u_j^{(2)}}\lp \frac{ A^{(i)}_*}{Z^{1-c_3}_*B^{(i)}}\rp\rp.
\end{eqnarray}
Utilizing the above changing powers of $Z_*$ reasoning we then easily have after following the derivation of (\ref{eq:gen1anal5}) and (\ref{eq:Mgen1anal5})
\begin{eqnarray}\label{eq:liftgen1anal5}
\mE_{\u^{(i,1)},\u^{(2)},\u^{(i,3)},u^{(4)}}  \frac{ A^{(i)}_*(\u_j^{(2)})^2}{Z^{1-c_3}_*B^{(i)}}
 & = &
\mE\lp\frac{ A^{(i)}_*}{Z^{1-c_3}_*B^{(i)}}\lp 1 +\lp \beta_i s -\frac{1}{B^{(i)}}\rp\frac{\u_j^{(2)}\sqrt{1-t}+\sqrt{t}\u_j^{(i,1)}}{B^{(i)}}\u_j^{(2)}\sqrt{1-t}\rp\rp \nonumber \\
& &  -(1-c_3)\mE\sum_{p=1}^{l}\frac{\beta_p s  A^{(i)}_*A^{(p)}_*\u_j^{(2)}\sqrt{1-t}}{Z^{2-c_3}_*B^{(i)}B^{(p)}} (\u_j^{(2)}\sqrt{1-t}+\sqrt{t}\u_j^{(p,1)}).\nonumber \\
\end{eqnarray}

\textbf{\underline{\emph{3) Computing $\mE_{\u^{(i,1)},\u^{(2)},\u^{(i,3)},u^{(4)}}  \frac{ A^{(i)}_*(\u_j^{(i,1)})^2}{Z^{1-c_3}_*B^{(i)}}$}}}

\noindent Following (\ref{eq:gen2anal12}) we have
\begin{eqnarray}\label{eq:liftgen2anal12}
\mE_{\u^{(i,1)},\u^{(2)},\u^{(i,3)},u^{(4)}}  \frac{ A^{(i)}_*\u_j^{(i,1)}\u_j^{(2)}}{Z^{1-c_3}_*B^{(i)}}  & = &
\mE\lp\sum_{p=1,p\neq i}^{l} \mE (\u_j^{(i,1)}\u_j^{(p,1)})\frac{d}{d\u_j^{(p,1)}}\lp \frac{ A^{(i)}_*\u_j^{(i,1)}}{Z^{1-c_3}_*B^{(i)}}\rp \rp \nonumber \\
& & +\mE\lp\mE (\u_j^{(i,1)}\u_j^{(i,1)})\frac{d}{d\u_j^{(i,1)}}\lp \frac{ A^{(i)}_*\u_j^{(i,1)}}{Z^{1-c_3}_*B^{(i)}}\rp\rp,
\end{eqnarray}
and through the changing powers of $Z_*$ from (\ref{eq:gen2anal19}) and (\ref{eq:Mgen2anal19})
\begin{multline}\label{eq:liftgen2anal19}
\mE_{\u^{(i,1)},\u^{(2)},\u^{(i,3)},u^{(4)}}  \frac{ A^{(i)}_*(\u_j^{(i,1)})^2}{Z^{1-c_3}_*B^{(i)}}
 =
\mE\lp\frac{ A^{(i)}_*}{Z^{1-c_3}_*B^{(i)}}\lp 1 +\lp \beta_i s -\frac{1}{B^{(i)}}\rp\frac{\u_j^{(2)}\sqrt{1-t}+\sqrt{t}\u_j^{(i,1)}}{B^{(i)}}\u_j^{(i,1)}\sqrt{t}\rp\rp  \\  -(1-c_3)\mE\sum_{p=1}^{l} \frac{(\x^{(i)})^T\x^{(p)}}{\|\x^{(i)}\|_2\|\x^{(p)}\|_2} \frac{ \beta_p s A^{(i)}_*A^{(p)}_*\u_j^{(i,1)}\sqrt{t}(\u_j^{(p,1)}\sqrt{t}+\sqrt{1-t}\u_j^{(2)})}{Z^{2-c_3}_*B^{(i)}B^{(p)}}.
\end{multline}

\textbf{\underline{\emph{4) Computing $\mE_{\u^{(i,1)},\u^{(2)},\u^{(i,3)},u^{(4)}}  \frac{ A^{(i)}_*\u^{(i,3)}}{Z^{1-c_3}_*}$}}}

\noindent From (\ref{eq:genDanal12}) we also have
\begin{equation}\label{eq:liftgenDanal12}
\mE_{\u^{(i,1)},\u^{(2)},\u^{(i,3)},u^{(4)}}  \frac{ A^{(i)}_*\u^{(i,3)}}{Z^{1-c_3}_*} =
\mE\lp\sum_{p=1,p\neq i}^{l} \mE (\u^{(i,3)}\u^{(p,3)})\frac{d}{d\u^{(p,3)}}\lp \frac{ A^{(i)}_*}{Z^{1-c_3}_*}\rp
+\mE (\u^{(i,3)}\u^{(i,3)})\frac{d}{d\u^{(i,3)}}\lp \frac{ A^{(i)}_*}{Z^{1-c_3}_*}\rp\rp,
\end{equation}
and, similarly as above, through the derivation of (\ref{eq:genDanal19}) and (\ref{eq:MgenDanal19})
\begin{equation}\label{eq:liftgenDanal19}
\mE_{\u^{(i,1)},\u^{(2)},\u^{(i,3)},u^{(4)}}  \frac{ A^{(i)}_*\u^{(i,3)}}{Z^{1-c_3}_*}  =
 \mE\frac{\beta_i A^{(i)}_*\sqrt{1-t}}{Z^{1-c_3}_*} -(1-c_3)\mE\sum_{p=1}^{l} \frac{(\x^{(i)})^T\x^{(p)}}{\|\x^{(i)}\|_2\|\x^{(p)}\|_2}\frac{\beta_p A^{(i)}_*  A^{(p)}_* \sqrt{1-t}}{Z^{2-c_3}_*}.
\end{equation}

\vspace{.1in}
\textbf{\underline{\emph{5) Computing $\mE_{\u^{(i,1)},\u^{(2)},\u^{(i,3)},u^{(4)}}  \frac{ A^{(i)}_*\u^{(4)}}{Z^{1-c_3}_*}$}}}

\noindent Utilizing integration by parts we obtain
\begin{equation}\label{eq:liftgenEanal12}
\mE_{\u^{(i,1)},\u^{(2)},\u^{(i,3)},u^{(4)}}  \frac{ A^{(i)}_*\u^{(4)}}{Z^{1-c_3}_*} =
\mE\lp\mE (\u^{(4)}\u^{(4)})\frac{d}{d\u^{(4)}}\lp \frac{ A^{(i)}_*}{Z^{1-c_3}_*}\rp\rp
=\mE\lp\frac{d}{d\u^{(4)}}\lp \frac{ A^{(i)}_*}{Z^{1-c_3}_*}\rp\rp.
\end{equation}
We also have
\begin{eqnarray}\label{eq:liftgenEanal18}
\frac{d}{d\u^{(4)}}\lp \frac{ A^{(i)}_*}{Z^{1-c_3}_*}\rp  & = &
\frac{1}{Z^{1-c_3}_*}\frac{d A^{(i)}_*}{d\u^{(4)}}-(1-c_3)\frac{A^{(i)}_*}{Z^{2-c_3}}\frac{dZ_*}{d\u^{(4)}} \nonumber \\
& = &   \frac{\beta_i A^{(i)}_*\sqrt{t}}{Z^{1-c_3}_*}
-(1-c_3)\frac{ A^{(i)}_*}{Z^{2-c_3}_*}\sum_{p=1}^{l}\frac{d A^{(i)}_*}{d\u^{(4)}}  = \frac{\beta_i A^{(i)}_*\sqrt{t}}{Z^{1-c_3}_*} -(1-c_3)\sum_{p=1}^{l}\frac{ A^{(i)}_*A^{(p)}_*\beta_p\sqrt{t}}{Z^{2-c_3}_*}.\nonumber \\
\end{eqnarray}
A combination of (\ref{eq:liftgenEanal12}) and (\ref{eq:liftgenEanal18}) gives
\begin{equation}\label{eq:liftgenEanal19}
\mE_{\u^{(i,1)},\u^{(2)},\u^{(i,3)},u^{(4)}}  \frac{ A^{(i)}_*\u^{(4)}}{Z^{1-c_3}_*}  =
\mE \frac{\beta_i A^{(i)}_*\sqrt{t}}{Z^{1-c_3}_*} -(1-c_3)\mE\sum_{p=1}^{l}\frac{ A^{(i)}_*A^{(p)}_*\beta_p\sqrt{t}}{Z^{2-c_3}_*}.
\end{equation}

\subsection{Connecting all pieces together}
\label{sec:liftconalt}

Using (\ref{eq:liftanal11}), (\ref{eq:liftgenanal19}), (\ref{eq:liftgenCanal5}), (\ref{eq:liftgen1anal5}), (\ref{eq:liftgen2anal19}), (\ref{eq:liftgenDanal19}), and (\ref{eq:liftgenEanal19}) and following the procedure from Section \ref{sec:conalt} we obtain similarly to (\ref{eq:conalt4}) and (\ref{eq:Mconalt4})
\begin{eqnarray}\label{eq:liftconalt4}
\frac{d\psi_*(\calX,\beta,s,c_3,t)}{dt}
 & = & \frac{sc_3(1-c_3)}{2} \sum_{i=1}^{l}\mE\sum_{p=1}^{l}\sum_{j=1}^{m} (\|\x^{(i)}\|_2\|\x^{(p)}\|_2-(\x^{(i)})^T\x^{(p)})\nonumber \\
  & & \times \frac{ \beta^2 s A^{(i)}_*A^{(p)}_*(\u_j^{(i,1)}\sqrt{t}+\u_j^{(2)}\sqrt{1-t})(\u_j^{(p,1)}\sqrt{t}+\sqrt{1-t}\u_j^{(2)})}{Z^{2-c_3}_*B^{(i)}B^{(p)}} \nonumber \\
& & +\frac{c_3}{2}   \mE\sum_{i=1}^{l}\lp \mE \frac{\beta_i^2 A^{(i)}_*}{Z^{1-c_3}_*} -(1-c_3)\mE\sum_{p=1}^{l}\frac{ A^{(i)}_*A^{(p)}_*\beta_i\beta_p}{Z^{2-c_3}_*}\rp \nonumber \\
& & - \frac{c_3}{2}   \mE\sum_{i=1}^{l}\lp\mE\frac{\beta _i^2A^{(i)}_*}{Z^{1-c_3}_*} -(1-c_3)\mE\sum_{p=1}^{l} \frac{(\x^{(i)})^T\x^{(p)}}{\|\x^{(i)}\|_2\|\x^{(p)}\|_2}\frac{\beta_i\beta_p A^{(i)}_*  A^{(p)}_* }{Z^{2-c_3}_*}\rp.
\end{eqnarray}
The sum to the right of the equality sign in the first row is obtained in the same way as the corresponding one in (\ref{eq:conalt4}) with the above mentioned small adjustment to the power of $Z_*$ and the corresponding multiplying constant and to the cross-correlations. The second row follows from (\ref{eq:liftgenEanal19}) and the third row follows from (\ref{eq:liftgenDanal19}). From (\ref{eq:liftconalt4}) we further have
\begin{eqnarray}\label{eq:liftconalt5}
\frac{d\psi_*(\calX,\beta,s,c_3,t)}{dt}
 & = & \frac{c_3(1-c_3)}{2} \sum_{i=1}^{l}\mE\sum_{p=1}^{l}\sum_{j=1}^{m} (\|\x^{(i)}\|_2\|\x^{(p)}\|_2-(\x^{(i)})^T\x^{(p)})\nonumber \\
  & & \times \frac{ \beta^2  A^{(i)}_*A^{(p)}_*(\u_j^{(i,1)}\sqrt{t}+\u_j^{(2)}\sqrt{1-t})(\u_j^{(p,1)}\sqrt{t}+\sqrt{1-t}\u_j^{(2)})}{Z^{2-c_3}_*B^{(i)}B^{(p)}} \nonumber \\
& & - \frac{c_3(1-c_3)}{2} \mE \lp\sum_{i=1}^{l}\sum_{p=1}^{l} (\|\x^{(i)}\|_2\|\x^{(p)}\|_2-(\x^{(i)})^T\x^{(p)})\frac{\beta^2 A^{(i)}_*  A^{(p)}_* }{Z^{2-c_3}_*}\rp.
\end{eqnarray}
(\ref{eq:liftconalt5}) then easily gives
\begin{eqnarray}\label{eq:liftconalt6}
\frac{d\psi_*(\calX,\beta,s,c_3,t)}{dt} & = &  -\frac{\beta^2 c_3(1-c_3)}{2} \sum_{i=1}^{l}\mE\sum_{p=1}^{l}\frac{ A^{(i)}_*A^{(p)}_*}{Z^{(2-c_3)}_*}
(\|\x^{(i)}\|_2\|\x^{(p)}\|_2-(\x^{(i)})^T\x^{(p)})\nonumber \\
 & &  \times \lp1-
\sum_{j=1}^{m}\frac{(\u_j^{(i,1)}\sqrt{t}+\u_j^{(2)}\sqrt{1-t})(\u_j^{(p,1)}\sqrt{t}+\sqrt{1-t}\u_j^{(2)})}{B^{(i)}B^{(p)}}\rp.
\end{eqnarray}
Finally we have
\begin{eqnarray}\label{eq:liftconalt7}
\frac{d\psi_*(\calX,\beta,s,c_3,t)}{dt} & = &  -\frac{\beta^2 c_3(1-c_3)}{2} \sum_{i=1}^{l}\mE\sum_{p=1}^{l} \frac{ A^{(i)}_*A^{(p)}_*}{Z^{(2-c_3)}_*} (\|\x^{(i)}\|_2\|\x^{(p)}\|_2-(\x^{(i)})^T\x^{(p)}) \nonumber \\
& & \times \lp1-
\frac{(\u^{(i,1)}\sqrt{t}+\u^{(2)}\sqrt{1-t})^T(\u^{(p,1)}\sqrt{t}+\sqrt{1-t}\u^{(2)})}{B^{(i)}B^{(p)}}\rp.\nonumber\\
\end{eqnarray}
We summarize the above results in the following theorem.
\begin{theorem}
\label{thm:thm2}
  Assume the setup of Theorem \ref{thm:thm1}. We then have
\begin{eqnarray}\label{eq:liftco1eq1}
\psi_*(\calX,\beta,s,c_3,t)= \psi_*(\calX,\beta,s,c_3,0)+\int_{0}^{t}\frac{d\psi_*(\calX,\beta,s,c_3,t)}{dt}dt,
\end{eqnarray}
where $\frac{d\psi_*(\calX,\beta,s,c_3,t)}{dt}$ is given by (\ref{eq:liftconalt7}).
\end{theorem}
\begin{proof}
  Follows trivially from the above discussion.
\end{proof}

\begin{corollary}\label{cor:liftcor1}
  Assume the setup of Theorem \ref{thm:thm2}.

  1) If $0< c_3< 1$ then $\frac{d\psi_*(\calX,\beta,s,c_3,t)}{dt}<0$ and $\psi_*(\calX,\beta,s,c_3,t)$ is decreasing in $t$ and the following comparison principle holds
\begin{eqnarray}\label{eq:liftco2aeq1}
\lim_{\beta\rightarrow\infty}\psi_*(\calX,\beta,s,c_3,0) \geq \lim_{\beta\rightarrow\infty} \psi_*(\calX,\beta,s,c_3,t)\geq \lim_{\beta\rightarrow\infty}\psi_*(\calX,\beta,s,c_3,1).
\end{eqnarray}

  2) If $c_3> 1$ or $c_3< 0$ then $\frac{d\psi_*(\calX,\beta,s,c_3,t)}{dt}>0$ and $\psi_*(\calX,\beta,s,c_3,t)$ is increasing in $t$ and the following comparison principle holds
\begin{eqnarray}\label{eq:liftco2aeq2}
\lim_{\beta\rightarrow\infty}\psi_*(\calX,\beta,s,c_3,0) \leq \lim_{\beta\rightarrow\infty} \psi_*(\calX,\beta,s,c_3,t)\leq \lim_{\beta\rightarrow\infty}\psi_*(\calX,\beta,s,c_3,1).
\end{eqnarray}
\end{corollary}
\begin{proof}
  Follows trivially by the above arguments.
\end{proof}

\subsection{Simulations}
\label{sec:liftedsim}

Below we present a few numerical results that highlight the precision one can achieve with the above theory. To facilitate the following we again parallel the presentation of the results as much as possible to the ones already presented earlier. To that end, we again chose $m=5$, $n=5$, $l=10$, and selected set $\calX=\calX^{+}$. The derivatives $\frac{d\psi_*(\calX,\beta,s,c_3,t)}{dt}$ are again simulated according to (\ref{eq:liftanal11}) and (\ref{eq:liftconalt7}) and $\psi_*(\calX,\beta,s,c_3,t)$ was computed according to (\ref{eq:liftco1eq1}). As earlier, we also simulated $\psi_*(\calX,\beta,s,c_3,t)$ directly through (\ref{eq:liftanal8}) and averaged all quantities over a set of $5e4$ experiments. Also, all simulations presented in this subsection were done with $\beta=3$ and $c_3=.1$. As earlier, we simulated two different scenarios with all other parameters being the same: 1) $s=1$ and 2) $s=-1$.

\textbf{\underline{\emph{1) $s=1$ -- numerical results}}}

The results that we obtained for $s=1$ are presented in Figure \ref{fig:liftedsplus1xnorm1psi} and Table \ref{tab:liftedsplus1xnorm1psi}. As usual, Figure \ref{fig:liftedsplus1xnorm1psi} shows the entire range for $t$ (i.e. $t\in(0,1)$) whereas Table \ref{tab:liftedsplus1xnorm1psi} focuses on several concrete values of $t$. In Table \ref{tab:liftedsplus1xnorm1psi}, the values for $\psi_*(\calX,\beta,s,c_3,t)$ are given in two forms: 1) the value itself and 2) the adjusted value $\lp\frac{1}{\beta c_3}\log\lp \psi_*(\calX,\beta,s,c_3,t)\rp-\frac{\beta c_3}{2}\rp/\sqrt{n}$ (the adjusted value is a way how one can think about connecting $\psi_*(\calX,\beta,s,c_3,t)$ and $\psi(\calX,\beta,s,t)$; we will discuss this in a greater detail below and whenever we deem as needed, we will present these values as well). As both, Figure \ref{fig:liftedsplus1xnorm1psi} and Table \ref{tab:liftedsplus1xnorm1psi}, show, the agreement between all presented results is again overwhelming.

\begin{figure}[htb]
\centering
\centerline{\epsfig{figure=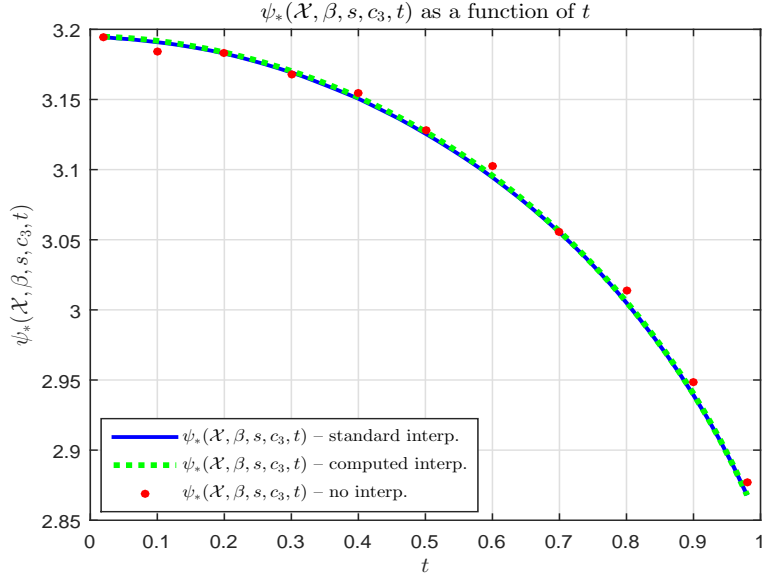,width=11.5cm,height=8cm}}
\caption{$\psi_*(\calX,\beta,s,c_3,t)$ as a function of $t$; $m=5$, $n=5$, $l=10$, $\calX=\calX^{+}$, $\beta=3$, $s=1$, $c_3=.1$}
\label{fig:liftedsplus1xnorm1psi}
\end{figure}
\begin{table}[h]
\caption{Simulated results --- $m=5$, $n=5$, $l=10$, $\calX=\calX^{+}$, $\beta=3$, $s=1$, $c_3=.1$}\vspace{.1in}
\hspace{-0in}\centering
\begin{tabular}{||c||c|c|c|c|c||}\hline\hline
$ t$  &  $\frac{d\psi_*}{dt}$; (\ref{eq:liftanal11}) & $\frac{d\psi_*}{dt}$;  (\ref{eq:liftconalt7}) & $\psi_*$;  (\ref{eq:liftanal11}) and (\ref{eq:liftco1eq1}) & $\psi_*$; (\ref{eq:liftconalt7}) and (\ref{eq:liftco1eq1}) & $\psi_*$;  (\ref{eq:liftanal8})\\  \hline\hline
$ 0.1000 $ & $ -0.0549 $ & $ -0.0525 $ & $\bl{\mathbf{ 3.1908 / 1.6626 }}$ & $\bl{\mathbf{ 3.1918 / 1.6630 }}$ & $\mathbf{ 3.1846 / 1.6597 }$
\\ \hline$ 0.2000 $ & $ -0.1054 $ & $ -0.1022 $ & $\bl{\mathbf{ 3.1827 / 1.6588 }}$ & $\bl{\mathbf{ 3.1836 / 1.6592 }}$ & $\mathbf{ 3.1832 / 1.6590 }$
\\ \hline$ 0.3000 $ & $ -0.1523 $ & $ -0.1541 $ & $\bl{\mathbf{ 3.1693 / 1.6525 }}$ & $\bl{\mathbf{ 3.1703 / 1.6529 }}$ & $\mathbf{ 3.1674 / 1.6516 }$
\\ \hline$ 0.4000 $ & $ -0.1999 $ & $ -0.2093 $ & $\bl{\mathbf{ 3.1506 / 1.6436 }}$ & $\bl{\mathbf{ 3.1516 / 1.6441 }}$ & $\mathbf{ 3.1546 / 1.6455 }$
\\ \hline$ 0.5000 $ & $ -0.2761 $ & $ -0.2712 $ & $\bl{\mathbf{ 3.1256 / 1.6318 }}$ & $\bl{\mathbf{ 3.1270 / 1.6325 }}$ & $\mathbf{ 3.1279 / 1.6329 }$
\\ \hline$ 0.6000 $ & $ -0.3448 $ & $ -0.3408 $ & $\bl{\mathbf{ 3.0944 / 1.6168 }}$ & $\bl{\mathbf{ 3.0956 / 1.6174 }}$ & $\mathbf{ 3.1024 / 1.6207 }$
\\ \hline$ 0.7000 $ & $ -0.4328 $ & $ -0.4311 $ & $\bl{\mathbf{ 3.0551 / 1.5978 }}$ & $\bl{\mathbf{ 3.0560 / 1.5982 }}$ & $\mathbf{ 3.0555 / 1.5980 }$
\\ \hline$ 0.8000 $ & $ -0.5582 $ & $ -0.5535 $ & $\bl{\mathbf{ 3.0051 / 1.5731 }}$ & $\bl{\mathbf{ 3.0057 / 1.5735 }}$ & $\mathbf{ 3.0138 / 1.5775 }$
\\ \hline$ 0.9000 $ & $ -0.7420 $ & $ -0.7393 $ & $\bl{\mathbf{ 2.9392 / 1.5401 }}$ & $\bl{\mathbf{ 2.9401 / 1.5405 }}$ & $\mathbf{ 2.9482 / 1.5447 }$
\\ \hline\hline
\end{tabular}
\label{tab:liftedsplus1xnorm1psi}
\end{table}

\textbf{\underline{\emph{2) $s=-1$ -- numerical results}}}

The results that we obtained for $s=-1$ are presented in Figure \ref{fig:liftedsmin1xnorm1psi} and Table \ref{tab:liftedsmin1xnorm1psi}. Figure \ref{fig:liftedsmin1xnorm1psi} again shows the entire range for $t$, whereas Table \ref{tab:liftedsmin1xnorm1psi} focuses on several concrete values of $t$. As was the case for $s=1$, here we again have that both, Figure \ref{fig:liftedsmin1xnorm1psi} and Table \ref{tab:liftedsmin1xnorm1psi}, show that the agreement between all presented results is rather overwhelming.

\begin{figure}[htb]
\centering
\centerline{\epsfig{figure=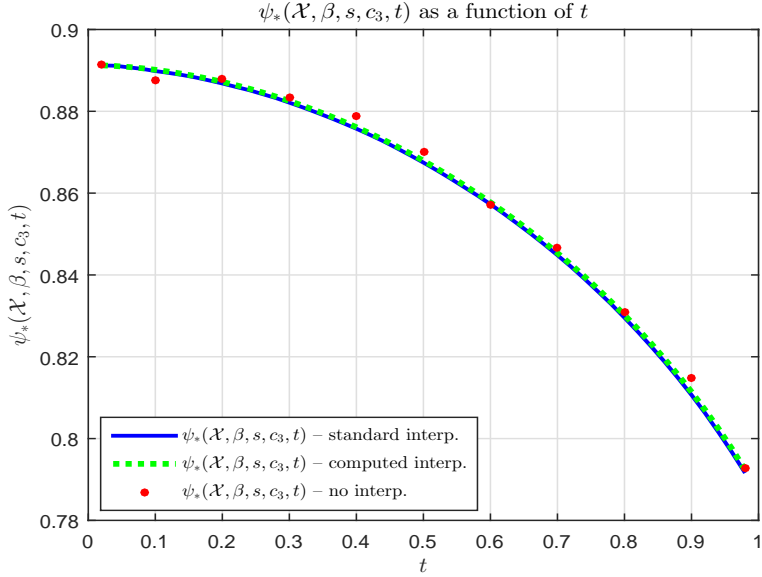,width=11.5cm,height=8cm}}
\caption{$\psi_*(\calX,\beta,s,c_3,t)$ as a function of $t$; $m=5$, $n=5$, $l=10$, $\calX=\calX^{+}$, $\beta=3$, $s=-1$, $c_3=.1$}
\label{fig:liftedsmin1xnorm1psi}
\end{figure}
\begin{table}[h]
\caption{Simulated results --- $m=5$, $n=5$, $l=10$, $\calX=\calX^{+}$, $\beta=3$, $s=-1$, $c_3=.1$}\vspace{.1in}
\hspace{-0in}\centering
\begin{tabular}{||c||c|c|c|c|c||}\hline\hline
$ t$  &  $\frac{d\psi_*}{dt}$; (\ref{eq:liftanal11}) & $\frac{d\psi_*}{dt}$;  (\ref{eq:liftconalt7}) & $\psi_*$;  (\ref{eq:liftanal11}) and (\ref{eq:liftco1eq1}) & $\psi_*$; (\ref{eq:liftconalt7}) and (\ref{eq:liftco1eq1}) & $\psi_*$;  (\ref{eq:liftanal8})\\  \hline\hline
$ 0.1000 $ & $ -0.0273 $ & $ -0.0190 $ & $\bl{\mathbf{ 0.8898 / -0.2411 }}$ & $\bl{\mathbf{ 0.8902 / -0.2405 }}$ & $\mathbf{ 0.8875 / -0.2450 }$
\\ \hline$ 0.2000 $ & $ -0.0378 $ & $ -0.0367 $ & $\bl{\mathbf{ 0.8868 / -0.2462 }}$ & $\bl{\mathbf{ 0.8872 / -0.2455 }}$ & $\mathbf{ 0.8878 / -0.2445 }$
\\ \hline$ 0.3000 $ & $ -0.0563 $ & $ -0.0541 $ & $\bl{\mathbf{ 0.8821 / -0.2541 }}$ & $\bl{\mathbf{ 0.8825 / -0.2534 }}$ & $\mathbf{ 0.8834 / -0.2519 }$
\\ \hline$ 0.4000 $ & $ -0.0693 $ & $ -0.0715 $ & $\bl{\mathbf{ 0.8757 / -0.2649 }}$ & $\bl{\mathbf{ 0.8761 / -0.2642 }}$ & $\mathbf{ 0.8788 / -0.2596 }$
\\ \hline$ 0.5000 $ & $ -0.0900 $ & $ -0.0899 $ & $\bl{\mathbf{ 0.8675 / -0.2790 }}$ & $\bl{\mathbf{ 0.8679 / -0.2783 }}$ & $\mathbf{ 0.8700 / -0.2747 }$
\\ \hline$ 0.6000 $ & $ -0.1095 $ & $ -0.1105 $ & $\bl{\mathbf{ 0.8573 / -0.2966 }}$ & $\bl{\mathbf{ 0.8577 / -0.2959 }}$ & $\mathbf{ 0.8572 / -0.2968 }$
\\ \hline$ 0.7000 $ & $ -0.1347 $ & $ -0.1337 $ & $\bl{\mathbf{ 0.8448 / -0.3185 }}$ & $\bl{\mathbf{ 0.8453 / -0.3175 }}$ & $\mathbf{ 0.8467 / -0.3152 }$
\\ \hline$ 0.8000 $ & $ -0.1680 $ & $ -0.1637 $ & $\bl{\mathbf{ 0.8295 / -0.3457 }}$ & $\bl{\mathbf{ 0.8302 / -0.3445 }}$ & $\mathbf{ 0.8308 / -0.3434 }$
\\ \hline$ 0.9000 $ & $ -0.2064 $ & $ -0.2066 $ & $\bl{\mathbf{ 0.8107 / -0.3800 }}$ & $\bl{\mathbf{ 0.8114 / -0.3786 }}$ & $\mathbf{ 0.8146 / -0.3727 }$
\\ \hline \hline
\end{tabular}
\label{tab:liftedsmin1xnorm1psi}
\end{table}

\subsection{$\beta\rightarrow \infty$}
\label{sec:betainf}

There are many interesting consequences of the above lifting principle. Here we select a particular one that connects to some of the comparison principles that we successively utilized in e.g. \cite{StojnicLiftStrSec13,StojnicMoreSophHopBnds10,StojnicRicBnds13}. Let us assume that $\beta$ is large, say $\beta\rightarrow\infty$ and that we have the following scaling $c_3\leftarrow \frac{c^{(s)}_3}{\beta}$, where $c^{(s)}_3$ is a finite positive real number. Then we have $c_3(1-c_3)\geq 0$ and since $\|\u^{(i,1)}\sqrt{t}+\u^{(2)}\sqrt{1-t}\|_2=B^{(i)}$ we have $\frac{\psi_*(\calX,\beta,s,c_3,t)}{dt}\leq 0$ and function $\psi_*(\calX,\beta,s,c_3,t)$ is decreasing in $t$. In other words, we have the following corollary of Theorem \ref{thm:thm2}.
\begin{corollary}\label{cor:liftcor2}
  Assume the setup of Theorem \ref{thm:thm2}. Let $c_3\leftarrow \frac{c^{(s)}_3}{\beta}$, where $c^{(s)}_3$ is a finite positive real number. Then $\psi_*(\calX,\beta,s,c_3,t)$ is decreasing in $t$ and the following comparison principle holds
\begin{eqnarray}\label{eq:liftco2eq2}
\lim_{\beta\rightarrow\infty}\psi_*(\calX,\beta,s,\frac{c^{(s)}_3}{\beta},0) \geq \lim_{\beta\rightarrow\infty} \psi_*(\calX,\beta,s,\frac{c^{(s)}_3}{\beta},t)\geq \lim_{\beta\rightarrow\infty}\psi_*(\calX,\beta,s,\frac{c^{(s)}_3}{\beta},1).
\end{eqnarray}
\end{corollary}
\begin{proof}
  Follows trivially by the above arguments.
\end{proof}
Now, it is often of particular interest to study the following limiting behavior of $\xi_*(\calX,\beta,s,\frac{c^{(s)}_3}{\beta})$, i.e.
\begin{eqnarray}\label{eq:liftbetainf1}
 \log\lim_{\beta\rightarrow\infty} \xi_*(\calX,\beta,s,\frac{c^{(s)}_3}{\beta})
& = & \log\lim_{\beta\rightarrow\infty} \mE_{G,u^{(4)}}\lp \sum_{i=1}^{l}e^{\beta\lp s\|
 G\x^{(i)}\|_2+\|\x^{(i)}\|_2u^{(4)}\rp} \rp^{\frac{c^{(s)}_3}{\beta}}\nonumber \\
 & = &  \log \mE_{G,u^{(4)}}\lp e^{c^{(s)}_3\lp \max_{\x^{(i)}\in\calX}\lp s\|
 G\x^{(i)}\|_2\rp+\|\x^{(i)}\|_2u^{(4)}\rp} \rp.\nonumber \\
\end{eqnarray}

\subsubsection{$s=1$ -- a lifted Slepian's comparison principle}
\label{sec:liftbetainfsplus1}

In particular when $s=1$ we have
\begin{eqnarray}\label{eq:liftbetainfsplus1}
 \log\lim_{\beta\rightarrow\infty} \xi_*(\calX,\beta,1,\frac{c^{(s)}_3}{\beta})
 & = &  \log \mE_{G,u^{(4)}}\lp e^{c^{(s)}_3\lp \max_{\x^{(i)}\in\calX}\lp \|
 G\x^{(i)}\|_2\rp+\|\x^{(i)}\|_2u^{(4)}\rp} \rp\nonumber \\
\end{eqnarray}
Recalling that $\xi_*(\calX,\beta,1,\frac{c^{(s)}_3}{\beta})=\psi_*(\calX,\beta,1,\frac{c^{(s)}_3}{\beta},1)$ and utilizing what we presented above one obtains
\begin{eqnarray}\label{eq:liftbetainfsplus2}
\log \mE_{G,u^{(4)}}\lp e^{c^{(s)}_3\lp \max_{\x^{(i)}\in\calX}\lp \|
 G\x^{(i)}\|_2\rp+\|\x^{(i)}\|_2u^{(4)}\rp} \rp & = &
 \log\lim_{\beta\rightarrow\infty} \psi_*(\calX,\beta,1,\frac{c^{(s)}_3}{\beta},1)\nonumber \\
& \leq &    \log\lim_{\beta\rightarrow\infty} \psi_*(\calX,\beta,1,\frac{c^{(s)}_3}{\beta},t) \nonumber \\
& \leq &    \log\lim_{\beta\rightarrow\infty} \psi_*(\calX,\beta,1,\frac{c^{(s)}_3}{\beta},0) \nonumber \\
& = &
\log \mE_{\u^{(2)},\h}\lp e^{c^{(s)}_3\lp \max_{\x^{(i)}\in\calX}\lp \|\x^{(i)}\|_2\|\u^{(2)}\|_2+\h^T\x^{(i)}\rp\rp} \rp.\nonumber \\
\end{eqnarray}
Connecting beginning and end in (\ref{eq:liftbetainfsplus2}) one obtains exactly the same as the comparison principle that we utilized in \cite{StojnicMoreSophHopBnds10}. Namely, in \cite{StojnicMoreSophHopBnds10}, we relied on a Gordon's upgrade of the Slepian's principle to obtain
\begin{eqnarray}\label{eq:liftbetainfsplus3}
\log \mE_{G,u^{(4)}} e^{c^{(s)}_3\max_{\x^{(i)}\in \calX} \lp \|
 G\x^{(i)}\|_2+\|\x^{(i)}\|_2u^{(4)}\rp} & = &
 \log \mE_{G,u^{(4)}} e^{c^{(s)}_3\max_{\x^{(i)}\in \calX,\y\in S^{m-1}} \lp \y^T G\x^{(i)}+\|\x^{(i)}\|_2u^{(4)}\rp}\nonumber \\
 & \leq & \log \mE_{\u^{(2)},\h} e^{c^{(s)}_3\max_{\x^{(i)}\in \calX,\y\in S^{m-1}} \lp \|\x^{(i)}\|_2\y^T \u^{(2)}+\h^T\x^{(i)}\rp}\nonumber \\
& = & \log  \mE_{\u^{(2)},\h} e^{c^{(s)}_3\max_{\x^{(i)}\in \calX} \lp \|\x^{(i)}\|_2\|\u^{(2)}\|_2+\h^T\x^{(i)}\rp}.
\end{eqnarray}
In a way one can think of (\ref{eq:liftbetainfsplus1}) and (\ref{eq:liftbetainfsplus2}) as being a lifted Slepian comparison principle. This lifting procedure often offers a solid improvement over the standard Slepian's comparison (see, e.g. \cite{StojnicMoreSophHopBnds10}). In fact, it is often the only known tool that can outmatch strategies based on the original Slepian's principle. As mentioned above, in \cite{StojnicMoreSophHopBnds10}, we created powerful comparison results relying on (\ref{eq:liftbetainfsplus2}) which we obtained relying on a Gordon's upgrade \cite{Gordon85} of the Slepian's max principle. As the above shows, this form is only a special case of a much stronger concept introduced in Theorem \ref{thm:thm2}.

\textbf{\underline{\emph{Numerical results}}}

In Figure \ref{fig:liftedbetainfsplus1xnorm1psi} and Table \ref{tab:liftedbetainfsplus1xnorm1psi} a few simulated results to complement the above theoretical results related to the lifting procedure. All parameters are again the same as earlier, with the exception that now $\beta=10$, which again in a way emulates $\beta\rightarrow\infty$, and $c_3=.03$. Both, Figure \ref{fig:liftedbetainfsplus1xnorm1psi} and Table \ref{tab:liftedbetainfsplus1xnorm1psi}, again demonstrate a solid agreement between all the presented results with $\beta=10$ being a solid approximation of $\beta\rightarrow\infty$ (to in a way ensure validity of comparison, the values for $\lim_{\beta\rightarrow\infty}\psi_*$ were obtained with $c_3^{(s)}=c_3\beta$; here $\beta=10$ and $c_3=.03$). In the right part of the figure, we also added how the obtained results compare to the same scenario with no lifting (a natural comparison goes through the adjusted $\psi_*(\cdot)$ discussed earlier). One can observe a bit of the flattening effect which is a consequence of the lifting procedure and in turn tightens the corresponding comparisons from Section \ref{sec:gencon}. This effect will be way more pronounced in the example that follows below.


\begin{figure}[htb]
\begin{minipage}[b]{.5\linewidth}
\centering
\centerline{\epsfig{figure=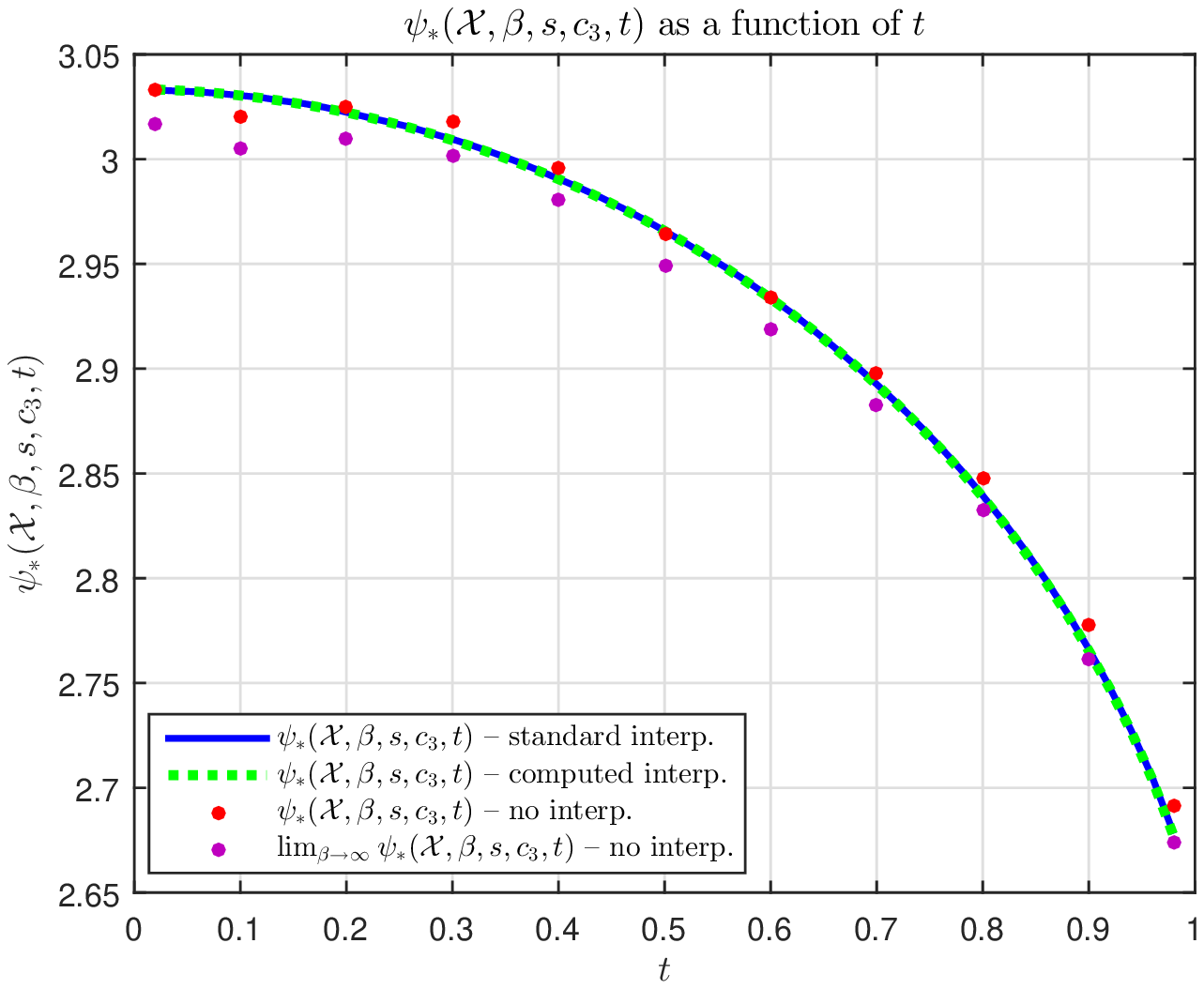,width=9cm,height=7cm}}
\end{minipage}
\begin{minipage}[b]{.5\linewidth}
\centering
\centerline{\epsfig{figure=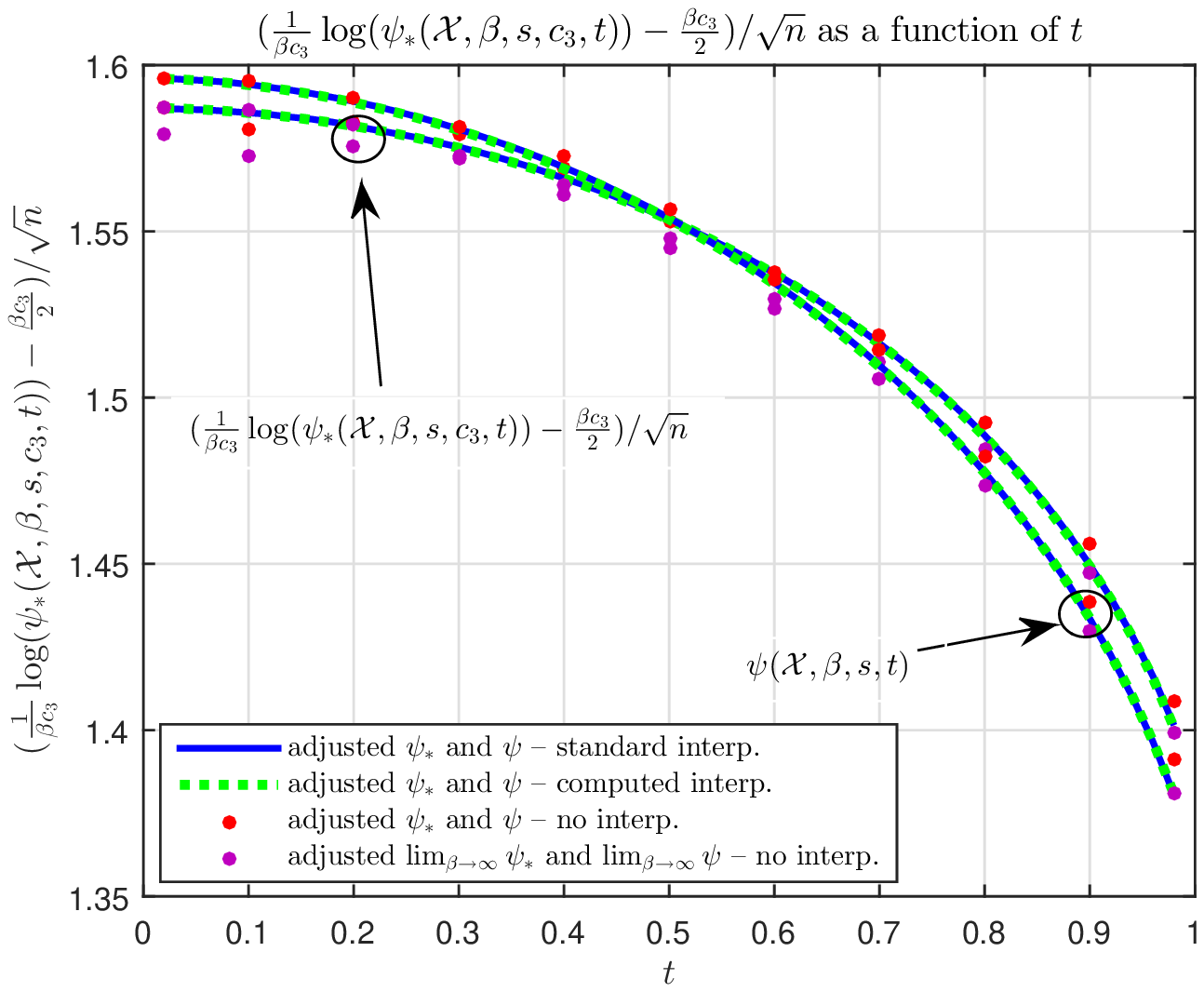,width=9cm,height=7cm}}
\end{minipage}
\caption{Left -- $\psi_*(\calX,\beta,s,c_3,t)$ as a function of $t$; $m=5$, $n=5$, $l=10$, $\calX=\calX^{+}$, $\beta=10$, $s=1$, $c_3=.03$; right -- comparison between adjusted $\psi_*(\calX,\beta,s,c_3,t)$ and $\psi(\calX,\beta,s,t)$ for $\beta=10$ (lifting versus no-lifting)}
\label{fig:liftedbetainfsplus1xnorm1psi}
\end{figure}

\begin{table}[h]
\caption{Simulated results --- $m=5$, $n=5$, $l=10$, $\calX=\calX^{+}$, $\beta=10$, $s=1$, $c_3=.03$}\vspace{.1in}
\hspace{-0in}\centering
{\small
\begin{tabular}{||c||c|c|c|c|c|c||}
\hline\hline
$ t$  &  $\frac{d\psi_*}{dt}$; (\ref{eq:liftanal11}) & $\frac{d\psi_*}{dt}$;  (\ref{eq:liftconalt7}) & $\psi_*$;  (\ref{eq:liftanal11}) and (\ref{eq:liftco1eq1}) & $\psi$; (\ref{eq:liftconalt7}) and (\ref{eq:liftco1eq1}) & $\psi_*$;  (\ref{eq:liftanal8})& $\lim_{\beta\rightarrow\infty}\psi_*$;  (\ref{eq:liftanal8})\\  \hline\hline
$ 0.1 $ & $ -0.0479 $ & $ -0.0518 $ & $\bl{\mathbf{ 3.0304 / 1.5857 }}$ & $\bl{\mathbf{ 3.0304 / 1.5857 }}$ & $\mathbf{ 3.0204 / 1.5808 }$& $\prp{\mathbf{ 3.0047 / 1.5730 }}$\\ \hline
$ 0.2 $ & $ -0.1145 $ & $ -0.0993 $ & $\bl{\mathbf{ 3.0224 / 1.5817 }}$ & $\bl{\mathbf{ 3.0223 / 1.5817 }}$ & $\mathbf{ 3.0251 / 1.5831 }$& $\prp{\mathbf{ 3.0093 / 1.5753 }}$\\ \hline
$ 0.3 $ & $ -0.1411 $ & $ -0.1547 $ & $\bl{\mathbf{ 3.0095 / 1.5754 }}$ & $\bl{\mathbf{ 3.0091 / 1.5751 }}$ & $\mathbf{ 3.0177 / 1.5794 }$& $\prp{\mathbf{ 3.0019 / 1.5716 }}$\\ \hline
$ 0.4 $ & $ -0.2146 $ & $ -0.2079 $ & $\bl{\mathbf{ 2.9906 / 1.5660 }}$ & $\bl{\mathbf{ 2.9905 / 1.5659 }}$ & $\mathbf{ 2.9962 / 1.5687 }$& $\prp{\mathbf{ 2.9808 / 1.5611 }}$\\ \hline
$ 0.5 $ & $ -0.2839 $ & $ -0.2748 $ & $\bl{\mathbf{ 2.9658 / 1.5535 }}$ & $\bl{\mathbf{ 2.9658 / 1.5535 }}$ & $\mathbf{ 2.9644 / 1.5528 }$& $\prp{\mathbf{ 2.9491 / 1.5451 }}$\\ \hline
$ 0.6 $ & $ -0.3462 $ & $ -0.3522 $ & $\bl{\mathbf{ 2.9339 / 1.5374 }}$ & $\bl{\mathbf{ 2.9335 / 1.5372 }}$ & $\mathbf{ 2.9345 / 1.5377 }$& $\prp{\mathbf{ 2.9192 / 1.5299 }}$\\ \hline
$ 0.7 $ & $ -0.4512 $ & $ -0.4511 $ & $\bl{\mathbf{ 2.8926 / 1.5163 }}$ & $\bl{\mathbf{ 2.8924 / 1.5162 }}$ & $\mathbf{ 2.8976 / 1.5188 }$& $\prp{\mathbf{ 2.8824 / 1.5110 }}$\\ \hline
$ 0.8 $ & $ -0.5984 $ & $ -0.5980 $ & $\bl{\mathbf{ 2.8395 / 1.4886 }}$ & $\bl{\mathbf{ 2.8387 / 1.4882 }}$ & $\mathbf{ 2.8475 / 1.4928 }$& $\prp{\mathbf{ 2.8320 / 1.4847 }}$\\ \hline
$ 0.9 $ & $ -0.8504 $ & $ -0.8428 $ & $\bl{\mathbf{ 2.7661 / 1.4496 }}$ & $\bl{\mathbf{ 2.7654 / 1.4492 }}$ & $\mathbf{ 2.7775 / 1.4557 }$& $\prp{\mathbf{ 2.7615 / 1.4472 }}$
\\ \hline\hline
\end{tabular}
}
\label{tab:liftedbetainfsplus1xnorm1psi}
\end{table}

\subsubsection{$s=-1$ -- a lifted Gordon's comparison principle}
\label{sec:betainfsminus1}

When $s=-1$ we have
\begin{eqnarray}\label{eq:liftbetainfsmin1}
 \log\lim_{\beta\rightarrow\infty} \xi_*(\calX,\beta,-1,\frac{c^{(s)}_3}{\beta})
 & = &  \log \mE_{G,u^{(4)}}\lp e^{c^{(s)}_3\lp \max_{\x^{(i)}\in\calX}\lp -\|
 G\x^{(i)}\|_2\rp+\|\x^{(i)}\|_2u^{(4)}\rp} \rp\nonumber \\
\end{eqnarray}
Utilizing what we presented above while again keeping in mind that $\xi_*(\calX,\beta,-1,\frac{c^{(s)}_3}{\beta})=\psi_*(\calX,\beta,-1,\frac{c^{(s)}_3}{\beta},1)$ we obtain
\begin{eqnarray}\label{eq:liftbetainfsmin2}
\log \mE_{G,u^{(4)}}\lp e^{c^{(s)}_3\lp \max_{\x^{(i)}\in\calX}\lp -\|
 G\x^{(i)}\|_2\rp+\|\x^{(i)}\|_2u^{(4)}\rp} \rp & = &
 \log\lim_{\beta\rightarrow\infty} \psi_*(\calX,\beta,-1,\frac{c^{(s)}_3}{\beta},1)\nonumber \\
& \leq  &    \log\lim_{\beta\rightarrow\infty} \psi_*(\calX,\beta,-1,\frac{c^{(s)}_3}{\beta},t) \nonumber \\
& \leq  &    \log\lim_{\beta\rightarrow\infty} \psi_*(\calX,\beta,-1,\frac{c^{(s)}_3}{\beta},0) \nonumber \\
& = &
\log \mE_{\u^{(2)},\h}\lp e^{c^{(s)}_3\lp \max_{\x^{(i)}\in\calX}\lp -\|\x^{(i)}\|_2\|\u^{(2)}\|_2+\h^T\x^{(i)}\rp\rp} \rp.\nonumber \\
\end{eqnarray}
Connecting beginning and end in (\ref{eq:liftbetainfsmin2}) we again observe exactly the same inequality as in the comparison principle we utilized in \cite{StojnicMoreSophHopBnds10} (as well as in e.g. \cite{StojnicLiftStrSec13,StojnicRicBnds13}). Namely, in \cite{StojnicMoreSophHopBnds10}, we relied on a Gordon's minmax principle to obtain
\begin{eqnarray}\label{eq:liftbetainfsmin2}
\log \mE_{G,u^{(4)}} e^{c^{(s)}_3\max_{\x^{(i)}\in \calX} \lp -\|
 G\x^{(i)}\|_2+\|\x^{(i)}\|_2u^{(4)}\rp} & = &
 \log \mE_{G,u^{(4)}} e^{c^{(s)}_3\max_{\x^{(i)}\in \calX}\min_{\y\in S^{m-1}} \lp \y^T G\x^{(i)}+\|\x^{(i)}\|_2u^{(4)}\rp}\nonumber \\
 & \leq & \log \mE_{\u^{(2)},\h} e^{c^{(s)}_3\max_{\x^{(i)}\in \calX}\min_{\y\in S^{m-1}} \lp \|\x^{(i)}\|_2\y^T \u^{(2)}+\h^T\x^{(i)}\rp}\nonumber \\
& = & \log  \mE_{\u^{(2)},\h} e^{c^{(s)}_3\max_{\x^{(i)}\in \calX} \lp -\|\x^{(i)}\|_2\|\u^{(2)}\|_2+\h^T\x^{(i)}\rp}.
\end{eqnarray}
Similarly to what we had above, one can think of (\ref{eq:liftbetainfsmin1}) and (\ref{eq:liftbetainfsmin2}) as being a lifted Gordon's minmax comparison principle. The improvement that this lifting idea brings in studying hard optimization problems is often also rather massive (see, e.g. \cite{StojnicMoreSophHopBnds10,StojnicLiftStrSec13,StojnicRicBnds13}). In \cite{StojnicMoreSophHopBnds10}, we created powerful comparison results relying on (\ref{eq:liftbetainfsmin2}) which we obtained relying on a Gordon's minmax principle \cite{Gordon85}. Clearly, this form is only a special case of a much stronger concept introduced in Theorem \ref{thm:thm2}.

When $\x^{(i)}=1,1\leq i\leq l$, we have a particularly elegant consequence of the above even for any $\beta$ and any sign $s$
\begin{eqnarray}\label{eq:liftbetainfsmin3}
\frac{(c_3^{(s)})^2}{2}+c_3^{(s)}\mE_{G} \max_{\x^{(i)}\in \calX} \lp s\|
 G\x^{(i)}\|_2\rp & = & \frac{(c_3^{(s)})^2}{2}+ \mE_{G} \log e^{c^{(s)}_3\max_{\x^{(i)}\in \calX} \lp s\|
 G\x^{(i)}\|_2\rp} \nonumber \\
&\leq & \log \mE_{G,u^{(4)}} e^{c^{(s)}_3\max_{\x^{(i)}\in \calX} \lp s\|
 G\x^{(i)}\|_2+u^{(4)}\rp} \nonumber \\
&\leq & \log \mE_{G,u^{(4)}} \lp \sum_{i=1}^{l}e^{\beta  \lp s\|
 G\x^{(i)}\|_2+u^{(4)}\rp}\rp^{c_3}  \nonumber \\
&= & \log \mE_{G,u^{(4)}} \psi_*(\calX,\beta,s,c_3,1)  \nonumber \\
&\leq & \log \mE_{G,u^{(4)}} \psi_*(\calX,\beta,s,c_3,t)  \nonumber \\
&\leq & \log \mE_{G,u^{(4)}} \psi_*(\calX,\beta,s,c_3,0)  \nonumber \\
& = & \log \mE_{u^{(2)},\h} \lp \sum_{i=1}^{l}e^{\beta  \lp s\|
 \u^{(2)}\|_2+\h^T\x^{(i)}\rp}\rp^{c_3}.
\end{eqnarray}
In other words, we have
\begin{eqnarray}\label{eq:liftbetainfsmin3a}
\mE_{G} \max_{\x^{(i)}\in \calX} \lp s\|
 G\x^{(i)}\|_2\rp & \leq & \frac{1}{c_3^{(s)}} \log \mE_{G,u^{(4)},u^{(2)},\h} \lp \sum_{i=1}^{l}e^{\beta  \lp s\|
 G\x^{(i)}+\sqrt{1-t}\u^{(2)}\|_2+\sqrt{t}u^{(4)}+\h^T\x^{(i)}\rp}\rp^{c_3}-\frac{c_3^{(s)}}{2} \nonumber \\
& = & \frac{1}{\beta c_3} \log\mE_{G,u^{(4)},u^{(2)},\h} \lp \sum_{i=1}^{l}e^{\beta  \lp s\|
 G\x^{(i)}+\sqrt{1-t}\u^{(2)}\|_2+\sqrt{t}u^{(4)}+\h^T\x^{(i)}\rp}\rp^{c_3}-\frac{\beta c_3}{2}.\nonumber \\
\end{eqnarray}
For $t=1$, (\ref{eq:liftbetainfsmin3a}) becomes
\begin{eqnarray}\label{eq:liftbetainfsmin4}
\mE_{G} \max_{\x^{(i)}\in \calX} \lp s\|
 G\x^{(i)}\|_2\rp & \leq & \frac{1}{c_3^{(s)}} \log \mE_{u^{(2)},\h} \lp \sum_{i=1}^{l}e^{\beta  \lp s\|
 \u^{(2)}\|_2+\h^T\x^{(i)}\rp}\rp^{c_3}-\frac{c_3^{(s)}}{2} \nonumber \\
& = & \frac{1}{\beta c_3} \log \mE_{u^{(2)},\h} \lp \sum_{i=1}^{l}e^{\beta  \lp s\|
 \u^{(2)}\|_2+\h^T\x^{(i)}\rp}\rp^{c_3}-\frac{\beta c_3}{2}.
\end{eqnarray}
Of course, finally, for $\beta\rightarrow\infty$ (and $c_3=\frac{c_3^{(2)}}{\beta}$) we have
\begin{eqnarray}\label{eq:liftbetainfsmin5}
\mE_{G} \max_{\x^{(i)}\in \calX} \lp s\|
 G\x^{(i)}\|_2\rp  \leq  \frac{1}{c_3^{(s)}} \log \mE_{u^{(2)},\h} \lp e^{c_3^{(s)}\max_{\x^{(i)}\in\calX} \lp s\|
 \u^{(2)}\|_2+\h^T\x^{(i)}\rp}\rp-\frac{c_3^{(s)}}{2},
\end{eqnarray}
basically one of the key components of the mechanisms we introduced and utilized in \cite{StojnicMoreSophHopBnds10,StojnicLiftStrSec13,StojnicRicBnds13}.

\textbf{\underline{\emph{Numerical results}}}

In Figure \ref{fig:liftedbetainfsmin1xnorm1psi} and Table \ref{tab:liftedbetainfsmin1xnorm1psi} simulated results are shown. All parameters are again the same as earlier, with the exception that now $\beta=10$ which again in a way emulates $\beta\rightarrow\infty$ and $c_3=.1$.
Both, Figure \ref{fig:liftedbetainfsmin1xnorm1psi} and Table \ref{tab:liftedbetainfsmin1xnorm1psi}, again demonstrate a solid agreement between all the presented results with $\beta=10$ being a pretty good approximation of $\beta\rightarrow\infty$. In particular, the right part of the figure shows a more pronounced flattening effect as a consequence of the lifting procedure. This of course tightens the corresponding comparisons from Section \ref{sec:gencon}.


\begin{figure}[htb]
\begin{minipage}[b]{.5\linewidth}
\centering
\centerline{\epsfig{figure=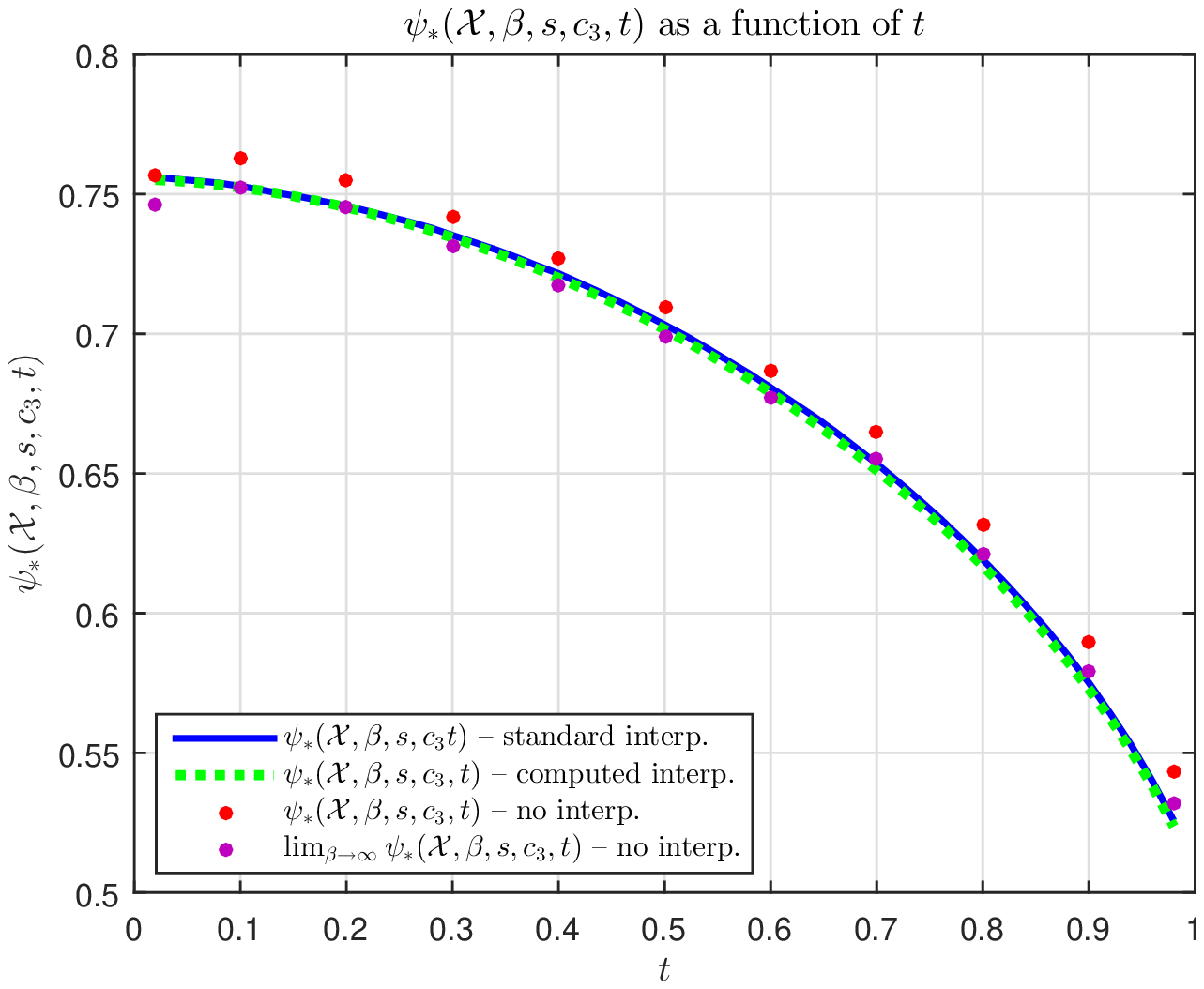,width=9cm,height=7cm}}
\end{minipage}
\begin{minipage}[b]{.5\linewidth}
\centering
\centerline{\epsfig{figure=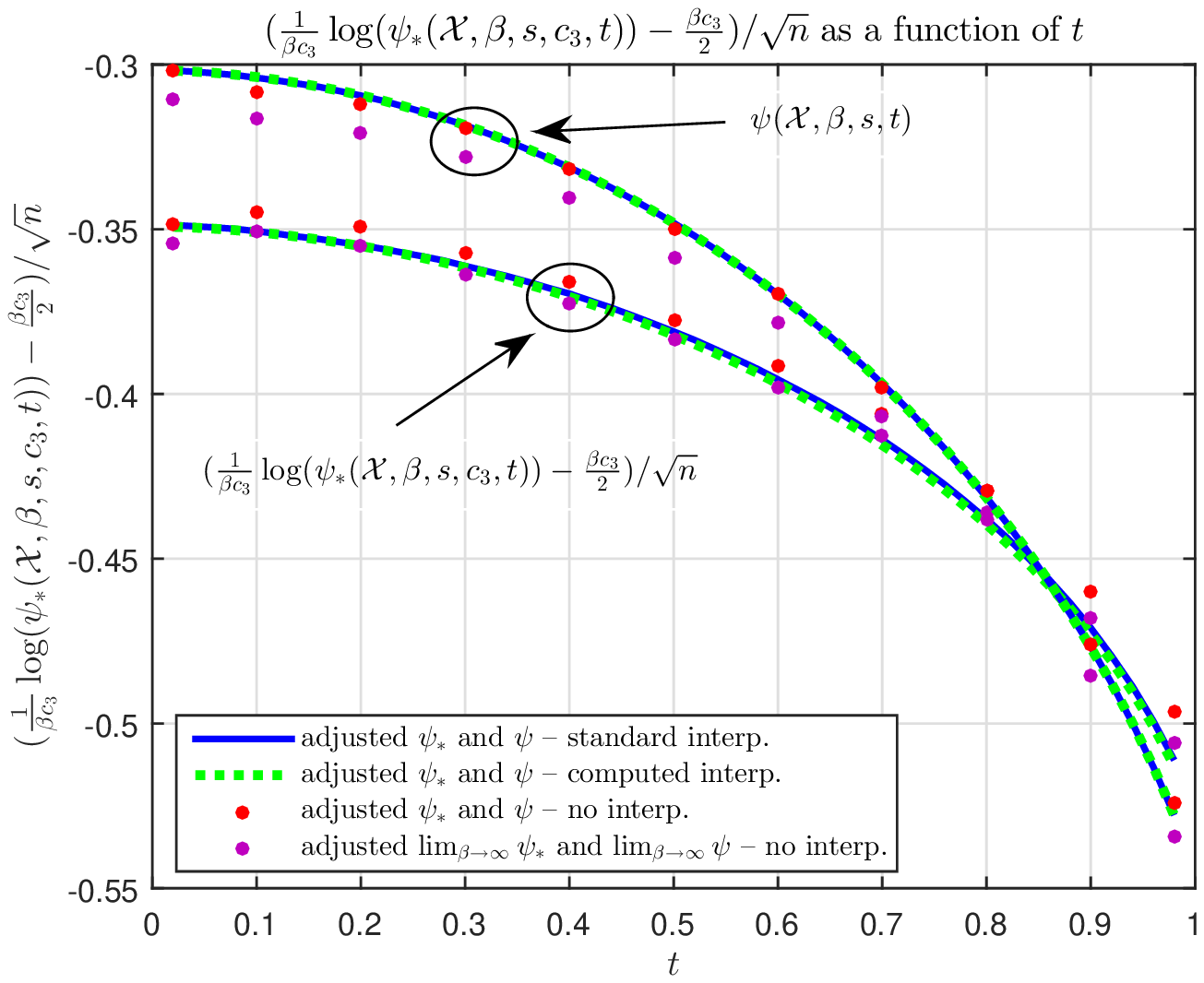,width=9cm,height=7cm}}
\end{minipage}
\caption{Left -- $\psi_*(\calX,\beta,s,c_3,t)$ as a function of $t$; $m=5$, $n=5$, $l=10$, $\calX=\calX^{+}$, $\beta=10$, $s=-1$, $c_3=.1$; right -- comparison between adjusted $\psi_*(\calX,\beta,s,c_3,t)$ and $\psi(\calX,\beta,s,t)$ for $\beta=10$ (lifting versus no-lifting)}
\label{fig:liftedbetainfsmin1xnorm1psi}
\end{figure}

{\footnotesize
\begin{table}[h]
\caption{Simulated results --- $m=5$, $n=5$, $l=10$, $\calX=\calX^{+}$, $\beta=10$, $s=-1$, $c_3=.1$}\vspace{.1in}
\hspace{-0in}\centering
{\small
\begin{tabular}{||c||c|c|c|c|c|c||}\hline\hline
$ t$  &  $\frac{d\psi_*}{dt}$; (\ref{eq:liftanal11}) & $\frac{d\psi_*}{dt}$;  (\ref{eq:liftconalt7}) & $\psi_*$;  (\ref{eq:liftanal11}) and (\ref{eq:liftco1eq1}) & $\psi$; (\ref{eq:liftconalt7}) and (\ref{eq:liftco1eq1}) & $\psi_*$;  (\ref{eq:liftanal8})& $\lim_{\beta\rightarrow\infty}\psi_*$;  (\ref{eq:liftanal8})\\  \hline\hline
$ 0.1 $ & $ -0.0575 $ & $ -0.0480 $ & $\bl{\mathbf{ 0.7528 / -0.3506 }}$ & $\bl{\mathbf{ 0.7523 / -0.3509 }}$ & $\mathbf{ 0.7627 / -0.3448 }$& $\prp{\mathbf{ 0.7524 / -0.3508 }}$\\ \hline
$ 0.2 $ & $ -0.0746 $ & $ -0.0866 $ & $\bl{\mathbf{ 0.7457 / -0.3548 }}$ & $\bl{\mathbf{ 0.7453 / -0.3551 }}$ & $\mathbf{ 0.7551 / -0.3492 }$& $\prp{\mathbf{ 0.7449 / -0.3553 }}$\\ \hline
$ 0.3 $ & $ -0.1312 $ & $ -0.1244 $ & $\bl{\mathbf{ 0.7353 / -0.3611 }}$ & $\bl{\mathbf{ 0.7344 / -0.3616 }}$ & $\mathbf{ 0.7416 / -0.3573 }$& $\prp{\mathbf{ 0.7315 / -0.3634 }}$\\ \hline
$ 0.4 $ & $ -0.1393 $ & $ -0.1596 $ & $\bl{\mathbf{ 0.7216 / -0.3695 }}$ & $\bl{\mathbf{ 0.7200 / -0.3705 }}$ & $\mathbf{ 0.7270 / -0.3662 }$& $\prp{\mathbf{ 0.7171 / -0.3723 }}$\\ \hline
$ 0.5 $ & $ -0.1963 $ & $ -0.2010 $ & $\bl{\mathbf{ 0.7032 / -0.3810 }}$ & $\bl{\mathbf{ 0.7016 / -0.3821 }}$ & $\mathbf{ 0.7092 / -0.3773 }$& $\prp{\mathbf{ 0.6991 / -0.3837 }}$\\ \hline
$ 0.6 $ & $ -0.2478 $ & $ -0.2464 $ & $\bl{\mathbf{ 0.6808 / -0.3956 }}$ & $\bl{\mathbf{ 0.6789 / -0.3968 }}$ & $\mathbf{ 0.6867 / -0.3917 }$& $\prp{\mathbf{ 0.6767 / -0.3982 }}$\\ \hline
$ 0.7 $ & $ -0.2937 $ & $ -0.3015 $ & $\bl{\mathbf{ 0.6535 / -0.4139 }}$ & $\bl{\mathbf{ 0.6511 / -0.4155 }}$ & $\mathbf{ 0.6649 / -0.4061 }$& $\prp{\mathbf{ 0.6549 / -0.4129 }}$\\ \hline
$ 0.8 $ & $ -0.3713 $ & $ -0.3731 $ & $\bl{\mathbf{ 0.6192 / -0.4380 }}$ & $\bl{\mathbf{ 0.6170 / -0.4396 }}$ & $\mathbf{ 0.6316 / -0.4291 }$& $\prp{\mathbf{ 0.6215 / -0.4363 }}$\\ \hline
$ 0.9 $ & $ -0.4972 $ & $ -0.4954 $ & $\bl{\mathbf{ 0.5751 / -0.4710 }}$ & $\bl{\mathbf{ 0.5728 / -0.4728 }}$ & $\mathbf{ 0.5894 / -0.4600 }$& $\prp{\mathbf{ 0.5789 / -0.4680 }}$
\\ \hline\hline
\end{tabular}
}
\label{tab:liftedbetainfsmin1xnorm1psi}
\end{table}
}

\section{Conclusion}
\label{sec:conc}

In this paper we introduce a series of very powerful statistical comparison results. Two types of comparisons are presented: 1) the general one which relies on the basic analytical properties and 2) the lifted one that relies on a lifting procedure as a way to strengthen the general one. A substantial set of numerical experiments is also presented and a very strong agreement with the theoretical predictions is observed. A particular feature of the presented results is their generality. Although they don't treat extremes of random processes per se, they are strong enough so that they cover these scenarios as well (as we also demonstrated, one can quickly show that they contain as special cases many well known results from the theory of the extremes of random process, including the Slepian's maz and Gordon's minmax principles; in fact, much stronger statement is true, namely, the results that we presented here show that various minmax scenarios are basically just the max ones).

Another important feature of the presented results is that they contain as special cases some of the comparison concepts that we utilized as starting points in solving many hard problems in various areas of mathematics in recent years (in particular, for many of these problems it often happened that no other tool was available even to come close to the characterizations that we provided through the comparison principles).  Moreover, what we present here is derived pretty much while relying only on the axioms and Leibniz-Newton derivative calculus. As quite a few other results that we have created in recent years essentially use as the starting blocks the concepts presented here, that makes the theories that we developed for handling many hard mathematical problems pretty much fully self-contained and, modulo Leibniz-Newton, entirely based on our own work.

Creating a theory while utilizing not much more than just the axioms, establishes a very powerful self-sustainable tool. Its underlying concepts can then serve as major cornerstones for various other extensions, which include both, the core comparison ones and the application oriented ones. We will present many of them separately to ensure that in this introductory paper the clarity and simplicity of the exposition are preserved.

\begin{singlespace}
\bibliographystyle{plain}
\bibliography{gscompRefs}
\end{singlespace}

\end{document}